\documentclass[11pt,reqno]{amsart}

\usepackage{amsmath,amssymb,amsfonts}

\usepackage{graphicx}
\providecommand{\U}[1]{\protect\rule{.1in}{.1in}}
\newtheorem{theorem}{Theorem}
\theoremstyle{plain}

\newtheorem{corollary}{Corollary}

\newtheorem{definition}{Definition}

\newtheorem{lemma}{Lemma}

\newtheorem{proposition}{Proposition}
\newtheorem{remark}{Remark}

\numberwithin{equation}{section}
\numberwithin{theorem}{section}
\numberwithin{proposition}{section}
\numberwithin{remark}{section}
\numberwithin{definition}{section}
\numberwithin{lemma}{section}
\numberwithin{corollary}{section}
\numberwithin{example}{section}
\numberwithin{claim}{section}
\begin{document}
\title[Nonlocal Hele-Shaw-Cahn-Hilliard flow]{Derivation and analysis of a nonlocal Hele-Shaw-Cahn-Hilliard system for flow
in thin heterogeneous layers}

\author{Giuseppe Cardone}
\address{G. Cardone, Dipartimento di Matematica e Applicazioni "Renato Caccioppoli",
Universit\`{a} degli Studi di Napoli Frederico II, via Cintia, 80126 Napoli, Italy}
\email{giuseppe.cardone@unina.it}
\author{Willi J\"{a}ger}
\address{W. J\"{a}ger, Interdisciplinary Center for Scientific Computing (IWR),
University of Heidelberg, Im Neuenheimer Feld 205, 69120 Heidelberg, Germany}
\email{wjaeger@iwr.uni-heidelberg.de}
\author{Jean Louis Woukeng}
\address{J.L. Woukeng, Department of Mathematics and Computer Science, University of
Dschang, P.O. Box 67, Dschang, Cameroon}
\email{jeanlouis.woukeng@univ-dschang.org, jean.woukeng@uni-a.de}

\begin{abstract}
We derive, through the deterministic homogenization theory in thin domains, a
new model consisting of Hele-Shaw equation with memory coupled with the
convective Cahn-Hilliard equation. The obtained system, which models in
particular tumor growth, is then analyzed and we prove its well-posedness in
dimension 2. To achieve our goal, we develop and use the new concept of
sigma-convergence in thin heterogeneous media, and we prove some regularity
results for the upscaled model.

\end{abstract}
\maketitle

\section{Introduction and the main results\label{sec1}}

We develop a rigorous mathematical analysis for the study of a mixture of
fluids occurring in a thin layer. The problem addressed is related to the
study of a phase field model for the evolution of a mixture of two
incompressible immiscible fluids modeled by Stokes-Cahn-Hilliard equations
evolving in a highly heterogeneous thin layer whose heterogeneities are
discontinuous and present a greater flexibility in behaviour. This kind of
problems arise especially in the study of the depollution of soils, \cite{pol}
filtering, \cite{sol} blood flow and the flow of liquid-gases in the energetic
cell \cite{sole2}.

The Stokes-Cahn-Hilliard evolution system, which consists of the Stokes
equation for the fluid velocity suitably coupled with a convective
Cahn-Hilliard equation for the order parameter has for a long time been widely
used to describe the evolution of an incompressible mixture of two immiscible
fluids (see Ref. \cite{6', Feng2006, 28'} and references therein). In this work we
are concerned with the model stated as follows.

Let $\Omega$ be a bounded open domain in $\mathbb{R}^{d-1}$ ($d=2,3$) which is
assumed throughout to be (except where otherwise stated) of class
$\mathcal{C}^{4}$. For $\varepsilon>0$ we define the thin heterogeneous domain
$\Omega_{\varepsilon}$ in $\mathbb{R}^{d}$ by
\[
\Omega_{\varepsilon}=\Omega\times(-\varepsilon,\varepsilon)=\left\{
(\overline{x},x_{d})\in\mathbb{R}^{d}:\overline{x}\in\Omega\text{ and
}-\varepsilon<x_{d}<\varepsilon\right\}  .
\]
In the thin layer $\Omega_{\varepsilon}$, the flow of two-phase immiscible
fluids at the micro-scale is described by the Stokes-Cahn-Hilliard system
\begin{equation}
\left\{
\begin{array}
[c]{l}%
\dfrac{\partial\boldsymbol{u}_{\varepsilon}}{\partial t}-\alpha\varepsilon
^{2}\Delta\boldsymbol{u}_{\varepsilon}+\nabla p_{\varepsilon}-\mu
_{\varepsilon}\nabla\varphi_{\varepsilon}=\boldsymbol{h}\text{ in
}Q_{\varepsilon}=(0,T)\times\Omega_{\varepsilon}\\
\\
\operatorname{div}\boldsymbol{u}_{\varepsilon}=0\text{ in }Q_{\varepsilon}\\
\\
\dfrac{\partial\varphi_{\varepsilon}}{\partial t}+\boldsymbol{u}_{\varepsilon
}\cdot\nabla\varphi_{\varepsilon}-\Delta\mu_{\varepsilon}=0\text{ in
}Q_{\varepsilon}\\
\\
\mu_{\varepsilon}=-\beta\Delta\varphi_{\varepsilon}+\lambda f(\varphi
_{\varepsilon})\text{ in }Q_{\varepsilon}\\
\\
\dfrac{\partial\mu_{\varepsilon}}{\partial\nu}=0\text{, }\dfrac{\partial
\varphi_{\varepsilon}}{\partial\nu}=0\text{ and }\boldsymbol{u}_{\varepsilon
}=0\text{\ on }(0,T)\times\partial\Omega_{\varepsilon}\\
\\
\boldsymbol{u}_{\varepsilon}(0,x)=\boldsymbol{u}_{0}^{\varepsilon}(x)\text{
and }\varphi_{\varepsilon}(0,x)=\varphi_{0}^{\varepsilon}(x)\text{ in }%
\Omega_{\varepsilon},
\end{array}
\right. \label{1.1}%
\end{equation}
where $\alpha$, $\beta$ and $\lambda$ are positive fixed parameters, and $\nu$
is a unit outward normal to $\partial\Omega_{\varepsilon}$. Here,
$\boldsymbol{u}_{\varepsilon}$, $p_{\varepsilon}$, $\varphi_{\varepsilon}$ and
$\mu_{\varepsilon}$ are respectively the unknown velocity, pressure, the order
parameter and the chemical potential. The order parameter $\varphi
_{\varepsilon}$ is the difference of the fluid relative concentrations and
usually takes values between $-1$ and $1$. In (\ref{1.1}), $\nabla$ (resp.
$\operatorname{div}$ and $\Delta$) denotes the usual gradient (resp.
divergence and Laplace) operator in $\Omega_{\varepsilon}$. The function
$\boldsymbol{h}$ has the form
\begin{equation}
\boldsymbol{h}(t,x)=(\boldsymbol{h}_{1}(t,\overline{x}),0)\text{ for a.e.
}(t,x=(\overline{x},x_{d}))\in(0,T)\times\Omega\times(-1,1)\equiv
Q_{1},\label{1.2}%
\end{equation}
where $\boldsymbol{h}_{1}\in L^{2}((0,T)\times\Omega)^{d-1}$. The function
$f\in\mathcal{C}^{2}(\mathbb{R})$ satisfies
\begin{equation}
\underset{\left\vert r\right\vert \rightarrow\infty}{\lim\inf}f^{\prime
}(r)>0\text{ and }\left\vert f^{\prime\prime}(r)\right\vert \leq
c_{f}(1+\left\vert r\right\vert )\ \forall r\in\mathbb{R},\label{1.3}%
\end{equation}
where $c_{f}$ is a positive constant.

Finally the initial conditions $\boldsymbol{u}_{0}^{\varepsilon}\in
L^{2}(\Omega_{\varepsilon})^{d}$ and $\varphi_{0}^{\varepsilon}\in
H^{1}(\Omega_{\varepsilon})$ satisfy the conditions
\begin{equation}
\left\Vert \boldsymbol{u}_{0}^{\varepsilon}\right\Vert _{L^{2}(\Omega
_{\varepsilon})^{d}}+\left\Vert \varphi_{0}^{\varepsilon}\right\Vert
_{H^{1}(\Omega_{\varepsilon})}\leq C\varepsilon^{\frac{1}{2}}\text{ and }%
\int_{\Omega_{\varepsilon}}F(\varphi_{0}^{\varepsilon})dx\leq C\varepsilon
,\label{1.4}%
\end{equation}
where $C>0$ is a constant independent of $\varepsilon$ and
\begin{equation}
F(r)=\int_{0}^{r}f(s)ds,\ r\in\mathbb{R}%
,\ \ \ \ \ \ \ \ \ \ \ \ \ \ \ \ \ \ \ \ \ \ \ \ \ \ \ \ \label{1.5}%
\end{equation}
and we assume without loss of generality that
\begin{equation}
\varepsilon^{-\frac{1}{2}}\left\Vert \boldsymbol{u}_{0}^{\varepsilon
}-\boldsymbol{u}^{0}\right\Vert _{L^{2}(\Omega_{\varepsilon})^{d}}%
\rightarrow0\text{ and }\varepsilon^{-\frac{1}{2}}\left\Vert \varphi
_{0}^{\varepsilon}-\varphi^{0}\right\Vert _{L^{2}(\Omega_{\varepsilon}%
)}\rightarrow0\label{1.6}%
\end{equation}
when $\varepsilon\rightarrow0$, where $\boldsymbol{u}^{0}\in L^{2}(\Omega
)^{d}$ and $\varphi^{0}\in H^{1}(\Omega)$.

It follows from (\ref{1.3}) that
\begin{equation}%
\begin{array}
[c]{l}%
\left\vert f^{\prime}(r)\right\vert \leq C(1+\left\vert r\right\vert
^{2})\text{, }\left\vert f(r)\right\vert \leq C(1+\left\vert r\right\vert
^{3})\text{ and }\\
\left\vert f^{\prime}(r)-f^{\prime}(s)\right\vert \leq C(1+\left\vert
r\right\vert +\left\vert s\right\vert )\left\vert r-s\right\vert \ \ \forall
r,s\in\mathbb{R},
\end{array}
\label{1.7}%
\end{equation}
for a positive constant $C$ depending on $f$.

A typical example of regular double well potential is the Landau potential
\[
F(r)=\frac{1}{4}(r^{2}-1)^{2}%
,\ \ \ \ \ \ \ \ \ \ \ \ \ \ \ \ \ \ \ \ \ \ \ \ \ \ \ \ \ \ \ \ \ \ \ \ \ \ \ \ \ \
\]
a function fulfilling conditions (\ref{1.3}), (\ref{1.5}) and (\ref{1.7}%
).\ One can also consider a fourth order polynomial with positive leading coefficient.

Throughout the work, we will denote by (\ref{1.1})$_{i}$ the $i$th equation of
system (\ref{1.1}).

\begin{remark}
\label{r1}\emph{Assumption (\ref{1.6}) is physically relevant. Indeed we may
think of }$\boldsymbol{u}_{0}^{\varepsilon}$\emph{\ as a solution of the
Stokes system }%
\[
\left\{
\begin{array}
[c]{l}%
-\Delta\boldsymbol{u}_{0}^{\varepsilon}+\nabla p_{0}^{\varepsilon
}=g\text{\emph{\ in }}\Omega_{\varepsilon},\\
\operatorname{div}\boldsymbol{u}_{0}^{\varepsilon}=0\text{\emph{\ in }}%
\Omega_{\varepsilon}\text{\emph{\ and }}\boldsymbol{u}_{0}^{\varepsilon
}=0\text{\emph{\ on }}\partial\Omega_{\varepsilon},
\end{array}
\right.
\ \ \ \ \ \ \ \ \ \ \ \ \ \ \ \ \ \ \ \ \ \ \ \ \ \ \ \ \ \ \ \ \ \ \ \ \ \ \
\]
\emph{with }$g(x)=(g_{1}(\overline{x}),0)$\emph{, }$g_{1}\in L^{2}%
(\Omega)^{d-1}$\emph{. Then by standard energy estimates, we get }$\left\Vert
\boldsymbol{u}_{0}^{\varepsilon}\right\Vert _{H_{0}^{1}(\Omega_{\varepsilon
})^{d}}\leq C\varepsilon^{1/2}$\emph{. Therefore, appealing to the two-scale
convergence for thin periodic domains (see e.g. Ref. \cite{RJ2007}) we derive the
existence of }$\boldsymbol{u}^{0}\in L^{2}(\Omega)^{d}$\emph{\ such that
}$\varepsilon^{-\frac{1}{2}}\left\Vert \boldsymbol{u}_{0}^{\varepsilon
}-\boldsymbol{u}^{0}\right\Vert _{L^{2}(\Omega_{\varepsilon})^{d}}%
\rightarrow0$\emph{\ as }$\varepsilon\rightarrow0$\emph{. We may do the same
for }$\varphi_{0}^{\varepsilon}$\emph{.}
\end{remark}

The $\varepsilon$-model (\ref{1.1}) consists of a convective Cahn-Hilliard
equation coupled with the Stokes equation through the surface tension term
$\mu_{\varepsilon}\nabla\varphi_{\varepsilon}$. Thus (\ref{1.1}) belongs to
the class of diffuse interface models that are used to describe the behaviour
of multi-phase fluids. It is also very important to note that the scaling in
(\ref{1.1})$_{1}$ is exactly the one leading to memory effects in the
upscaling limit. Indeed; it was shown in Ref. \cite{Allaire} that the exact scaling
for the Darcy law with memory in the time dependent Stokes system was the one
considered in (\ref{1.1}). So, the main goal of this contribution is to
investigate the asymptotic behaviour when $\varepsilon\rightarrow0$, of the
sequence of solutions to (\ref{1.1}).

The motivation for this study lies at several levels some of which are
enumerated below.

--- \emph{The domain}. There is a huge literature on homogenization in fixed
or porous media. A few works deal with the homogenization theory in thin
heterogeneous domains; see e.g. Ref. \cite{Gahn2,Gahn1,Gahn3,Gahn4,RJ2007}. All
the previous works deal with thin periodic structures. Our model problem is
stated in a highly heterogeneous thin domain whose heterogeneities are
distributed inside in a general deterministic way including the periodic one,
the almost periodic one and others. Therefore we need to develop a suitable
version of the sigma-convergence for thin domains, which generalizes the
two-scale convergence concept for thin periodic structures introduced in
\cite{RJ2007} by the second author.

--- \emph{The model}. Several works have considered homogenization of single
phase fluid. The most relevant ones are concerned with the derivation of Darcy
and Darcy-type laws (see for instance Ref. \cite{Allaire,Mikelic}). We also refer
the reader to \cite{CLS2013} in which the study of the asymptotic behaviour of
solutions of the Navier-Stokes system in a thin domain satisfying the Navier
boundary condition on a periodic rough surface is considered. Contrasting with
the study of single phase fluids, the homogenization theory for multi-phase
flow is less developed. Let us mention Ref. \cite{Auriault1989,Banas2017,Woukeng2020,Daly2015,Hornung1997,Schmuck2012,Pop2022}. In the current
contribution, we deal with a model for two-phase thin heterogeneous media flow
with surface tension described by (\ref{1.1}).

--- \emph{The expected upscaled model}. One of the main motivations of this
study is the expected homogenized model (corresponding to the $3D$
$\varepsilon$-model) which, to the best of our knowledge, is new and is stated
below as one of the main results.

\begin{theorem}
\label{t1.1}Assume $d=3$. For each $\varepsilon>0$, let $(\boldsymbol{u}%
_{\varepsilon},\varphi_{\varepsilon},\mu_{\varepsilon},p_{\varepsilon}) $ be
the unique solution of \emph{(\ref{1.1})}. Then up to a subsequence not
relabeled, $(\boldsymbol{u}_{\varepsilon},\mu_{\varepsilon},p_{\varepsilon
})_{\varepsilon>0}$ weakly $\Sigma_{A}$-converges (as $\varepsilon
\rightarrow0$) in $L^{2}(Q_{\varepsilon})^{3}\times L^{2}(Q_{\varepsilon
})\times L^{2}(Q_{\varepsilon})$ towards $(\boldsymbol{u}_{0},\mu_{0},p_{0})$
and $(\varphi_{\varepsilon})_{\varepsilon>0}$ strongly $\Sigma_{A}$-converges
in $L^{2}(Q_{\varepsilon})$ towards $\varphi_{0}$ with $\varphi_{0}\in
L^{\infty}(0,T;H^{1}(\Omega))$, $\boldsymbol{u}_{0}\in L^{2}(Q;\mathcal{B}%
_{A}^{1,2}(\mathbb{R}^{2};H_{0}^{1}(I))^{3})$, $\mu_{0}\in L^{2}%
(0,T;H^{1}(\Omega))$ and $p_{0}\in L^{2}(0,T;L_{0}^{2}(\Omega))$. Setting
\[
M_{\varepsilon}\phi(t,\overline{x})=\frac{1}{2\varepsilon}\int_{-\varepsilon
}^{\varepsilon}\phi(t,\overline{x},\zeta)d\zeta\text{ for }(t,\overline{x})\in
Q,
\]
and
\[
\boldsymbol{u}(t,\overline{x})=\frac{1}{2}\int_{-1}^{1}M(\boldsymbol{u}%
_{0}(t,\overline{x},\cdot,\zeta))d\zeta\equiv(\overline{\boldsymbol{u}%
}(t,\overline{x}),u_{3}(t,\overline{x})),
\]
one has $u_{3}=0$ and, up to the same subsequence above, we have, as
$\varepsilon\rightarrow0$,
\begin{equation}%
\begin{array}
[c]{l}%
M_{\varepsilon}\boldsymbol{u}_{\varepsilon}\rightarrow(\overline
{\boldsymbol{u}},0)\text{ in }L^{2}(Q)^{3}\text{-weak, \ }M_{\varepsilon
}\varphi_{\varepsilon}\rightarrow\varphi_{0}\text{ in }L^{2}(Q)\text{-strong,}%
\\
M_{\varepsilon}\mu_{\varepsilon}\rightarrow\mu_{0}\text{ in }L^{2}%
(Q)\text{-weak and }M_{\varepsilon}p_{\varepsilon}\rightarrow p_{0}\text{ in
}L^{2}(Q)\text{-weak.}%
\end{array}
\label{1.8'}%
\end{equation}
Moreover it holds that $\overline{\boldsymbol{u}}\in\mathcal{C}%
([0,T];\mathbb{H})$, $\varphi_{0}\in\mathcal{C}([0,T];H^{1}(\Omega))\cap
L^{2}(0,T;H^{3}(\Omega))$, $p_{0}\in L^{2}(0,T;H^{1}(\Omega)\cap L_{0}%
^{2}(\Omega))$ and the quadruple $(\overline{\boldsymbol{u}},\varphi_{0}%
,\mu_{0},p_{0})$ is a weak solution of the effective $2D$ problem
\begin{equation}
\left\{
\begin{array}
[c]{l}%
\overline{\boldsymbol{u}}=G\overline{\boldsymbol{u}}^{0}+G\ast(\boldsymbol{h}%
_{1}+\mu_{0}\nabla_{\overline{x}}\varphi_{0}-\nabla_{\overline{x}}p_{0})\text{
in }Q,\\
\\
\operatorname{div}_{\overline{x}}\overline{\boldsymbol{u}}=0\text{ in }Q\text{
and }\overline{\boldsymbol{u}}\cdot\boldsymbol{n}=0\text{ on }(0,T)\times
\partial\Omega,\\
\\
\dfrac{\partial\varphi_{0}}{\partial t}+\overline{\boldsymbol{u}}\cdot
\nabla_{\overline{x}}\varphi_{0}-\Delta_{\overline{x}}\mu_{0}=0\text{ in }Q,\\
\\
\mu_{0}=-\beta\Delta_{\overline{x}}\varphi_{0}+\lambda f(\varphi_{0})\text{ in
}Q,\\
\\
\dfrac{\partial\varphi_{0}}{\partial\boldsymbol{n}}=\dfrac{\partial\mu_{0}%
}{\partial\boldsymbol{n}}=0\text{ on }(0,T)\times\partial\Omega,\\
\\
\varphi_{0}(0)=\varphi^{0}\text{ in }\Omega,
\end{array}
\right.  \ \ \ \ \ \ \ \ \ \ \ \ \ \ \ \ \ \ \ \ \ \ \ \ \ \ \ \ \label{1.8}%
\end{equation}
where $\ast$ stands for the convolution operator with respect to time and
$G=(G_{ij})_{1\leq i,j\leq2}$ is a symmetric positive definite $2\times2$
matrix defined by its entries $G_{ij}(t)=\frac{1}{2}\int_{-1}^{1}M(\omega
^{i}(t,\cdot,\zeta))e_{j}d\zeta$. Here $\omega^{j}=(\omega_{i}^{j})_{1\leq
i\leq3}$ is the unique solution in $\mathcal{C}(0,T;\mathcal{B}_{A}%
^{2}(\mathbb{R}^{2};L^{2}(I))^{3})\cap L^{2}(0,T;\mathcal{B}_{A}%
^{1,2}(\mathbb{R}^{2};H_{0}^{1}(I))^{3})$ of the auxiliary Stokes system
\[
\left\{
\begin{array}
[c]{l}%
\dfrac{\partial\omega^{j}}{\partial t}-\alpha\overline{\Delta}_{y}\omega
^{j}+\overline{\nabla}_{y}\pi^{j}=0\text{ in }(0,T)\times\mathbb{R}^{2}\times
I,\\
\\
\overline{\operatorname{div}}_{y}\omega^{j}=0\text{ in }(0,T)\times
\mathbb{R}^{2}\times I,\\
\\
\omega^{j}(0)=e_{j}\text{ in }\mathbb{R}^{2}\times I\text{ and }%
{\displaystyle\int_{-1}^{1}}
M(\omega_{3}^{j}(t,\cdot,\zeta))d\zeta=0,
\end{array}
\right.
\]
$e_{j}$ being the $j$th vector of the canonical basis in $\mathbb{R}^{3}$.
Assuming $\varphi^{0}\in H^{2}(\Omega)$ with $\nabla\varphi^{0}\cdot
\boldsymbol{n}=0$ on $\partial\Omega$, then $\varphi_{0}\in\mathcal{C}%
([0,T];H^{2}(\Omega))\cap L^{2}(0,T;H^{4}(\Omega))\cap H^{1}(0,T;L^{2}%
(\Omega))$, $\mu_{0}\in\mathcal{C}([0,T];H^{1}(\Omega))\cap L^{2}%
(0,T;H^{2}(\Omega))$, and the quadruple $(\overline{\boldsymbol{u}}%
,\varphi_{0},\mu_{0},p_{0})$ is the unique solution of \emph{(\ref{1.8})}, so
that the whole sequence $(\boldsymbol{u}_{\varepsilon},\varphi_{\varepsilon
},\mu_{\varepsilon},p_{\varepsilon})_{\varepsilon>0}$ converges in the sense
of \emph{(\ref{1.8'})}.
\end{theorem}

Here above in Theorem \ref{t1.1}, the letter $M$ and the space $\mathcal{B}%
_{A}^{2}$ stand respectively for the mean value operator and the generalized
Besicovitch space associated to the algebra with mean value $A$; see Section
\ref{sec3} for details about these concepts.

The equation (\ref{1.8})$_{1}$ is a Hele-Shaw equation with memory, that is, a
nonlocal (in time) Hele-Shaw equation. The system (\ref{1.8}) is an
interesting variant of the Hele-Shaw-Cahn-Hilliard system since it requires
the initial value for the velocity. Moreover, the pressure, the velocity, the
order parameter and the chemical potential depend on the history of the system
and there is no non-physical jump in velocity at $t=0$. It has many
applications in two-phase flow in porous media and Hele-Shaw cell, but also
widely used to model tumor growth\cite{Lowengrub2013,Wise2008}. It is
therefore a \emph{nonlocal} (in time) \emph{Hele-Shaw-Cahn-Hilliard} (HSCH)
system. Although this could have been foreseen, surprisingly, to the best of
our knowledge, this is the first time that such a system is derived in the
literature. For that reason, we need to make a qualitative analysis of
(\ref{1.8}) in order to prove some regularity results and its well-posedness.
This is one of the main aims of this work.

There are some studies regarding the analysis of the local version of
(\ref{1.8}), that is the version in which (\ref{1.8})$_{1}$ is replaced by the
following equation
\[
\overline{\boldsymbol{u}}=\boldsymbol{h}_{1}+\mu_{0}\nabla_{\overline{x}%
}\varphi_{0}-\nabla_{\overline{x}}p_{0}\text{ in }Q\text{.}%
\ \ \ \ \ \ \ \ \ \ \ \ \ \ \ \ \ \ \ \ \ \ \ \ \ \ \ \ \ \ \
\]
Indeed, in Ref. \cite{Wise2010}, the local version was studied numerically. It has
also been studied analytically in Ref. \cite{Feng2012} where existence and
uniqueness of weak solutions in two or three dimensional bounded domains were
proved, and in Ref. \cite{Wang2012,Wu2012} where the well-posedness and longtime
behaviour of strong solutions in two or three dimensional torus were
considered. We also cite Ref. \cite{Lowengrub2013} where systematic analysis of the
local version was considered in a $2D$ rectangle or in a $3D$ parallelepiped.

In our study, after the derivation of model (\ref{1.8}), we are concerned with
its analysis. Precisely, we improve the regularity of its solutions by
establishing some regularity estimates. We rely on these regularity results to
prove the well-posedness of (\ref{1.8}). To the best of our knowledge, this is
the first time that such a model is derived and analyzed in the literature.

The second main result of the work corresponds to the $2D$ $\varepsilon$-model
posed in $\Omega_{\varepsilon}=(a,b)\times(-\varepsilon,\varepsilon)$. It
reads as follows.

\begin{theorem}
\label{t1.2}Assume $d=2$ and $\boldsymbol{u}^{0}=0$. For each $\varepsilon>0
$, let $(\boldsymbol{u}_{\varepsilon},\varphi_{\varepsilon},\mu_{\varepsilon
},p_{\varepsilon})$ be as in Theorem \emph{\ref{t1.1}}. Then the sequence
$(\boldsymbol{u}_{\varepsilon},\mu_{\varepsilon},p_{\varepsilon}%
)_{\varepsilon>0}$ weakly $\Sigma_{A}$-converges (as $\varepsilon\rightarrow
0$) in $L^{2}(Q_{\varepsilon})^{2}\times L^{2}(Q_{\varepsilon})\times
L^{2}(Q_{\varepsilon})$ towards $(\boldsymbol{u}_{0},\mu_{0},p_{0})$ and the
sequence $(\varphi_{\varepsilon})_{\varepsilon>0}$ strongly $\Sigma_{A}%
$-converges in $L^{2}(Q_{\varepsilon})$ towards $\varphi_{0}$ with
$\varphi_{0}\in L^{\infty}(0,T;H^{1}(\Omega))$, $\boldsymbol{u}_{0}\in
L^{2}(Q;\mathcal{B}_{A}^{1,2}(\mathbb{R};H_{0}^{1}(I))^{2})$, $\mu_{0}\in
L^{2}(0,T;H^{1}(\Omega))$ and $p_{0}\in L^{2}(0,T;L_{0}^{2}(\Omega))$.
Moreover setting
\[
\boldsymbol{u}(t,x_{1})=\frac{1}{2}\int_{-1}^{1}M(\boldsymbol{u}_{0}%
(t,x_{1},\cdot,\zeta))d\zeta,
\]
one has $\boldsymbol{u}=0$, and the couple $(\varphi_{0},\mu_{0})$ is the
unique solution to the $1D$ Cahn-Hilliard equation
\begin{equation}
\left\{
\begin{array}
[c]{l}%
\dfrac{\partial\varphi_{0}}{\partial t}-\dfrac{\partial^{2}\mu_{0}}{\partial
x_{1}^{2}}=0\text{ \ in }(0,T)\times(a,b),\\
\\
\mu_{0}=-\beta\dfrac{\partial^{2}\varphi_{0}}{\partial x_{1}^{2}}+\lambda
f(\varphi_{0})\text{ in }(0,T)\times(a,b),\\
\\
\varphi_{0}^{\prime}(t,a)=\varphi_{0}^{\prime}(t,b)=0,\ \mu_{0}^{\prime
}(t,a)=\mu_{0}^{\prime}(t,b)=0\text{ in }(0,T),\\
\\
\varphi_{0}(0)=\varphi^{0}\text{ in }(a,b).
\end{array}
\right. \label{1.9}%
\end{equation}
Furthermore the pressure $p_{0}$ is the unique solution to the equation
\begin{equation}
\dfrac{\partial p_{0}}{\partial x_{1}}=\boldsymbol{h}_{1}+\mu_{0}%
\dfrac{\partial\varphi_{0}}{\partial x_{1}},\ \ \ \int_{a}^{b}p_{0}%
dx_{1}=0.\ \ \ \ \ \ \ \ \ \ \ \ \ \ \ \ \ \ \ \ \ \label{1.10}%
\end{equation}

\end{theorem}

The plan of this paper goes as follows. In Section \ref{sec2}, we recall the
well-posedness and derive some useful uniform estimates for the sequence of
solutions of (\ref{1.1}). Section \ref{sec3} deals with the treatment of the
concept of sigma-convergence for thin heterogeneous domains. We prove therein
some compactness results that will be used in the homogenization process. With
the help of the results obtained in Section \ref{sec3}, we pass to the limit
in (\ref{1.1}) in Section \ref{sec4} and derive the upscaled model. We next
analyze the $2D$ model (obtained in Section \ref{sec4}) and prove its
well-posedness in Section \ref{sec5}. We close this section by the proof of
the main results of the work.

Unless otherwise specified, the vector spaces throughout are assumed to be
real vector spaces, and the scalar functions are assumed to take real values.
We shall always assume that the numerical space $\mathbb{R}^{m}$ (integer
$m\geq1$) and its open sets are each provided with the Lebesgue measure
denoted by $dx=dx_{1}...dx_{m}$. Finally we will adopt the following notation
in the remaining part of the work. If $A=(a_{ij})_{1\leq i,j\leq m}$ and
$B=(b_{ij})_{1\leq i,j\leq m}$, we denote $A\cdot B:=\sum_{i,j=1}^{m}%
a_{ij}b_{ij}$; we use the same notation for the scalar product in
$\mathbb{R}^{m}$, namely, if $\boldsymbol{u}=(u_{i})_{1\leq i\leq m}$ and
$\boldsymbol{v}=(v_{i})_{1\leq i\leq m}$, then $\boldsymbol{u}\cdot
\boldsymbol{v}=\sum_{i=1}^{m}u_{i}v_{i}$.

\section{Existence result and uniform estimates\label{sec2}}

\subsection{Existence result}

In order to define the notion of weak solutions we will deal with in this
work, we first introduce the functional setup. Let $X$ be a Banach space. The
notation $\left\langle \cdot,\cdot\right\rangle $ will stand for the duality
pairings between $X$ and its topological dual $X^{\prime}$ while $\mathbb{X}$
will denote the space $X\times\cdots\times X$ ($d$ times) endowed with the
product structure. If in particular $X$ is a real Hilbert space with inner
product $(\cdot,\cdot)_{X}$, then we denote by $\left\Vert \cdot\right\Vert
_{X}$ the induced norm. Especially, by $\mathbb{H}_{\varepsilon}$ and
$\mathbb{V}_{\varepsilon}$ we denote the Hilbert spaces defined as the closure
in $\mathbb{L}^{2}(\Omega_{\varepsilon})=L^{2}(\Omega_{\varepsilon})^{d}$
(resp. $\mathbb{H}_{0}^{1}(\Omega^{\varepsilon})=H_{0}^{1}(\Omega
_{\varepsilon})^{d}$) of the space $\{\boldsymbol{u}\in\mathbb{C}_{0}^{\infty
}(\Omega_{\varepsilon}):\text{${\rm div}$}\boldsymbol{u}=0$ in $\Omega
_{\varepsilon}\}$ where $\mathbb{C}_{0}^{\infty}(\Omega_{\varepsilon
})=\mathcal{C}_{0}^{\infty}(\Omega_{\varepsilon})^{d}$. Then $\mathbb{V}%
_{\varepsilon}=\{\boldsymbol{u}\in\mathbb{H}_{0}^{1}(\Omega_{\varepsilon
}):{\rm div}$$\boldsymbol{u}=0$ in $\Omega_{\varepsilon}\}$ and $\mathbb{H}%
_{\varepsilon}=\{\boldsymbol{u}\in\mathbb{L}^{2}(\Omega_{\varepsilon}%
):{\rm div}$$\boldsymbol{u}=0$ in $\Omega_{\varepsilon}$ and $\boldsymbol{u}%
\cdot\nu=0$ on $\partial\Omega_{\varepsilon}\}$ where $\nu$ is the outward
unit normal to $\partial\Omega_{\varepsilon}$. The space $\mathbb{H}%
_{\varepsilon}$ is endowed with the scalar product denoted by $(\cdot,\cdot)$
whose associated norm is denoted by $\left\Vert \cdot\right\Vert
_{\mathbb{H}_{\varepsilon}}$. The space $\mathbb{V}_{\varepsilon}$ is equipped
with the scalar product
\[
(\boldsymbol{u},\boldsymbol{v}):=(\nabla\boldsymbol{u},\nabla\boldsymbol{v}%
)\ \ (\boldsymbol{u},\boldsymbol{v}\in\mathbb{V}_{\varepsilon})
\]
whose associated norm is the norm of the gradient and is denoted by
$\left\Vert \cdot\right\Vert _{\mathbb{V}_{\varepsilon}}$. Owing to the
Poincar\'{e} inequality, the norm in $\mathbb{V}_{\varepsilon}$\ is equivalent
to the $\mathbb{H}^{1}(\Omega_{\varepsilon})$-norm. We also define the space
$L_{0}^{2}(\Omega_{\varepsilon})=\{v\in L^{2}(\Omega_{\varepsilon}%
):\int_{\Omega_{\varepsilon}}vdx=0\}$. We denote by $\mathbb{V}$ (resp.
$\mathbb{H}$) the space defined as $\mathbb{V}_{\varepsilon}$ (resp.
$\mathbb{H}_{\varepsilon}$) when replacing $\Omega_{\varepsilon}$ by $\Omega$.
For the sake of simplicity, we shall often use the notation \medskip
$\left\Vert \cdot\right\Vert _{H^{s}}$ to denote the norm in $H^{s}(G)$ for
$s$ an integer and $G$ any open subset of $\mathbb{R}^{m}$ (integer $m\geq1$).

This being so, the concept of weak solution we will deal with in this work, is
defined as follows.

\begin{definition}
\label{d1.1}\emph{Let }$\boldsymbol{u}_{0}^{\varepsilon}\in\mathbb{H}%
_{\varepsilon}$\emph{\ and }$\varphi_{0}^{\varepsilon}\in H^{1}(\Omega
_{\varepsilon})$\emph{\ with }$F(\varphi_{0}^{\varepsilon})\in L^{1}%
(\Omega_{\varepsilon})$\emph{, and let }$0<T<\infty$\emph{\ be given. The
triplet }$(\boldsymbol{u}_{\varepsilon},\varphi_{\varepsilon},\mu
_{\varepsilon})$\emph{\ is a weak solution to (\ref{1.1}) if }

\begin{itemize}
\item \emph{It holds that }

\begin{itemize}
\item[(i)] $\boldsymbol{u}_{\varepsilon}\in L^{\infty}(0,T;\mathbb{H}%
_{\varepsilon})\cap L^{2}(0,T;\mathbb{V}_{\varepsilon})$\emph{\ with
}$\partial\boldsymbol{u}_{\varepsilon}/\partial t\in L^{2}(0,T;\mathbb{V}%
_{\varepsilon}^{\prime}),$

\item[(ii)] $\varphi_{\varepsilon}\in L^{\infty}(0,T;H^{1}(\Omega
_{\varepsilon}))$\emph{\ with }$\partial\varphi_{\varepsilon}/\partial t\in
L^{2}(0,T;H^{1}(\Omega_{\varepsilon})^{\prime}),$

\item[(iii)] $\mu_{\varepsilon}\in L^{2}(0,T;H^{1}(\Omega_{\varepsilon}));$
\end{itemize}

\item \emph{For all }$\phi,\chi\in L^{2}(0,T;H^{1}(\Omega_{\varepsilon}))
$\emph{\ and all }$\psi\in L^{2}(0,T;\mathbb{V}_{\varepsilon})$\emph{, }%
\begin{equation}
\int_{0}^{T}\left\langle \frac{\partial\boldsymbol{u}_{\varepsilon}}{\partial
t},\psi\right\rangle dt+\alpha\varepsilon^{2}\int_{Q_{\varepsilon}}%
\nabla\boldsymbol{u}_{\varepsilon}\cdot\nabla\psi dxdt+\int_{Q_{\varepsilon}%
}(\psi\cdot\nabla\mu_{\varepsilon})\varphi_{\varepsilon}dxdt=\int
_{Q_{\varepsilon}}\boldsymbol{h}\psi dxdt,\label{e2.1}%
\end{equation}%
\begin{equation}
\int_{0}^{T}\left\langle \frac{\partial\varphi_{\varepsilon}}{\partial t}%
,\phi\right\rangle dt-\int_{Q_{\varepsilon}}(\boldsymbol{u}_{\varepsilon}%
\cdot\nabla\phi)\varphi_{\varepsilon}dxdt+\int_{Q_{\varepsilon}}\nabla
\mu_{\varepsilon}\cdot\nabla\phi dxdt=0,\label{e2.2}%
\end{equation}%
\begin{equation}
\int_{Q_{\varepsilon}}\mu_{\varepsilon}\chi dxdt=\beta\int_{Q_{\varepsilon}%
}\nabla\varphi_{\varepsilon}\cdot\nabla\chi dxdt+\lambda\int_{Q_{\varepsilon}%
}f(\varphi_{\varepsilon})\chi dxdt;\label{e2.3}%
\end{equation}

\item $\boldsymbol{u}_{\varepsilon}(0)=\boldsymbol{u}_{0}^{\varepsilon}%
$\emph{\ and }$\varphi_{\varepsilon}(0)=\varphi_{0}^{\varepsilon}$\emph{.}
\end{itemize}

\noindent\emph{Furthermore to each weak solution }$(\boldsymbol{u}%
_{\varepsilon},\varphi_{\varepsilon},\mu_{\varepsilon})$\emph{\ is associated
a pressure }$p_{\varepsilon}\in L^{2}(0,T;L_{0}^{2}(\Omega_{\varepsilon}%
))$\emph{\ that satisfies (\ref{1.1})}$_{1}$\emph{\ in the distributional
sense.}
\end{definition}

The existence of a weak solution in the sense of Definition \ref{d1.1} has
been extensively addressed by many authors; see e.g. Ref. \cite{Colli2012,Feng2006} in which a more general system (the Stokes equation is replaced
therein by the Navier-Stokes one) is treated. Following the same way of
reasoning as in the above cited references, we get straightforwardly the
following result that can be proved exactly as its homologue in
Ref. \cite{Colli2012}.

\begin{theorem}
\label{th2.1}For each fixed $\varepsilon>0$, let $\boldsymbol{u}%
_{0}^{\varepsilon}\in\mathbb{H}_{\varepsilon}$ and $\varphi_{0}^{\varepsilon
}\in H^{1}(\Omega_{\varepsilon})$ with $F(\varphi_{0}^{\varepsilon})\in
L^{1}(\Omega_{\varepsilon})$. Then under assumptions \emph{(\ref{1.2})} and
\emph{(\ref{1.3})}, there exists a unique weak solution $(\boldsymbol{u}%
_{\varepsilon},\varphi_{\varepsilon},\mu_{\varepsilon})$ to \emph{(\ref{1.1})}
in the sense of Definition \emph{\ref{d1.1}}. Moreover $\varphi_{\varepsilon
}\in L^{2}(0,T;H^{2}(\Omega_{\varepsilon}))$, and there exists a unique
$p_{\varepsilon}\in L^{2}(0,T;L_{0}^{2}(\Omega_{\varepsilon}))$ such that
\emph{(\ref{1.1})}$_{1}$ is satisfied in the distributional sense.
\end{theorem}

\begin{proof}
The existence of a unique $(\boldsymbol{u}_{\varepsilon},\varphi_{\varepsilon
},\mu_{\varepsilon})$ follows by applying step by step the method used in
Ref. \cite{Colli2012} mutatis mutandis. To show that $\varphi_{\varepsilon}\in
L^{2}(0,T;H^{2}(\Omega_{\varepsilon}))$, we notice that $\varphi_{\varepsilon
}(t)$ (for a.e. $t\in(0,T)$) solves the Neumann problem
\[
-\Delta\varphi_{\varepsilon}=\mu_{\varepsilon}-f(\varphi_{\varepsilon})\text{
in }\Omega_{\varepsilon}\text{, }\frac{\partial\varphi_{\varepsilon}}%
{\partial\nu}=0\text{ on }\partial\Omega_{\varepsilon}.
\]
Owing to (\ref{1.7}), we have $f(\varphi_{\varepsilon}(t))\in L^{2}%
(\Omega_{\varepsilon})$ for a.e. $t\in(0,T)$. Indeed, we have
\[
\int_{\Omega_{\varepsilon}}\left\vert f(\varphi_{\varepsilon}(t))\right\vert
^{2}dx\leq C\int_{\Omega_{\varepsilon}}(1+\left\vert \varphi_{\varepsilon
}(t)\right\vert ^{6})dx,
\]
so that the continuous embedding $H^{1}(\Omega_{\varepsilon})\hookrightarrow
L^{6}(\Omega_{\varepsilon})$ yields $\left\Vert \varphi_{\varepsilon
}(t)\right\Vert _{L^{6}(\Omega_{\varepsilon})}\leq C\left\Vert \varphi
_{\varepsilon}(t)\right\Vert _{H^{1}(\Omega_{\varepsilon})}$, and hence
\[
\int_{\Omega_{\varepsilon}}\left\vert f(\varphi_{\varepsilon}(t))\right\vert
^{2}dx\leq C+C\left\Vert \varphi_{\varepsilon}(t)\right\Vert _{H^{1}%
(\Omega_{\varepsilon})}^{6}.
\]
Thus $f(\varphi_{\varepsilon})\in L^{\infty}(0,T;L^{2}(\Omega_{\varepsilon}%
))$. Therefore $\mu_{\varepsilon}(t)-f(\varphi_{\varepsilon}(t))\in
L^{2}(\Omega_{\varepsilon})$, a.e. $t\in(0,T)$. By a classical regularity
result, we get $\varphi_{\varepsilon}(t)\in H^{2}(\Omega_{\varepsilon})$, and
so $\varphi_{\varepsilon}\in L^{2}(0,T;H^{2}(\Omega_{\varepsilon}))$.

For the existence of the pressure, since $\boldsymbol{h}\in L^{2}%
(0,T;\mathbb{H}^{-1}(\Omega_{\varepsilon}))$, the necessary condition of Section 4 in Ref. 
\cite{Simon1999} for the existence of the pressure is satisfied.
Next, let us set
\[
\boldsymbol{h}_{\varepsilon}=\boldsymbol{h}-\frac{\partial\boldsymbol{u}%
_{\varepsilon}}{\partial t}+\alpha\varepsilon^{2}\Delta\boldsymbol{u}%
_{\varepsilon}+\mu_{\varepsilon}\nabla\varphi_{\varepsilon},
\]
which belongs to $L^{2}(0,T;\mathbb{H}^{-1}(\Omega_{\varepsilon}))$. Then for
a.e. $t\in(0,T)$, $\left\langle \boldsymbol{h}_{\varepsilon}(t),\boldsymbol{v}%
\right\rangle =0$ for all $\boldsymbol{v}\in\mathcal{C}_{0}^{\infty}%
(\Omega_{\varepsilon})^{d}$ with $\operatorname{div}\boldsymbol{v}=0$, where
$\left\langle ,\right\rangle $ stands for the duality pairings between
$\mathcal{D}^{\prime}(\Omega_{\varepsilon})^{d}$ and $\mathcal{D}%
(\Omega_{\varepsilon})^{d}$. Arguing as in the proof of Proposition
5 in Ref. \cite{Simon1999}, we derive the existence of a unique $p_{\varepsilon}\in
L^{2}(0,T;L^{2}(\Omega_{\varepsilon}))$ such that $\nabla p_{\varepsilon
}=\boldsymbol{h}_{\varepsilon}$ and $\int_{\Omega_{\varepsilon}}%
p_{\varepsilon}(t,x)dx=0$.
\end{proof}

\subsection{Uniform estimates}

We are now concerned with some uniform estimates that will be useful in the
sequel. Before we state them, we need the following result whose proof can be
found in Ref. \cite{Marusic2000}, see Lemmas 8, 10 and Remark 5.
\begin{lemma}
\label{le2.1}It holds that
\begin{equation}
\left\Vert u\right\Vert _{L^{2}(\Omega_{\varepsilon})}\leq C\varepsilon
\left\Vert \nabla u\right\Vert _{L^{2}(\Omega_{\varepsilon})^{d}%
}\ \ \ \ \ \ \ \ \ \ \ \ \ \ \ \ \ \ \ \ \ \ \ \ \label{e2.4}%
\end{equation}
and
\begin{equation}
\left\Vert u\right\Vert _{L^{4}(\Omega_{\varepsilon})}\leq C\varepsilon
^{\frac{1}{2}}\left\Vert \nabla u\right\Vert _{L^{2}(\Omega_{\varepsilon}%
)^{d}}\ \ \ \ \ \ \ \ \ \ \ \ \ \ \ \ \ \ \ \ \ \ \ \label{e2.5}%
\end{equation}
for any $u\in H_{0}^{1}(\Omega_{\varepsilon})$, where $C>0$ is independent of
$\varepsilon$.
\end{lemma}

In all what follows, the letter $C$ will denote a positive constant that may
vary from line to line. This being so, the following holds true.

\begin{proposition}
\label{pr2.1}Under the assumptions \emph{(\ref{1.2})}, \emph{(\ref{1.3})} and
\emph{(\ref{1.4})}, the weak solution $(\boldsymbol{u}_{\varepsilon}%
,\varphi_{\varepsilon},\mu_{\varepsilon})$ of \emph{(\ref{1.1})} in the sense
of Definition \emph{\ref{d1.1}} satisfies the following estimates
\begin{equation}
\left\Vert \boldsymbol{u}_{\varepsilon}\right\Vert _{L^{\infty}(0,T;L^{2}%
(\Omega_{\varepsilon})^{d})}\leq C\varepsilon^{\frac{1}{2}},\label{e2.7}%
\end{equation}%
\begin{equation}
\varepsilon\left\Vert \nabla\boldsymbol{u}_{\varepsilon}\right\Vert
_{L^{2}(Q_{\varepsilon})^{d\times d}}\leq C\varepsilon^{\frac{1}{2}%
},\label{e2.8}%
\end{equation}%
\begin{equation}
\left\Vert \varphi_{\varepsilon}\right\Vert _{L^{\infty}(0,T;H^{1}%
(\Omega_{\varepsilon}))}\leq C\varepsilon^{\frac{1}{2}},\label{e2.9}%
\end{equation}%
\begin{equation}
\left\Vert \mu_{\varepsilon}\right\Vert _{L^{2}(0,T;H^{1}(\Omega_{\varepsilon
}))}\leq C\varepsilon^{\frac{1}{2}},\label{e2.10}%
\end{equation}%
\begin{equation}
\left\Vert \frac{\partial\boldsymbol{u}_{\varepsilon}}{\partial t}\right\Vert
_{L^{2}(0,T;\mathbb{V}_{\varepsilon}^{\prime})}\leq C\varepsilon^{\frac{3}{2}%
},\label{e2.12}%
\end{equation}
and
\begin{equation}
\left\Vert f(\varphi_{\varepsilon})\right\Vert _{L^{\infty}(0,T;L^{1}%
(\Omega_{\varepsilon}))}\leq C\varepsilon,\label{e2.13}%
\end{equation}
where $C>0$ is a constant independent of $\varepsilon$.
\end{proposition}

\begin{proof}
We take the scalar product in $\mathbb{H}_{\varepsilon}$ of (\ref{1.1})$_{1} $
with $\boldsymbol{u}_{\varepsilon}$ and use the boundary condition
$\boldsymbol{u}_{\varepsilon}=0$ on $\partial\Omega_{\varepsilon}$ to obtain
\begin{equation}
\frac{1}{2}\frac{d}{dt}\int_{\Omega_{\varepsilon}}\left\vert \boldsymbol{u}%
_{\varepsilon}\right\vert ^{2}dx-\int_{\Omega_{\varepsilon}}\mu_{\varepsilon
}(\nabla\varphi_{\varepsilon}\cdot\boldsymbol{u}_{\varepsilon})dx+\alpha
\varepsilon^{2}\int_{\Omega_{\varepsilon}}\left\vert \nabla\boldsymbol{u}%
_{\varepsilon}\right\vert ^{2}dx=\int_{\Omega_{\varepsilon}}\boldsymbol{h}%
\cdot\boldsymbol{u}_{\varepsilon}dx.\label{e2.14}%
\end{equation}
Next, taking the inner product in $L^{2}(\Omega_{\varepsilon})$ of
(\ref{1.1})$_{3}$ with $\mu_{\varepsilon}$, and accounting of (\ref{1.1}%
)$_{4}$ together with (\ref{1.1})$_{5}$, one obtains
\begin{equation}
\frac{d}{dt}\left[  \frac{\beta}{2}\int_{\Omega_{\varepsilon}}\left\vert
\nabla\varphi_{\varepsilon}\right\vert ^{2}dx+\lambda\int_{\Omega
_{\varepsilon}}F(\varphi_{\varepsilon})dx\right]  +\int_{\Omega_{\varepsilon}%
}\left\vert \nabla\mu_{\varepsilon}\right\vert ^{2}dx+\int_{\Omega
_{\varepsilon}}\mu_{\varepsilon}\nabla\varphi_{\varepsilon}\cdot
\boldsymbol{u}_{\varepsilon}dx=0.\label{e2.15}%
\end{equation}
Let us notice the fact in getting (\ref{e2.15}) we have used the equations
$\operatorname{div}\boldsymbol{u}_{\varepsilon}=0$ and $\frac{\partial
\varphi_{\varepsilon}}{\partial\nu}=0$ together with the fact that $F^{\prime
}=f $, so that $\int_{\Omega_{\varepsilon}}\frac{d\varphi_{\varepsilon}}%
{dt}f(\varphi_{\varepsilon})dx=\frac{d}{dt}\int_{\Omega_{\varepsilon}%
}F(\varphi_{\varepsilon})dx$. Now summing up (\ref{e2.14}) and (\ref{e2.15})
gives
\begin{equation}%
\begin{array}
[c]{l}%
\frac{d}{dt}\left[  \frac{1}{2}\left\Vert \boldsymbol{u}_{\varepsilon
}(t)\right\Vert _{L^{2}}^{2}+\frac{\beta}{2}\left\Vert \nabla\varphi
_{\varepsilon}(t)\right\Vert _{L^{2}}^{2}+\lambda\int_{\Omega_{\varepsilon}%
}F(\varphi_{\varepsilon}(t))dx\right]  +\alpha\varepsilon^{2}\left\Vert
\nabla\boldsymbol{u}_{\varepsilon}(t)\right\Vert _{L^{2}}^{2}\\
\ \ \ \ +\left\Vert \nabla\mu_{\varepsilon}(t)\right\Vert _{L^{2}}^{2}%
=\int_{\Omega_{\varepsilon}}\boldsymbol{h}(t)\cdot\boldsymbol{u}_{\varepsilon
}(t)dx.
\end{array}
\label{e2.16}%
\end{equation}
Since $\boldsymbol{h}(t,x)=(\boldsymbol{h}_{1}(t,\overline{x}),0)$, we get
\begin{align*}
\left\vert \int_{\Omega_{\varepsilon}}\boldsymbol{h}(t)\cdot\boldsymbol{u}%
_{\varepsilon}(t)dx\right\vert  & \leq C\varepsilon^{\frac{1}{2}}\left\Vert
\boldsymbol{h}_{1}(t)\right\Vert _{L^{2}(\Omega)^{d-1}}\left\Vert
\boldsymbol{u}_{\varepsilon}(t)\right\Vert _{L^{2}(\Omega_{\varepsilon})^{d}%
}\\
& \leq C\varepsilon^{\frac{3}{2}}\left\Vert \boldsymbol{h}_{1}(t)\right\Vert
_{L^{2}(\Omega)^{d-1}}\left\Vert \nabla\boldsymbol{u}_{\varepsilon
}(t)\right\Vert _{L^{2}(\Omega_{\varepsilon})^{d\times d}}\text{ by
(\ref{e2.4})}\\
& \leq C\varepsilon\left\Vert \boldsymbol{h}_{1}(t)\right\Vert _{L^{2}%
(\Omega)^{d-1}}^{2}+\frac{\alpha}{2}\varepsilon^{2}\left\Vert \nabla
\boldsymbol{u}_{\varepsilon}(t)\right\Vert _{L^{2}(\Omega_{\varepsilon
})^{d\times d}}^{2}.
\end{align*}
Integrating (\ref{e2.16}) over $(0,t)$, we readily get
\begin{equation}%
\begin{array}
[c]{l}%
\dfrac{1}{2}\left\Vert \boldsymbol{u}_{\varepsilon}(t)\right\Vert _{L^{2}}%
^{2}+\frac{\alpha}{2}\varepsilon^{2}%
{\displaystyle\int_{0}^{t}}
\left\Vert \nabla\boldsymbol{u}_{\varepsilon}(s)\right\Vert _{L^{2}}%
^{2}ds+\dfrac{\beta}{2}\left\Vert \nabla\varphi_{\varepsilon}(t)\right\Vert
_{L^{2}}^{2}+\lambda%
{\displaystyle\int_{\Omega_{\varepsilon}}}
F(\varphi_{\varepsilon}(t))dx\\
\\
\ \ \ \ +%
{\displaystyle\int_{0}^{t}}
\left\Vert \nabla\mu_{\varepsilon}(s)\right\Vert _{L^{2}}^{2}ds\leq
C\varepsilon+\left\Vert \boldsymbol{u}_{0}^{\varepsilon}\right\Vert _{L^{2}%
}^{2}+\dfrac{\beta}{2}\left\Vert \nabla\varphi_{0}^{\varepsilon}\right\Vert
_{L^{2}}^{2}+\lambda%
{\displaystyle\int_{\Omega_{\varepsilon}}}
F(\varphi_{0}^{\varepsilon})dx.
\end{array}
\label{e2.17}%
\end{equation}
It follows therefore from (\ref{1.4}) that (\ref{e2.7}) and (\ref{e2.8}) hold
and further
\begin{equation}
\left\Vert \nabla\varphi_{\varepsilon}\right\Vert _{L^{\infty}(0,T;L^{2}%
(\Omega_{\varepsilon})^{d})}\leq C\varepsilon^{\frac{1}{2}},\label{e2.18}%
\end{equation}%
\begin{equation}
\left\Vert \nabla\mu_{\varepsilon}\right\Vert _{L^{2}(Q_{\varepsilon})^{d}%
}\leq C\varepsilon^{\frac{1}{2}}\ \ \ \ \ \ \ \ \ \ \ \ \ \label{e2.19}%
\end{equation}
and
\[
\left\Vert F(\varphi_{\varepsilon})\right\Vert _{L^{\infty}(0,T;L^{1}%
(\Omega_{\varepsilon}))}\leq C\varepsilon.
\]
This being so, the no-flux boundary condition $\frac{\partial\varphi
_{\varepsilon}}{\partial\nu}=\frac{\partial\mu_{\varepsilon}}{\partial\nu}=0$
on $\partial\Omega_{\varepsilon}$ ensures the mass conservation of the
following quantity
\[
\left\langle \varphi_{\varepsilon}(t)\right\rangle
=\mathchoice {{\setbox0=\hbox{$\displaystyle{\textstyle
-}{\int}$ } \vcenter{\hbox{$\textstyle -$
}}\kern-.6\wd0}}{{\setbox0=\hbox{$\textstyle{\scriptstyle -}{\int}$ } \vcenter{\hbox{$\scriptstyle -$
}}\kern-.6\wd0}}{{\setbox0=\hbox{$\scriptstyle{\scriptscriptstyle -}{\int}$
} \vcenter{\hbox{$\scriptscriptstyle -$
}}\kern-.6\wd0}}{{\setbox0=\hbox{$\scriptscriptstyle{\scriptscriptstyle
-}{\int}$ } \vcenter{\hbox{$\scriptscriptstyle -$ }}\kern-.6\wd0}}\!\int
_{\Omega_{\varepsilon}}\varphi_{\varepsilon}(t,x)dx,
\]
where $\mathchoice {{\setbox0=\hbox{$\displaystyle{\textstyle
-}{\int}$ } \vcenter{\hbox{$\textstyle -$
}}\kern-.6\wd0}}{{\setbox0=\hbox{$\textstyle{\scriptstyle -}{\int}$ } \vcenter{\hbox{$\scriptstyle -$
}}\kern-.6\wd0}}{{\setbox0=\hbox{$\scriptstyle{\scriptscriptstyle -}{\int}$
} \vcenter{\hbox{$\scriptscriptstyle -$
}}\kern-.6\wd0}}{{\setbox0=\hbox{$\scriptscriptstyle{\scriptscriptstyle
-}{\int}$ } \vcenter{\hbox{$\scriptscriptstyle -$ }}\kern-.6\wd0}}\!\int
_{\Omega_{\varepsilon}}=\left\vert \Omega_{\varepsilon}\right\vert ^{-1}%
\int_{\Omega_{\varepsilon}}$ and $\left\vert \Omega_{\varepsilon}\right\vert $
denotes the Lebesgue measure of $\Omega_{\varepsilon}$. This yields
\begin{equation}
\left\langle \varphi_{\varepsilon}(t)\right\rangle =\left\langle
\varphi_{\varepsilon}(0)\right\rangle \ \ \forall0<t\leq T.\label{e2.20}%
\end{equation}
Thus the Poincar\'{e}-Wirtinger inequality associated to (\ref{e2.20}) gives
\begin{align*}
\left\Vert \varphi_{\varepsilon}(t)\right\Vert _{L^{2}}  & \leq\left\Vert
\varphi_{\varepsilon}(t)-\left\langle \varphi_{\varepsilon}(t)\right\rangle
\right\Vert _{L^{2}}+\left\Vert \left\langle \varphi_{0}^{\varepsilon
}\right\rangle \right\Vert _{L^{2}}\\
& \leq C\left\Vert \nabla\varphi_{\varepsilon}(t)\right\Vert _{L^{2}%
}+\left\Vert \varphi_{0}^{\varepsilon}\right\Vert _{L^{2}}\\
& \leq C\varepsilon^{\frac{1}{2}},
\end{align*}
where the last inequality above is a consequence of (\ref{e2.18}) and
(\ref{1.4}). This, together with (\ref{e2.18}) gives (\ref{e2.9}).

Let us now prove (\ref{e2.10}) and (\ref{e2.13}). First of all, in view of
(\ref{1.7}) one has
\begin{equation}
\int_{\Omega_{\varepsilon}}\left\vert f(\varphi_{\varepsilon}(t))\right\vert
dx\leq C\int_{\Omega_{\varepsilon}}(1+\left\vert \varphi_{\varepsilon
}(t)\right\vert ^{3})dx,\label{e2.21}%
\end{equation}
so that, from the Sobolev embedding $H^{1}(\Omega_{\varepsilon}%
)\hookrightarrow L^{3}(\Omega_{\varepsilon})$,
\begin{align*}
\left\Vert \varphi_{\varepsilon}(t)\right\Vert _{L^{3}(\Omega_{\varepsilon})}
& \leq C\left\Vert \varphi_{\varepsilon}(t)\right\Vert _{H^{1}(\Omega
_{\varepsilon})}\text{ for a.e. }t\in(0,T)\\
& \leq C\varepsilon^{\frac{1}{2}}.
\end{align*}
We infer from (\ref{e2.21}) that
\begin{equation}
\int_{\Omega_{\varepsilon}}\left\vert f(\varphi_{\varepsilon}(t))\right\vert
dx\leq C(\varepsilon+\varepsilon^{\frac{3}{2}})\leq C\varepsilon.\label{e2.22}%
\end{equation}
Whence (\ref{e2.13}). Now, as for (\ref{e2.10}), we first observe that
$\left\langle -\Delta\varphi_{\varepsilon},1\right\rangle =0$, so that from
(\ref{e2.22}),
\begin{align*}
\left\vert \int_{\Omega_{\varepsilon}}\mu_{\varepsilon}dx\right\vert  &
=\left\vert (\mu_{\varepsilon},1)\right\vert =\left\vert (\lambda
f(\varphi_{\varepsilon}),1)\right\vert \leq\lambda\int_{\Omega_{\varepsilon}%
}\left\vert f(\varphi_{\varepsilon}(t))\right\vert dx\\
& \leq C\varepsilon,
\end{align*}
hence
\begin{equation}
\left\vert \mathchoice {{\setbox0=\hbox{$\displaystyle{\textstyle
-}{\int}$ } \vcenter{\hbox{$\textstyle -$
}}\kern-.6\wd0}}{{\setbox0=\hbox{$\textstyle{\scriptstyle -}{\int}$ } \vcenter{\hbox{$\scriptstyle -$
}}\kern-.6\wd0}}{{\setbox0=\hbox{$\scriptstyle{\scriptscriptstyle -}{\int}$
} \vcenter{\hbox{$\scriptscriptstyle -$
}}\kern-.6\wd0}}{{\setbox0=\hbox{$\scriptscriptstyle{\scriptscriptstyle
-}{\int}$ } \vcenter{\hbox{$\scriptscriptstyle -$ }}\kern-.6\wd0}}\!\int
_{\Omega_{\varepsilon}}\mu_{\varepsilon}(t)dx\right\vert \leq C.\label{e2.23}%
\end{equation}
Applying Poincar\'{e}-Wirtinger's inequality, we deduce from (\ref{e2.23})
that
\begin{align}
\left\Vert \mu_{\varepsilon}(t)\right\Vert _{L^{2}}^{2}  & \leq2\left(
\left\Vert \mu_{\varepsilon}%
-\mathchoice {{\setbox0=\hbox{$\displaystyle{\textstyle
-}{\int}$ } \vcenter{\hbox{$\textstyle -$
}}\kern-.6\wd0}}{{\setbox0=\hbox{$\textstyle{\scriptstyle -}{\int}$ } \vcenter{\hbox{$\scriptstyle -$
}}\kern-.6\wd0}}{{\setbox0=\hbox{$\scriptstyle{\scriptscriptstyle -}{\int}$
} \vcenter{\hbox{$\scriptscriptstyle -$
}}\kern-.6\wd0}}{{\setbox0=\hbox{$\scriptscriptstyle{\scriptscriptstyle
-}{\int}$ } \vcenter{\hbox{$\scriptscriptstyle -$ }}\kern-.6\wd0}}\!\int
_{\Omega_{\varepsilon}}\mu_{\varepsilon}(t)dx\right\Vert _{L^{2}}%
^{2}+\left\Vert \mathchoice {{\setbox0=\hbox{$\displaystyle{\textstyle
-}{\int}$ } \vcenter{\hbox{$\textstyle -$
}}\kern-.6\wd0}}{{\setbox0=\hbox{$\textstyle{\scriptstyle -}{\int}$ } \vcenter{\hbox{$\scriptstyle -$
}}\kern-.6\wd0}}{{\setbox0=\hbox{$\scriptstyle{\scriptscriptstyle -}{\int}$
} \vcenter{\hbox{$\scriptscriptstyle -$
}}\kern-.6\wd0}}{{\setbox0=\hbox{$\scriptscriptstyle{\scriptscriptstyle
-}{\int}$ } \vcenter{\hbox{$\scriptscriptstyle -$ }}\kern-.6\wd0}}\!\int
_{\Omega_{\varepsilon}}\mu_{\varepsilon}(t)dx\right\Vert _{L^{2}}^{2}\right)
\label{e2.24}\\
& \leq C\left(  \left\Vert \nabla\mu_{\varepsilon}\right\Vert _{L^{2}}%
^{2}+\left\vert \Omega_{\varepsilon}\right\vert \right)  .\nonumber
\end{align}
Therefore, integrating (\ref{e2.24}) over $(0,T)$ and owing to (\ref{e2.19}),
we are led to
\[
\left\Vert \mu_{\varepsilon}\right\Vert _{L^{2}(Q_{\varepsilon})}\leq
C\varepsilon^{\frac{1}{2}},
\]
which together with (\ref{e2.19}) gives (\ref{e2.10}).

Let us finally check (\ref{e2.12}). To that end, let $v\in\mathbb{V}%
_{\varepsilon}$; then
\begin{align*}
\left\vert \left\langle \frac{\partial\boldsymbol{u}_{\varepsilon}}{\partial
t}(t),v\right\rangle \right\vert  & \leq\alpha\varepsilon^{2}\left\Vert
\nabla\boldsymbol{u}_{\varepsilon}(t)\right\Vert _{L^{2}}\left\Vert \nabla
v\right\Vert _{L^{2}}+\left\Vert \mu_{\varepsilon}(t)\right\Vert _{L^{4}%
}\left\Vert \nabla\varphi_{\varepsilon}(t)\right\Vert _{L^{2}}\left\Vert
v\right\Vert _{L^{4}}+\left\Vert \boldsymbol{h}(t)\right\Vert _{L^{2}%
}\left\Vert v\right\Vert _{L^{2}}\\
& \leq\alpha\varepsilon^{2}\left\Vert \nabla\boldsymbol{u}_{\varepsilon
}(t)\right\Vert _{L^{2}}\left\Vert \nabla v\right\Vert _{L^{2}}+C\varepsilon
^{\frac{1}{2}}\left\Vert \mu_{\varepsilon}(t)\right\Vert _{H^{1}}\left\Vert
\nabla\varphi_{\varepsilon}(t)\right\Vert _{L^{2}}\left\Vert \nabla
v\right\Vert _{L^{2}}+C\varepsilon^{\frac{3}{2}}\left\Vert \nabla v\right\Vert
_{L^{2}},
\end{align*}
where for the last inequality above we have used the continuous embedding
$H^{1}(\Omega_{\varepsilon})\hookrightarrow L^{4}(\Omega_{\varepsilon})$ to
control $\left\Vert \mu_{\varepsilon}(t)\right\Vert _{L^{4}}$, and
(\ref{e2.4}) and (\ref{e2.5}). Thus
\[
\sup_{v\in\mathbb{V}_{\varepsilon},\left\Vert v\right\Vert _{\mathbb{V}%
_{\varepsilon}}\leq1}\left\vert \left\langle \frac{\partial\boldsymbol{u}%
_{\varepsilon}}{\partial t}(t),v\right\rangle \right\vert \leq\alpha
\varepsilon^{2}\left\Vert \nabla\boldsymbol{u}_{\varepsilon}(t)\right\Vert
_{L^{2}}+C\varepsilon\left\Vert \mu_{\varepsilon}(t)\right\Vert _{H^{1}%
}+C\varepsilon^{\frac{3}{2}}.
\]
Integrating the square of $\sup_{v\in\mathbb{V}_{\varepsilon},\left\Vert
v\right\Vert _{\mathbb{V}_{\varepsilon}}\leq1}\left\vert \left\langle
\frac{\partial\boldsymbol{u}_{\varepsilon}}{\partial t}(t),v\right\rangle
\right\vert $ over $(0,T)$ and using the estimates (\ref{e2.8}) and
(\ref{e2.10}), we readily get
\[
\left\Vert \frac{\partial\boldsymbol{u}_{\varepsilon}}{\partial t}\right\Vert
_{L^{2}(0,T;\mathbb{V}_{\varepsilon}^{\prime})}\leq C\varepsilon^{\frac{3}{2}%
}.
\]
This completes the proof.
\end{proof}

The following lemma (Lemma 20 in Ref. \cite{Marusic2000}) will be used in estimating the pressure.

\begin{lemma}
\label{le2.2}For any $f\in L_{0}^{2}%
(\Omega_{\varepsilon})$, there exists a function $\phi\in H_{0}^{1}%
(\Omega_{\varepsilon})^{d}$ such that $\operatorname{div}\phi=f$ in
$\Omega_{\varepsilon}$. Moreover it holds that
\[
\left\Vert \phi\right\Vert _{L^{2}(\Omega_{\varepsilon})^{d}}\leq C\left\Vert
f\right\Vert _{L^{2}(\Omega_{\varepsilon})}\text{ and }\left\Vert \nabla
\phi\right\Vert _{L^{2}(\Omega_{\varepsilon})^{d}}\leq\frac{C}{\varepsilon
}\left\Vert f\right\Vert _{L^{2}(\Omega_{\varepsilon})},
\]
where $C>0$ is independent of $\varepsilon$.
\end{lemma}

\begin{proposition}
\label{pr2.2}Let $p_{\varepsilon}\in L^{2}(0,T;L_{0}^{2}(\Omega_{\varepsilon
}))$ satisfying \emph{(\ref{1.1})}$_{1}$. Then we have
\begin{equation}
\left\Vert p_{\varepsilon}\right\Vert _{L^{2}(Q_{\varepsilon})}\leq
C\varepsilon^{\frac{1}{2}}%
,\ \ \ \ \ \ \ \ \ \ \ \ \ \ \ \ \ \ \ \ \ \ \ \ \ \ \ \ \ \ \ \ \label{e2.25}%
\end{equation}
where $C>0$ is independent of $\varepsilon$.
\end{proposition}

\begin{proof}
In view of Lemma \ref{le2.2}, let us introduce $\phi_{\varepsilon}\in
L^{2}(0,T;H_{0}^{1}(\Omega_{\varepsilon})^{d})$ solution of
$\operatorname{div}\phi_{\varepsilon}=p_{\varepsilon}$ in $Q_{\varepsilon}$
such that
\begin{equation}
\left\Vert \phi\right\Vert _{L^{2}(Q_{\varepsilon})^{d}}\leq C\left\Vert
p_{\varepsilon}\right\Vert _{L^{2}(Q_{\varepsilon})}\text{ and }\left\Vert
\nabla\phi_{\varepsilon}\right\Vert _{L^{2}(Q_{\varepsilon})^{d}}\leq\frac
{C}{\varepsilon}\left\Vert p_{\varepsilon}\right\Vert _{L^{2}(Q_{\varepsilon
})}.\label{e2.26}%
\end{equation}
Using $\phi_{\varepsilon}$ as test function in the variational form of the
first equation in (\ref{1.1}), we obtain
\begin{align*}
\left\Vert p_{\varepsilon}\right\Vert _{L^{2}(Q_{\varepsilon})}^{2}  &
=\left\vert \int_{Q_{\varepsilon}}p_{\varepsilon}\operatorname{div}%
\phi_{\varepsilon}dxdt\right\vert \\
& \leq\left\vert \left\langle \frac{\partial\boldsymbol{u}_{\varepsilon}%
}{\partial t},\phi_{\varepsilon}\right\rangle \right\vert +\alpha
\varepsilon^{2}\left\vert \int_{Q_{\varepsilon}}\nabla\boldsymbol{u}%
_{\varepsilon}\cdot\nabla\phi_{\varepsilon}dxdt\right\vert +\left\vert
\int_{Q_{\varepsilon}}\mu_{\varepsilon}\nabla\varphi_{\varepsilon}\cdot
\phi_{\varepsilon}dxdt\right\vert \\
& +\left\vert \int_{Q_{\varepsilon}}\boldsymbol{h}\cdot\phi_{\varepsilon
}dxdt\right\vert \\
& \leq\left\Vert \frac{\partial\boldsymbol{u}_{\varepsilon}}{\partial
t}\right\Vert _{L^{2}(0,T;\mathbb{V}_{\varepsilon}^{\prime})}\left\Vert
\phi_{\varepsilon}\right\Vert _{L^{2}(0,T;\mathbb{V}_{\varepsilon})}%
+\alpha\varepsilon^{2}\left\Vert \nabla\boldsymbol{u}_{\varepsilon}\right\Vert
_{L^{2}(Q_{\varepsilon})}\left\Vert \nabla\phi_{\varepsilon}\right\Vert
_{L^{2}(Q_{\varepsilon})}\\
& +\left\Vert \mu_{\varepsilon}\right\Vert _{L^{4}(Q_{\varepsilon})}\left\Vert
\nabla\varphi_{\varepsilon}\right\Vert _{L^{2}(Q_{\varepsilon})}\left\Vert
\phi_{\varepsilon}\right\Vert _{L^{4}(Q_{\varepsilon})}+\left\Vert
\boldsymbol{h}\right\Vert _{L^{2}(Q_{\varepsilon})}\left\Vert \phi
_{\varepsilon}\right\Vert _{L^{2}(Q_{\varepsilon})}.
\end{align*}
We take into account (\ref{e2.12}) and (\ref{e2.26}) by noticing that
$\left\Vert \phi_{\varepsilon}\right\Vert _{L^{2}(0,T;\mathbb{V}_{\varepsilon
})}=\left\Vert \nabla\phi_{\varepsilon}\right\Vert _{L^{2}(Q_{\varepsilon})}$,
to obtain
\[
\left\Vert \frac{\partial\boldsymbol{u}_{\varepsilon}}{\partial t}\right\Vert
_{L^{2}(0,T;\mathbb{V}_{\varepsilon}^{\prime})}\left\Vert \phi_{\varepsilon
}\right\Vert _{L^{2}(0,T;\mathbb{V}_{\varepsilon})}\leq C\varepsilon^{\frac
{1}{2}}\left\Vert p_{\varepsilon}\right\Vert _{L^{2}(Q_{\varepsilon})}.
\]
Next employing (\ref{e2.8}) and (\ref{e2.26}) yields
\[
\alpha\varepsilon^{2}\left\Vert \nabla\boldsymbol{u}_{\varepsilon}\right\Vert
_{L^{2}(Q_{\varepsilon})}\left\Vert \nabla\phi_{\varepsilon}\right\Vert
_{L^{2}(Q_{\varepsilon})}\leq C\varepsilon^{\frac{1}{2}}\left\Vert
p_{\varepsilon}\right\Vert _{L^{2}(Q_{\varepsilon})}.
\]
Similarly, from the definition of $\boldsymbol{h}$ and (\ref{e2.26}), we
deduce that
\[
\left\Vert \boldsymbol{h}\right\Vert _{L^{2}(Q_{\varepsilon})}\left\Vert
\phi_{\varepsilon}\right\Vert _{L^{2}(Q_{\varepsilon})}\leq C\varepsilon
^{\frac{1}{2}}\left\Vert p_{\varepsilon}\right\Vert _{L^{2}(Q_{\varepsilon})}.
\]
Finally we use (\ref{e2.5}) together with the continuous embedding
$H^{1}(\Omega_{\varepsilon})\hookrightarrow L^{4}(\Omega_{\varepsilon})$ and
inequalities (\ref{e2.9}), (\ref{e2.10}) and (\ref{e2.26}) to get
\[
\left\Vert \mu_{\varepsilon}\right\Vert _{L^{4}(Q_{\varepsilon})}\left\Vert
\nabla\varphi_{\varepsilon}\right\Vert _{L^{2}(Q_{\varepsilon})}\left\Vert
\phi_{\varepsilon}\right\Vert _{L^{4}(Q_{\varepsilon})}\leq C\varepsilon
^{\frac{1}{2}}\left\Vert p_{\varepsilon}\right\Vert _{L^{2}(Q_{\varepsilon})}.
\]
We conclude that
\[
\left\Vert p_{\varepsilon}\right\Vert _{L^{2}(Q_{\varepsilon})}\leq
C\varepsilon^{\frac{1}{2}},
\]
which amounts to (\ref{e2.25}).
\end{proof}

We close this section by a further estimate on the order parameter
$\varphi_{\varepsilon}$. To that end we define the partial integral
$M_{\varepsilon}\varphi_{\varepsilon}$ of $\varphi_{\varepsilon}$ as the
average in the thin direction as follows:%

\begin{equation}
M_{\varepsilon}\varphi_{\varepsilon}(t,\overline{x}%
)=\mathchoice {{\setbox0=\hbox{$\displaystyle{\textstyle -}{\int}$ } \vcenter{\hbox{$\textstyle -$
}}\kern-.6\wd0}}{{\setbox0=\hbox{$\textstyle{\scriptstyle -}{\int}$ } \vcenter{\hbox{$\scriptstyle -$
}}\kern-.6\wd0}}{{\setbox0=\hbox{$\scriptstyle{\scriptscriptstyle -}{\int}$
} \vcenter{\hbox{$\scriptscriptstyle -$
}}\kern-.6\wd0}}{{\setbox0=\hbox{$\scriptscriptstyle{\scriptscriptstyle
-}{\int}$ } \vcenter{\hbox{$\scriptscriptstyle -$ }}\kern-.6\wd0}}\!\int
_{\varepsilon I}\varphi_{\varepsilon}(t,\overline{x},\zeta)d\zeta
,\ \ (t,\overline{x})\in Q.\label{2.25}%
\end{equation}
Then it can be easily shown (using the Lebesgue theorem about differentiation
under the sign $\int$) that
\begin{equation}
M_{\varepsilon}\nabla_{\overline{x}}\phi=\nabla_{\overline{x}}M_{\varepsilon
}\phi\text{ for all }\phi\in H^{1}(\Omega).\label{2.26}%
\end{equation}
This being so, we have the following result.

\begin{proposition}
\label{p2.3}Let $M_{\varepsilon}\varphi_{\varepsilon}$ be defined by
\emph{(\ref{2.25})}. Then $M_{\varepsilon}\varphi_{\varepsilon}\in L^{\infty
}(0,T;H^{1}(\Omega))$ with $\partial M_{\varepsilon}\varphi_{\varepsilon
}/\partial t\in L^{2}(0,T;H^{1}(\Omega)^{\prime})$, and further it holds that
\begin{equation}
\sup_{\varepsilon>0}\left[  \left\Vert M_{\varepsilon}\varphi_{\varepsilon
}\right\Vert _{L^{\infty}(0,T;H^{1}(\Omega))}+\left\Vert \frac{\partial
M_{\varepsilon}\varphi_{\varepsilon}}{\partial t}\right\Vert _{L^{2}%
(0,T;H^{1}(\Omega)^{\prime})}\right]  \leq C,\label{2.27}%
\end{equation}
where $C>0$ is independent of $\varepsilon$.
\end{proposition}

\begin{proof}
We recall that (\ref{1.1})$_{3}$ (with the help of (\ref{1.1})$_{2}$) is
equivalent to
\begin{equation}
\frac{\partial\varphi_{\varepsilon}}{\partial t}+\operatorname{div}%
(\boldsymbol{u}_{\varepsilon}\varphi_{\varepsilon})-\Delta\mu_{\varepsilon
}=0\text{ in }Q_{\varepsilon}\text{.\ \ \ }\label{2.28}%
\end{equation}
With this in mind, we set, for any function $v$ defined in $Q_{\varepsilon}$,
$\widetilde{v}(t,\overline{x})=(M_{\varepsilon}v)(t,\overline{x})$
\ ($(t,\overline{x})\in Q$). Then we apply $M_{\varepsilon}$ on (\ref{2.28})
to get
\begin{equation}
\frac{\partial\widetilde{\varphi}_{\varepsilon}}{\partial t}%
+\operatorname{div}_{\overline{x}}(\widetilde{\boldsymbol{u}_{\varepsilon
}\varphi_{\varepsilon}})-\Delta_{\overline{x}}\widetilde{\mu}_{\varepsilon
}=0\text{ in }Q.\label{2.29}%
\end{equation}
Indeed, in order to obtain (\ref{2.29}), we observe that it is enough to check
that $\Delta_{\overline{x}}\widetilde{\mu}_{\varepsilon}=\widetilde
{\Delta_{\overline{x}}\mu_{\varepsilon}}$ in $Q$. To achieve this, let us
first observe that in view of the equality (\ref{2.26}), one has from
(\ref{2.28})
\[
\widetilde{\Delta\mu_{\varepsilon}}=\frac{\partial\widetilde{\varphi
}_{\varepsilon}}{\partial t}+\operatorname{div}_{\overline{x}}(\widetilde
{\boldsymbol{u}_{\varepsilon}\varphi_{\varepsilon}})\text{ in }%
Q\text{.\ \ \ \ \ \ }%
\]
Thus, for any $v\in\mathcal{C}_{0}^{\infty}(Q)$,
\begin{align*}
\left\langle \Delta_{\overline{x}}\widetilde{\mu}_{\varepsilon}%
,v\right\rangle  &  =-\int_{Q}\nabla_{\overline{x}}\widetilde{\mu
}_{\varepsilon}\cdot\nabla_{\overline{x}}vd\overline{x}dt=-\int_{Q}%
\widetilde{\nabla_{\overline{x}}\mu_{\varepsilon}}\cdot\nabla_{\overline{x}%
}vd\overline{x}dt\\
&  =-\int_{Q_{\varepsilon}}\nabla_{\overline{x}}\mu_{\varepsilon}\cdot
\nabla_{\overline{x}}vdxdt=-\int_{Q_{\varepsilon}}\varphi_{\varepsilon}%
\frac{\partial v}{\partial t}dxdt-\int_{Q_{\varepsilon}}\boldsymbol{u}%
_{\varepsilon}\varphi_{\varepsilon}\cdot\nabla_{\overline{x}}vdxdt\\
&  =-\int_{Q}\widetilde{\varphi}_{\varepsilon}\frac{\partial v}{\partial
t}d\overline{x}dt-\int_{Q}\widetilde{\boldsymbol{u}_{\varepsilon}%
\varphi_{\varepsilon}}\cdot\nabla_{\overline{x}}vd\overline{x}dt\\
&  =\left\langle \frac{\partial\widetilde{\varphi}_{\varepsilon}}{\partial
t}+\operatorname{div}_{\overline{x}}(\widetilde{\boldsymbol{u}_{\varepsilon
}\varphi_{\varepsilon}}),v\right\rangle =\left\langle \widetilde{\Delta
\mu_{\varepsilon}},v\right\rangle .
\end{align*}

Next we notice that from (\ref{e2.9}) and (\ref{e2.10}), one has
\begin{equation}
\left\Vert \widetilde{\varphi}_{\varepsilon}\right\Vert _{L^{\infty}%
(0,T;H^{1}(\Omega))}\leq C\text{ and }\left\Vert \widetilde{\mu}_{\varepsilon
}\right\Vert _{L^{2}(0,T;H^{1}(\Omega))}\leq C,\label{2.9'}%
\end{equation}
where $C>0$ is independent of $\varepsilon$.

Now, for $\phi\in H^{1}(\Omega)$, we have
\begin{align*}
\left\vert \left\langle \frac{\partial\widetilde{\varphi}_{\varepsilon}%
}{\partial t},\phi\right\rangle \right\vert  & \leq\left\vert \int_{\Omega
}\widetilde{\boldsymbol{u}_{\varepsilon}\varphi_{\varepsilon}}\cdot
\nabla_{\overline{x}}\phi d\overline{x}\right\vert +\left\vert \int_{\Omega
}\nabla_{\overline{x}}\widetilde{\mu}_{\varepsilon}\cdot\nabla_{\overline{x}%
}\phi d\overline{x}\right\vert \\
& \leq\frac{1}{2\varepsilon}\left\vert \int_{\Omega_{\varepsilon}%
}\boldsymbol{u}_{\varepsilon}\varphi_{\varepsilon}\cdot\nabla_{\overline{x}%
}\phi dx\right\vert +\left\Vert \nabla_{\overline{x}}\widetilde{\mu
}_{\varepsilon}\right\Vert _{L^{2}(\Omega)}\left\Vert \nabla_{\overline{x}%
}\phi\right\Vert _{L^{2}(\Omega)}\\
& \leq\frac{C}{\varepsilon}\left\Vert \boldsymbol{u}_{\varepsilon
}(t)\right\Vert _{L^{4}(\Omega_{\varepsilon})}\left\Vert \varphi_{\varepsilon
}(t)\right\Vert _{L^{4}(\Omega_{\varepsilon})}\left\Vert \nabla_{\overline{x}%
}\phi\right\Vert _{L^{2}(\Omega_{\varepsilon})}+\\
& +\left\Vert \nabla_{\overline{x}}\widetilde{\mu}_{\varepsilon}\right\Vert
_{L^{2}(\Omega)}\left\Vert \nabla_{\overline{x}}\phi\right\Vert _{L^{2}%
(\Omega)}.
\end{align*}
Since $\left\Vert \nabla_{\overline{x}}\phi\right\Vert _{L^{2}(\Omega
_{\varepsilon})}=\sqrt{2}\varepsilon^{\frac{1}{2}}\left\Vert \nabla
_{\overline{x}}\phi\right\Vert _{L^{2}(\Omega)}$ and by (\ref{e2.4}) in (Lemma
\ref{le2.1}) together with the embedding $H^{1}(\Omega_{\varepsilon
})\hookrightarrow L^{4}(\Omega_{\varepsilon})$ with the Sobolev constant being
independent of $\varepsilon$, we are led to
\begin{align*}
\left\vert \left\langle \frac{\partial\widetilde{\varphi}_{\varepsilon}%
}{\partial t},\phi\right\rangle \right\vert  & \leq\left(  C\left\Vert
\nabla\boldsymbol{u}_{\varepsilon}(t)\right\Vert _{L^{2}(\Omega_{\varepsilon
})}\left\Vert \varphi_{\varepsilon}(t)\right\Vert _{H^{1}(\Omega_{\varepsilon
})}+\left\Vert \nabla_{\overline{x}}\widetilde{\mu}_{\varepsilon}\right\Vert
_{L^{2}(\Omega)}\right)  \left\Vert \nabla_{\overline{x}}\phi\right\Vert
_{L^{2}(\Omega)}\\
& \leq\left(  C\varepsilon^{\frac{1}{2}}\left\Vert \nabla\boldsymbol{u}%
_{\varepsilon}(t)\right\Vert _{L^{2}(\Omega_{\varepsilon})}+\left\Vert
\nabla_{\overline{x}}\widetilde{\mu}_{\varepsilon}\right\Vert _{L^{2}(\Omega
)}\right)  \left\Vert \nabla_{\overline{x}}\phi\right\Vert _{L^{2}(\Omega)}.
\end{align*}
We conclude as for $\frac{\partial\boldsymbol{u}_{\varepsilon}}{\partial t} $
by integrating the square of $\sup_{\phi\in H^{1}(\Omega),\left\Vert
\phi\right\Vert _{H^{1}(\Omega)}\leq1}\left\vert \left\langle \frac
{\partial\widetilde{\varphi}_{\varepsilon}}{\partial t},\phi\right\rangle
\right\vert $ over $(0,T)$ and using the estimates (\ref{e2.8}) and
(\ref{2.9'}), we get
\[
\left\Vert \frac{\partial\widetilde{\varphi}_{\varepsilon}}{\partial
t}\right\Vert _{L^{2}(0,T;H^{1}(\Omega)^{\prime})}\leq C\text{,}%
\]
where $C$ is independent of $\varepsilon$. This concludes the proof.
\end{proof}

\section{Sigma-convergence for thin heterogeneous domains\label{sec3}}

In this section we gather for the reader some basic concepts about the
algebras with mean value\cite{Jikov,ZK} and the associated Sobolev-type
spaces\cite{CMP,NA}.

Let $A$ be an algebra with mean value on $\mathbb{R}^{m}$ (integer $m\geq1$)\cite{Jikov,ZK}, that is, a closed subalgebra of the $\mathcal{C}^{\ast}%
$-algebra of bounded uniformly continuous real-valued functions on
$\mathbb{R}^{m}$, $\mathrm{BUC}(\mathbb{R}^{m})$, which contains the
constants, is translation invariant and is such that any of its elements
possesses a mean value in the following sense: for every $u\in A$, the
sequence $(u^{\varepsilon})_{\varepsilon>0}$ ($u^{\varepsilon}%
(x)=u(x/\varepsilon)$) weakly $\ast$-converges in $L^{\infty}(\mathbb{R}^{m})$
to some real number $M(u)$ (called the mean value of $u$) as $\varepsilon
\rightarrow0$. The mean value expresses as
\begin{equation}
M(u)=\lim_{R\rightarrow\infty}%
\mathchoice {{\setbox0=\hbox{$\displaystyle{\textstyle -}{\int}$ } \vcenter{\hbox{$\textstyle -$
}}\kern-.6\wd0}}{{\setbox0=\hbox{$\textstyle{\scriptstyle -}{\int}$ } \vcenter{\hbox{$\scriptstyle -$
}}\kern-.6\wd0}}{{\setbox0=\hbox{$\scriptstyle{\scriptscriptstyle -}{\int}$
} \vcenter{\hbox{$\scriptscriptstyle -$
}}\kern-.6\wd0}}{{\setbox0=\hbox{$\scriptscriptstyle{\scriptscriptstyle
-}{\int}$ } \vcenter{\hbox{$\scriptscriptstyle -$ }}\kern-.6\wd0}}\!\int
_{B_{R}}u(y)dy\text{ for }u\in A\label{0.1}%
\end{equation}
where we have set $\mathchoice {{\setbox0=\hbox{$\displaystyle{\textstyle
-}{\int}$ } \vcenter{\hbox{$\textstyle -$
}}\kern-.6\wd0}}{{\setbox0=\hbox{$\textstyle{\scriptstyle -}{\int}$ } \vcenter{\hbox{$\scriptstyle -$
}}\kern-.6\wd0}}{{\setbox0=\hbox{$\scriptstyle{\scriptscriptstyle -}{\int}$
} \vcenter{\hbox{$\scriptscriptstyle -$
}}\kern-.6\wd0}}{{\setbox0=\hbox{$\scriptscriptstyle{\scriptscriptstyle
-}{\int}$ } \vcenter{\hbox{$\scriptscriptstyle -$ }}\kern-.6\wd0}}\!\int
_{B_{R}}=\left\vert B_{R}\right\vert ^{-1}\int_{B_{R}}$.

To an algebra with mean value $A$ are associated its regular subalgebras
$A^{k}=\{\psi\in\mathcal{C}^{k}(\mathbb{R}^{m}):$ $D_{y}^{\alpha}\psi\in A $
$\forall\alpha=(\alpha_{1},...,\alpha_{m})\in\mathbb{N}^{m}$ with $\left\vert
\alpha\right\vert \leq k\}$ ($k\geq0$ an integer with $A^{0}=A$, and
$D_{y}^{\alpha}\psi=\frac{\partial^{\left\vert \alpha\right\vert }\psi
}{\partial y_{1}^{\alpha_{1}}\cdot\cdot\cdot\partial y_{m}^{\alpha_{m}}}$).
Under the norm $\left\Vert \left\vert u\right\vert \right\Vert _{k}%
=\sup_{\left\vert \alpha\right\vert \leq k}\left\Vert D_{y}^{\alpha}%
\psi\right\Vert _{\infty}$, $A^{k}$ is a Banach space. We also define the
space $A^{\infty}=\{\psi\in\mathcal{C}^{\infty}(\mathbb{R}^{m}):$
$D_{y}^{\alpha}\psi\in A$ $\forall\alpha=(\alpha_{1},...,\alpha_{m}%
)\in\mathbb{N}^{m}\}$, a Fr\'{e}chet space when endowed with the locally
convex topology defined by the family of norms $\left\Vert \left\vert
\cdot\right\vert \right\Vert _{m}$. The space $A^{\infty}$ is dense in any
$A^{k} $ (integer $k\geq0$).

The notion of a vector-valued algebra with mean value will be very useful in
this study.

Let $\mathbb{F}$ be a Banach space. We denote by \textrm{BUC}$(\mathbb{R}%
^{m};\mathbb{F})$ the Banach space of bounded uniformly continuous functions
$u:\mathbb{R}^{m}\rightarrow\mathbb{F}$, endowed with the norm
\[
\left\Vert u\right\Vert _{\infty}=\sup_{y\in\mathbb{R}^{m}}\left\Vert
u(y)\right\Vert _{\mathbb{F}}%
\]
where $\left\Vert \cdot\right\Vert _{\mathbb{F}}$ stands for the norm in
$\mathbb{F}$. Let $A$ be an algebra with mean value on $\mathbb{R}^{m}$. We
denote by $A\otimes\mathbb{F}$ the usual space of functions of the form
\[
\sum_{\text{finite}}u_{i}\otimes e_{i}\text{ with }u_{i}\in A\text{ and }%
e_{i}\in\mathbb{F}%
\]
where $(u_{i}\otimes e_{i})(y)=u_{i}(y)e_{i}$ for $y\in\mathbb{R}^{m}$. With
this in mind, we define the vector-valued algebra with mean value
$A(\mathbb{R}^{m};\mathbb{F})$ as the closure of $A\otimes\mathbb{F}$ in
\textrm{BUC}$(\mathbb{R}^{m};\mathbb{F})$. Then it holds that (see
Ref. \cite{CMP}), for any $f\in A(\mathbb{R}^{m};\mathbb{F})$, the set
$\{L(f):L\in\mathbb{F}^{\prime}$ with $\left\Vert L\right\Vert _{\mathbb{F}%
^{\prime}}\leq1\}$ is relatively compact in $A$.

Let us note that we may still define the space $A(\mathbb{R}^{m};\mathbb{F})$
where $\mathbb{F}$ in this case is a Fr\'{e}chet space. In that case, we
replace the norm by the family of seminorms defining the topology of
$\mathbb{F}$.

Now, let $f\in A(\mathbb{R}^{m};\mathbb{F})$. Then, defining $\left\Vert
f\right\Vert _{\mathbb{F}}$ by $\left\Vert f\right\Vert _{\mathbb{F}%
}(y)=\left\Vert f(y)\right\Vert _{\mathbb{F}}$ ($y\in\mathbb{R}^{m}$), we have
that $\left\Vert f\right\Vert _{\mathbb{F}}\in A$. Similarly we can define
(for $0<p<\infty$) the function $\left\Vert f\right\Vert _{\mathbb{F}}^{p}$
and $\left\Vert f\right\Vert _{\mathbb{F}}^{p}\in A$. This allows us to define
the Besicovitch seminorm on $A(\mathbb{R}^{m};\mathbb{F})$ as follows: for
$1\leq p<\infty$, we define the Marcinkiewicz-type space $\mathfrak{M}%
^{p}(\mathbb{R}^{m};\mathbb{F})$ to be the vector space of functions $u\in
L_{loc}^{p}(\mathbb{R}^{m};\mathbb{F})$ such that
\[
\left\Vert u\right\Vert _{p}=\left(  \underset{R\rightarrow\infty}{\lim\sup
}\mathchoice {{\setbox0=\hbox{$\displaystyle{\textstyle
-}{\int}$ } \vcenter{\hbox{$\textstyle -$
}}\kern-.6\wd0}}{{\setbox0=\hbox{$\textstyle{\scriptstyle -}{\int}$ } \vcenter{\hbox{$\scriptstyle -$
}}\kern-.6\wd0}}{{\setbox0=\hbox{$\scriptstyle{\scriptscriptstyle -}{\int}$
} \vcenter{\hbox{$\scriptscriptstyle -$
}}\kern-.6\wd0}}{{\setbox0=\hbox{$\scriptscriptstyle{\scriptscriptstyle
-}{\int}$ } \vcenter{\hbox{$\scriptscriptstyle -$ }}\kern-.6\wd0}}\!\int
_{B_{R}}\left\Vert u(y)\right\Vert _{\mathbb{F}}^{p}dy\right)  ^{\frac{1}{p}%
}<\infty
\]
where $B_{R}$ is the open ball in $\mathbb{R}^{m}$ centered at the origin and
of radius $R$. Under the seminorm $\left\Vert \cdot\right\Vert _{p,\mathbb{F}%
}$, $\mathfrak{M}^{p}(\mathbb{R}^{m};\mathbb{F})$ is a complete seminormed
space with the property that $A(\mathbb{R}^{m};\mathbb{F})\subset
\mathfrak{M}^{p}(\mathbb{R}^{m};\mathbb{F})$ since $\left\Vert u\right\Vert
_{p}<\infty$ for any $u\in A(\mathbb{R}^{m};\mathbb{F})$. We therefore define
the generalized Besicovitch space $B_{A}^{p}(\mathbb{R}^{m};\mathbb{F})$ as
the closure of $A(\mathbb{R}^{m};\mathbb{F})$ in $\mathfrak{M}^{p}%
(\mathbb{R}^{m};\mathbb{F})$. The following hold true\cite{CMP,NA}:

\begin{itemize}
\item[(\textbf{i)}] The space $\mathcal{B}_{A}^{p}(\mathbb{R}^{m}%
;\mathbb{F})=B_{A}^{p}(\mathbb{R}^{m};\mathbb{F})/\mathcal{N}$ (where
$\mathcal{N}=\{u\in B_{A}^{p}(\mathbb{R}^{m};\mathbb{F}):\left\Vert
u\right\Vert _{p}=0\} $) is a Banach space under the norm $\left\Vert
u+\mathcal{N}\right\Vert _{p}=\left\Vert u\right\Vert _{p}$ for $u\in
B_{A}^{p}(\mathbb{R}^{m};\mathbb{F})$.

\item[(\textbf{ii)}] The mean value $M:A(\mathbb{R}^{m};\mathbb{F}%
)\rightarrow\mathbb{F}$ extends by continuity to a continuous linear mapping
(still denoted by $M$) on $B_{A}^{p}(\mathbb{R}^{m};\mathbb{F})$ satisfying
\[
L(M(u))=M(L(u))\text{ for all }L\in\mathbb{F}^{\prime}\text{ and }u\in
B_{A}^{p}(\mathbb{R}^{m};\mathbb{F}).
\]
Moreover, for $u\in B_{A}^{p}(\mathbb{R}^{m};\mathbb{F})$ we have
\[
\left\Vert u\right\Vert _{p}=\left[  M(\left\Vert u\right\Vert _{\mathbb{F}%
}^{p})\right]  ^{1/p}\equiv\left[  \lim_{R\rightarrow\infty}%
\mathchoice {{\setbox0=\hbox{$\displaystyle{\textstyle
-}{\int}$ } \vcenter{\hbox{$\textstyle -$
}}\kern-.6\wd0}}{{\setbox0=\hbox{$\textstyle{\scriptstyle -}{\int}$ } \vcenter{\hbox{$\scriptstyle -$
}}\kern-.6\wd0}}{{\setbox0=\hbox{$\scriptstyle{\scriptscriptstyle -}{\int}$
} \vcenter{\hbox{$\scriptscriptstyle -$
}}\kern-.6\wd0}}{{\setbox0=\hbox{$\scriptscriptstyle{\scriptscriptstyle
-}{\int}$ } \vcenter{\hbox{$\scriptscriptstyle -$ }}\kern-.6\wd0}}\!\int
_{B_{R}}\left\Vert u(y)\right\Vert _{\mathbb{F}}^{p}dy\right]  ^{\frac{1}{p}},
\]
and for $u\in\mathcal{N}$ one has $M(u)=0$.
\end{itemize}

It is worth noticing that $\mathcal{B}_{A}^{2}(\mathbb{R}^{m};H)$ (when
$\mathbb{F}=H$ is a Hilbert space) is a Hilbert space with inner product
\[
\left(  u,v\right)  _{2}=M\left[  \left(  u,v\right)  _{H}\right]  \text{ for
}u,v\in\mathcal{B}_{A}^{2}(\mathbb{R}^{m};H),
\]
$(~,~)_{H}$ denoting the inner product in $H$ and $\left(  u,v\right)  _{H}$
the function $y\mapsto\left(  u(y),v(y)\right)  _{H}$ from $\mathbb{R}^{m}$ to
$\mathbb{R}$, which belongs to $\mathcal{B}_{A}^{1}(\mathbb{R}^{m}%
;\mathbb{R})$.

We also define the Sobolev-Besicovitch type spaces as follows:
\[
B_{A}^{1,p}(\mathbb{R}^{m};\mathbb{F})=\{u\in B_{A}^{p}(\mathbb{R}%
^{m};\mathbb{F}):\nabla_{y}u\in(B_{A}^{p}(\mathbb{R}^{m};\mathbb{F}))^{m}\},
\]
endowed with the seminorm
\[
\left\Vert u\right\Vert _{1,p}=\left(  \left\Vert u\right\Vert _{p}%
^{p}+\left\Vert \nabla_{y}u\right\Vert _{p}^{p}\right)  ^{\frac{1}{p}},
\]
which is a complete seminormed space. The Banach counterpart of $B_{A}%
^{1,p}(\mathbb{R}^{m};\mathbb{F})$ denoted by $\mathcal{B}_{A}^{1,p}%
(\mathbb{R}^{m};\mathbb{F})$ is defined by replacing $B_{A}^{p}(\mathbb{R}%
^{m};\mathbb{F})$ by $\mathcal{B}_{A}^{p}(\mathbb{R}^{m};\mathbb{F})$ and
$\partial/\partial y_{i}$ by $\overline{\partial}/\partial y_{i}$, where
$\overline{\partial}/\partial y_{i}$ is defined by
\begin{equation}
\frac{\overline{\partial}}{\partial y_{i}}(u+\mathcal{N}):=\frac{\partial
u}{\partial y_{i}}+\mathcal{N}\text{ for }u\in B_{A}^{1,p}(\mathbb{R}%
^{m};\mathbb{F}).\label{e3}%
\end{equation}
It is important to note that $\overline{\partial}/\partial y_{i}$ is also
defined as the infinitesimal generator in the $i$th direction coordinate of
the strongly continuous group $\mathcal{T}(y):\mathcal{B}_{A}^{p}%
(\mathbb{R}^{m};\mathbb{F})\rightarrow\mathcal{B}_{A}^{p}(\mathbb{R}%
^{m};\mathbb{F});\ \mathcal{T}(y)(u+\mathcal{N})=u(\cdot+y)+\mathcal{N}$. Let
us denote by $\varrho:B_{A}^{p}(\mathbb{R}^{m};\mathbb{F})\rightarrow
\mathcal{B}_{A}^{p}(\mathbb{R}^{m};\mathbb{F})=B_{A}^{p}(\mathbb{R}%
^{m};\mathbb{F})/\mathcal{N}$, $\varrho(u)=u+\mathcal{N}$, the canonical
surjection. We remark that if $u\in B_{A}^{1,p}(\mathbb{R}^{m};\mathbb{F})$
then $\varrho(u)\in\mathcal{B}_{A}^{1,p}(\mathbb{R}^{m};\mathbb{F})$ with
further
\[
\frac{\overline{\partial}\varrho(u)}{\partial y_{i}}=\varrho\left(
\frac{\partial u}{\partial y_{i}}\right)  ,
\]
as seen above in (\ref{e3}).

We define a further notion by restricting ourselves to the case $\mathbb{F}%
=\mathbb{R}$. We say that the algebra $A$ is ergodic if any $u\in
\mathcal{B}_{A}^{1}(\mathbb{R}^{m};\mathbb{R})$ that is invariant under
$(\mathcal{T}(y))_{y\in\mathbb{R}^{m}}$ is a constant in $\mathcal{B}_{A}%
^{1}(\mathbb{R}^{m};\mathbb{R})$: this amounts to, if $\mathcal{T}(y)u=u$\ in
$\mathcal{B}_{A}^{1}(\mathbb{R}^{m};\mathbb{R})$ for every $y\in\mathbb{R}%
^{m}$, then $u=c$\ in $\mathcal{B}_{A}^{1}(\mathbb{R}^{m};\mathbb{R})$ in the
sense that $\left\Vert u-c\right\Vert _{1}=0$, $c$ being a constant.

The following \textit{corrector function}\ space will be useful in the sequel.
Let $G$ be an open bounded subset in $\mathbb{R}^{N}$. We define the
\textit{corrector function}\ space $B_{\#A}^{1,p}(\mathbb{R}^{m};W^{1,p}(G))$
by
\[%
\begin{array}
[c]{l}%
B_{\#A}^{1,p}(\mathbb{R}^{m};W^{1,p}(G))=\{u\in W_{loc}^{1,p}(\mathbb{R}%
^{m};W^{1,p}(G)):\nabla u\in B_{A}^{p}(\mathbb{R}^{m};L^{p}(G))^{m+N}\\
\text{
\ \ \ \ \ \ \ \ \ \ \ \ \ \ \ \ \ \ \ \ \ \ \ \ \ \ \ \ \ \ \ \ \ \ \ \ \ \ \ and
}\int_{G}M(\nabla u(\cdot,\zeta))d\zeta=0\}\text{,}%
\end{array}
\]
where in this case $\nabla=(\nabla_{y},\nabla_{\zeta})$, $\nabla_{y}$ (resp.
$\nabla_{\zeta}$) being the gradient operator with respect to the variable
$y\in\mathbb{R}^{m}$ (resp. $\zeta\in\mathbb{R}^{N}$). We identify two
elements of $B_{\#A}^{1,p}(\mathbb{R}^{m};W^{1,p}(G))$ by their gradients in
the sense that: $u=v$ in $B_{\#A}^{1,p}(\mathbb{R}^{m};W^{1,p}(G))$ iff
$\nabla(u-v)=0$, i.e. $\int_{G}\left\Vert \nabla(u(\cdot,\zeta)-v(\cdot
,\zeta))\right\Vert _{p}^{p}d\zeta=0$. The space $B_{\#A}^{1,p}(\mathbb{R}%
^{m};W^{1,p}(G))$ is therefore a Banach space under the norm $\left\Vert
u\right\Vert _{\#,p}=\left(  \int_{G}\left\Vert \nabla u(\cdot,\zeta
)\right\Vert _{p}^{p}d\zeta\right)  ^{1/p}$.

The sigma-convergence concept has been introduced in Ref. \cite{Hom1} in order to
tackle multiscale phenomena occurring in deterministic media. It is concerned
with multiscale phenomena taking place in all space dimensions. Its periodic
counterpart has then been generalized in Ref. \cite{RJ2007} to thin heterogeneous
media with periodic microstructures.

We provide here a suitable generalization of the definition contained in
Ref. \cite{RJ2007} to media displaying nonperiodic (but deterministic) structure.
Let us note that this generalization has already just been proposed for steady
state heterogeneous structures by the second and third authors in
Ref. \cite{JW2022}.

Our aim in this section is to provide, in the light of Ref.  \cite{JW2022}, a
systematic study of the concept of sigma-convergence applied to thin
heterogeneous domains whose heterogeneous structure is of general
deterministic type including the periodic one and the almost periodic one as
special cases. The compactness results obtained here generalize therefore
those in Ref. \cite{RJ2007} which are concerned only with periodic structures.

More precisely, let $d\geq2$ be a given integer, and let $\Omega
\subset\mathbb{R}^{d-1}$ be an open set, which will be assumed throughout this
section to be not necessarily bounded. For $\varepsilon>0$ a given small
parameter, we define the thin domain by $\Omega_{\varepsilon}=\Omega
\times(-\varepsilon,\varepsilon)$. When $\varepsilon\rightarrow0$,
$\Omega_{\varepsilon}$ shrinks to the "interface" $\Omega_{0}=\Omega
\times\{0\}$.

The space $\mathbb{R}_{\xi}^{m}$ is the numerical space $\mathbb{R}^{m}$ of
generic variable $\xi$. In this regard we set $\mathbb{R}^{d-1}=\mathbb{R}%
_{\overline{x}}^{d-1}$ or $\mathbb{R}_{\overline{y}}^{d-1}$ where
$\overline{x}=(x_{1},...,x_{d-1})$, so that $x\in\mathbb{R}^{d}$ writes
$(\overline{x},x_{d})$ or $(\overline{x},\zeta)$. We identify $\Omega_{0}$
with $\Omega$ so that the generic element in $\Omega_{0}$ is also denoted by
$\overline{x} $ instead of $(\overline{x},0)$.

To our spatial thin domain we associate the spatiotemporal domain
$Q_{\varepsilon}=(0,T)\times\Omega_{\varepsilon}$. Finally we set
$Q=(0,T)\times\Omega_{0}\equiv(0,T)\times\Omega$ and $I=(-1,1)$.

With this in mind, let $A$ be an algebra with mean value on $\mathbb{R}^{d-1}
$. We denote by $M$ the mean value on $A$ as well as its extension on the
associated generalized Besicovitch spaces $B_{A}^{p}(\mathbb{R}^{d-1}%
;L^{p}(I))$ and $\mathcal{B}_{A}^{p}(\mathbb{R}^{d-1};L^{p}(I))$, $1\leq
p<\infty$.

We introduce here below the notion of $\Sigma$-convergence for thin
heterogeneous domains; see Ref. \cite{JW2022} for the stationary version.

\begin{definition}
\label{d2.1}\emph{A sequence }$(u_{\varepsilon})_{\varepsilon>0}\subset
L^{p}(Q_{\varepsilon})$\emph{\ is said to }

\begin{itemize}
\item[(i)] \emph{weakly }$\Sigma$\emph{-converge in }$L^{p}(Q_{\varepsilon}%
)$\emph{\ to }$u_{0}\in L^{p}(Q;\mathcal{B}_{A}^{p}(\mathbb{R}^{d-1}%
;L^{p}(I)))$\emph{\ if as }$\varepsilon\rightarrow0$\emph{, }%
\[
\frac{1}{\varepsilon}\int_{Q_{\varepsilon}}u_{\varepsilon}(t,x)f\left(
t,\overline{x},\frac{x}{\varepsilon}\right)  dxdt\rightarrow\int_{Q}\int
_{I}M(u_{0}(t,\overline{x},\cdot,y_{d})f(t,\overline{x},\cdot,y_{d}%
))dy_{d}d\overline{x}dt
\]
\emph{for any }$f\in L^{p^{\prime}}(Q;A(\mathbb{R}^{d-1};L^{p^{\prime}}%
(I)))$\emph{\ (}$1/p^{\prime}=1-1/p$\emph{); we denote this by "}%
$u_{\varepsilon}\rightarrow u_{0}$\emph{\ in }$L^{p}(Q_{\varepsilon}%
)$\emph{-weak }$\Sigma_{A}$\emph{";}

\item[(ii)] \emph{strongly }$\Sigma$\emph{-converge in }$L^{p}(Q_{\varepsilon
})$\emph{\ to }$u_{0}\in L^{p}(Q;\mathcal{B}_{A}^{p}(\mathbb{R}^{d-1}%
;L^{p}(I)))$\emph{\ if it is weakly sigma-convergent and further }%
\begin{equation}
\varepsilon^{-\frac{1}{p}}\left\Vert u_{\varepsilon}\right\Vert _{L^{p}%
(Q_{\varepsilon})}\rightarrow\left\Vert u_{0}\right\Vert _{L^{p}%
(Q;\mathcal{B}_{A}^{p}(\mathbb{R}^{d-1};L^{p}(I)))};\label{2.2}%
\end{equation}
\emph{we denote this by "}$u_{\varepsilon}\rightarrow u_{0}$\emph{\ in }%
$L^{p}(Q_{\varepsilon})$\emph{-strong }$\Sigma_{A}$\emph{".}
\end{itemize}
\end{definition}

\begin{remark}
\label{r2.1}\emph{It is easy to see that if }$u_{0}\in L^{p}(Q;A(\mathbb{R}%
^{d-1};L^{p}(I)))$\emph{\ then (\ref{2.2}) is equivalent to }%
\begin{equation}
\varepsilon^{-\frac{1}{p}}\left\Vert u_{\varepsilon}-u_{0}^{\varepsilon
}\right\Vert _{L^{p}(Q_{\varepsilon})}\rightarrow0\text{\emph{\ as }%
}\varepsilon\rightarrow0,\label{2.3}%
\end{equation}
\emph{where }$u_{0}^{\varepsilon}(t,x)=u_{0}(t,\overline{x},x/\varepsilon
)$\emph{\ for }$(t,x)\in Q_{\varepsilon}$\emph{.}
\end{remark}

Before we state the first compactness result for this section, we need a
further notation. Throughout the work, the letter $E$ will stand for any
ordinary sequence $(\varepsilon_{n})_{n\geq1}$ with $0<\varepsilon_{n}\leq1$
and $\varepsilon_{n}\rightarrow0$ when $n\rightarrow\infty$. The generic term
of $E$ will be merely denote by $\varepsilon$ and $\varepsilon\rightarrow0$
will mean $\varepsilon_{n}\rightarrow0$ as $n\rightarrow\infty$. This being
so, the following compactness result holds true.

\begin{theorem}
\label{t2.1}Let $(u_{\varepsilon})_{\varepsilon\in E}$ be a sequence in
$L^{p}(Q_{\varepsilon})$ $(1<p<\infty)$ such that
\[
\sup_{\varepsilon\in E}\varepsilon^{-1/p}\left\Vert u_{\varepsilon}\right\Vert
_{L^{p}(Q_{\varepsilon})}\leq C
\]
where $C$ is a positive constant independent of $\varepsilon$. Then there
exists a subsequence $E^{\prime}$ of $E$ such that the sequence
$(u_{\varepsilon})_{\varepsilon\in E^{\prime}}$ weakly $\Sigma$-converges in
$L^{p}(Q_{\varepsilon})$ to some $u_{0}\in L^{p}(Q;\mathcal{B}_{A}%
^{p}(\mathbb{R}^{d-1};L^{p}(I)))$.
\end{theorem}

The proof of the above theorem follows the same way of proceeding as the one
in Ref. \cite{JW2022}.

\begin{remark}
\label{r2.2}\emph{Theorem \ref{t2.1} generalizes its periodic counterpart in
Ref. \cite{RJ2007}; see for instance Proposition 4.2 in Ref. \cite{RJ2007} that
corresponds to the special case }$A=\mathcal{C}_{per}(Y)$\emph{\ (with
}$Y=(0,1)^{d-1}$\emph{) of our result here.}
\end{remark}

We denote by $\varrho:B_{A}^{p}(\mathbb{R}^{d-1};\mathbb{F})\rightarrow
\mathcal{B}_{A}^{p}(\mathbb{R}^{d-1};\mathbb{F})$ the canonical mapping
defined by $\varrho(u)=u+\mathcal{N}$, $\mathcal{N}=\{u\in B_{A}%
^{p}(\mathbb{R}^{d-1};\mathbb{F}):\left\Vert u\right\Vert _{p}=0\}$, where
$\left\Vert u\right\Vert _{p}=\left[  M(\left\Vert u\right\Vert _{\mathbb{F}%
}^{p})\right]  ^{1/p}$ for $1\leq p<\infty$. We set $\mathcal{D}%
_{A}(\mathbb{R}^{d-1};\mathbb{F})=\varrho(A^{\infty}(\mathbb{R}^{d-1}%
;\mathbb{F}))$ where $A$ is an algebra with mean value on $\mathbb{R}^{d-1}$.
For function $\mathbf{g}=(g_{i})_{1\leq i\leq d}\in\lbrack\mathcal{B}_{A}%
^{p}(\mathbb{R}^{d-1};L^{p}(I))]^{d}$ we define the divergence $\overline
{\operatorname{div}}_{y}\mathbf{g}$ by
\[
\overline{\operatorname{div}}_{y}\mathbf{g}:=\sum_{i=1}^{d-1}\frac
{\overline{\partial}g_{i}}{\partial y_{i}}+\frac{\partial g_{d}}{\partial
y_{d}},
\]
that is, for any $\Phi=(\phi_{i})_{1\leq i\leq d}\in\lbrack\mathcal{B}%
_{A}^{1,p^{\prime}}(\mathbb{R}^{d-1};W^{1,p^{\prime}}(I))]^{d}$,
\[
\left\langle \overline{\operatorname{div}}_{y}\mathbf{g,\Phi}\right\rangle
=-\sum_{i=1}^{d-1}\int_{I}M(g_{i}(\cdot,y_{d})\frac{\overline{\partial}%
\phi_{i}}{\partial y_{i}}(\cdot,y_{d}))dy_{d}-\int_{I}M(g_{d}(\cdot
,y_{d})\frac{\partial\phi_{d}}{\partial y_{d}}(\cdot,y_{d}))dy_{d}.
\]

The following result arising from Ref. \cite{JW2022} (Corollary 3.1) is of interest in the
forthcoming compactness result.

\begin{lemma}
\label{c2.1}Let $1<p<\infty$ and let
$\mathbf{f}\in\lbrack\mathcal{B}_{A}^{p}(\mathbb{R}^{d-1};L^{p}(I))]^{d}$ be
such that
\[
\int_{I}M(\mathbf{f}(\cdot,y_{d})\cdot\mathbf{g}(\cdot,y_{d}))dy_{d}=0\text{
for all }\mathbf{g}\in\mathcal{V}_{\operatorname{div}},
\]
where $M$ stands for the mean value operator defined on $\mathcal{B}_{A}%
^{p}(\mathbb{R}^{d-1};L^{p}(I))$ and
\[
\mathcal{V}_{\operatorname{div}}=\{\Phi\in\lbrack\mathcal{D}_{A}%
(\mathbb{R}^{d-1};\mathcal{C}_{0}^{\infty}(I))]^{d}:\overline
{\operatorname{div}}_{y}\Phi=0\}.
\]
Then there exists a function $u\in B_{\#A}^{1,p}(\mathbb{R}^{d-1};W^{1,p}(I))
$, uniquely determined modulo constants, such that $\mathbf{f}=\nabla_{y}u$.
\end{lemma}

We are now able to state and prove the next compactness result dealing with
the convergence of the gradient.

\begin{theorem}
\label{t2.2}Assume that the algebra with mean value $A$ on $\mathbb{R}^{d-1}$
is ergodic. Let $(u_{\varepsilon})_{\varepsilon\in E}$ be a sequence in
$L^{p}(0,T;W^{1,p}(\Omega_{\varepsilon}))$ ($1<p<\infty$) such that
\begin{equation}
\sup_{\varepsilon\in E}\left(  \varepsilon^{-1/p}\left\Vert u_{\varepsilon
}\right\Vert _{L^{p}(Q_{\varepsilon})}+\varepsilon^{-1/p}\left\Vert \nabla
u_{\varepsilon}\right\Vert _{L^{p}(Q_{\varepsilon})}\right)  \leq C\label{2.7}%
\end{equation}
where $C>0$ is independent of $\varepsilon$. Then there exist a subsequence
$E^{\prime}$ of $E$ and a couple $(u_{0},u_{1})$ with $u_{0}\in L^{p}%
(0,T;W^{1,p}(\Omega_{0}))$ and $u_{1}\in L^{p}(Q;B_{\#A}^{1,p}(\mathbb{R}%
^{d-1};W^{1,p}(I)))$ such that, as $E^{\prime}\ni\varepsilon\rightarrow0$,
\begin{equation}
u_{\varepsilon}\rightarrow u_{0}\text{ in }L^{p}(Q_{\varepsilon})\text{-weak
}\Sigma_{A},\label{2.8}%
\end{equation}%
\begin{equation}
\frac{\partial u_{\varepsilon}}{\partial x_{i}}\rightarrow\frac{\partial
u_{0}}{\partial x_{i}}+\frac{\partial u_{1}}{\partial y_{i}}\text{ in }%
L^{p}(Q_{\varepsilon})\text{-weak }\Sigma_{A},\ 1\leq i\leq d-1,\label{2.9}%
\end{equation}
and
\begin{equation}
\frac{\partial u_{\varepsilon}}{\partial x_{d}}\rightarrow\frac{\partial
u_{1}}{\partial y_{d}}\text{ in }L^{p}(Q_{\varepsilon})\text{-weak }\Sigma
_{A}.\label{2.10}%
\end{equation}

\end{theorem}

\begin{remark}
\label{r2.3}\emph{If we set }%
\[
\nabla_{\overline{x}}u_{0}=\left(  \frac{\partial u_{0}}{\partial x_{1}%
},...,\frac{\partial u_{0}}{\partial x_{d-1}},0\right)  ,
\]
\emph{then (\ref{2.9}) and (\ref{2.10}) are equivalent to }%
\[
\nabla u_{\varepsilon}\rightarrow\nabla_{\overline{x}}u_{0}+\nabla_{y}%
u_{1}\text{\emph{\ in }}L^{p}(Q_{\varepsilon})^{d}\text{\emph{-weak }}%
\Sigma_{A}.
\]

\end{remark}

\begin{proof}
[Proof of Theorem \ref{t2.2}] In view of the assumption (\ref{2.7}), we appeal
to Theorem \ref{t2.1} to derive the existence of a subsequence $E^{\prime}$ of
$E$ and $u_{0}\in L^{p}(Q;\mathcal{B}_{A}^{p}(\mathbb{R}^{d-1};W^{1,p}(I)))$
and $\mathbf{v}\in\lbrack L^{p}(Q;\mathcal{B}_{A}^{p}(\mathbb{R}^{d-1}%
;W^{1,p}(I)))]^{d}$ such that
\begin{equation}
u_{\varepsilon}\rightarrow u_{0}\text{ in }L^{p}(Q_{\varepsilon})\text{-weak
}\Sigma_{A},\label{2.11}%
\end{equation}%
\begin{equation}
\frac{\partial u_{\varepsilon}}{\partial x_{i}}\rightarrow v_{i}\text{ in
}L^{p}(Q_{\varepsilon})\text{-weak }\Sigma_{A},\ 1\leq i\leq d-1,\label{2.12}%
\end{equation}
and
\begin{equation}
\frac{\partial u_{\varepsilon}}{\partial x_{d}}\rightarrow v_{d}\text{ in
}L^{p}(Q_{\varepsilon})\text{-weak }\Sigma_{A},\ \label{2.13}%
\end{equation}
where we have set $x=(\overline{x},x_{d})$ with $\overline{x}=(x_{i})_{1\leq
i\leq d-1}$ and thus $\nabla=(\nabla_{\overline{x}},\frac{\partial}{\partial
x_{d}})$. Let us first show that $u_{0}$ does not depend on $(\overline
{y},y_{d})=y$. To that end, let $\Phi\in(\mathcal{C}_{0}^{\infty}(Q)\otimes
A^{\infty}\otimes\mathcal{C}_{0}^{\infty}(I))^{d}$. One has
\begin{align*}
& \varepsilon^{-1}\int_{Q_{\varepsilon}}\varepsilon\nabla u_{\varepsilon
}(t,x)\cdot\Phi\left(  t,\overline{x},\frac{x}{\varepsilon}\right)  dxdt\\
& =-\int_{Q_{\varepsilon}}\varepsilon^{-1}u_{\varepsilon}(t,x)\left[
\varepsilon(\operatorname{div}_{\overline{x}}\Phi)\left(  t,\overline{x}%
,\frac{x}{\varepsilon}\right)  +(\operatorname{div}_{y}\Phi)\left(
t,\overline{x},\frac{x}{\varepsilon}\right)  \right]  dxdt.
\end{align*}
Letting $E^{\prime}\ni\varepsilon\rightarrow0$ and using (\ref{2.11}%
)-(\ref{2.12}), we get
\[
\int_{Q}\int_{I}M(u_{0}(t,\overline{x},\cdot,y_{d})\operatorname{div}_{y}%
\Phi(t,\overline{x},\cdot,y_{d}))dy_{d}d\overline{x}dt=0.
\]
This shows that $\overline{\nabla}_{y}u_{0}(t,\overline{x},\cdot)=0$ for a.e.
$(t,\overline{x})$, which amounts to $u_{0}(t,\overline{x},\overline{y}%
,\cdot)$ is independent of $y_{d}$, and $u_{0}(t,\overline{x},\cdot,y_{d})$ is
an invariant function. Since the algebra $A$ is ergodic, $u_{0}(t,\overline
{x},\cdot)$ does not depend on $y$, that is $u_{0}(t,\overline{x},\cdot
)=u_{0}(t,\overline{x})$.

Next let $\Phi_{\varepsilon}(t,x)=\varphi(t,\overline{x})\Psi(x/\varepsilon)$
($(t,x)\in Q_{\varepsilon}$) with $\varphi\in\mathcal{C}_{0}^{\infty}(Q)$ and
$\Psi\in(A^{\infty}(\mathbb{R}^{d-1};\mathcal{C}_{0}^{\infty}(I))^{d}$ with
$\operatorname{div}_{y}\Psi=0$. We set $\Psi=(\Psi_{\overline{x}},\psi_{d})$
with $\Psi_{\overline{x}}=(\psi_{j})_{1\leq j\leq d-1}$. We clearly have
\begin{align}
& \int_{Q_{\varepsilon}}\varepsilon^{-1}\left(  \nabla_{\overline{x}%
}u_{\varepsilon}(t,x)\cdot\Psi_{\overline{x}}\left(  \frac{x}{\varepsilon
}\right)  +\frac{\partial u_{\varepsilon}}{\partial x_{d}}(t,x)\psi_{d}\left(
\frac{x}{\varepsilon}\right)  \right)  \varphi(t,\overline{x})dxdt\label{2.14}%
\\
& =-\int_{Q_{\varepsilon}}\varepsilon^{-1}u_{\varepsilon}(t,x)\Psi
_{\overline{x}}\left(  \frac{x}{\varepsilon}\right)  \cdot\nabla_{\overline
{x}}\varphi(t,\overline{x})dxdt.\nonumber
\end{align}
Indeed
\begin{align*}
\varepsilon^{-1}\int_{Q_{\varepsilon}}\nabla u_{\varepsilon}\cdot
\Phi_{\varepsilon}dxdt  & =-\varepsilon^{-1}\int_{Q_{\varepsilon}%
}u_{\varepsilon}(t,x)\operatorname{div}\left(  \varphi(t,\overline{x}%
)\Psi(\frac{x}{\varepsilon})\right)  dxdt\\
& =-\varepsilon^{-1}\int_{Q_{\varepsilon}}u_{\varepsilon}(t,x)\left[
\varphi(t,\overline{x})\operatorname{div}_{x}\Psi(\frac{x}{\varepsilon}%
)+\Psi(\frac{x}{\varepsilon})\cdot\nabla_{x}\varphi(t,\overline{x})\right]
dxdt\\
& =-\varepsilon^{-1}\int_{Q_{\varepsilon}}u_{\varepsilon}\left[  \frac
{1}{\varepsilon}\varphi(t,\overline{x})(\operatorname{div}_{y}\Psi)(\frac
{x}{\varepsilon})+\Psi_{\overline{x}}(\frac{x}{\varepsilon})\cdot
\nabla_{\overline{x}}\varphi(t,\overline{x})\right]  dxdt,
\end{align*}
the last equality above stemming from the fact that $\varphi$ does not depend
on $x_{d}$, and so $\nabla_{x}\varphi=(\nabla_{\overline{x}}\varphi,0)$.
Finally we use the equality $\operatorname{div}_{y}\Psi=0$ to get (\ref{2.14}).

Letting $E^{\prime}\ni\varepsilon\rightarrow0$ in (\ref{2.14}) yields
\begin{align}
& \int_{Q}\int_{I}M(\mathbf{v}(t,\overline{x},\cdot,y_{d})\cdot\Psi
(\cdot,y_{d}))\varphi(t,\overline{x})d\overline{x}dy_{d}dt\label{2.15}\\
& =-\int_{Q}\int_{I}u_{0}(t,\overline{x})M(\Psi_{\overline{x}}(\cdot
,y_{d}))\cdot\nabla_{\overline{x}}\varphi(t,\overline{x})d\overline{x}%
dy_{d}dt.\nonumber
\end{align}
First, taking in (\ref{2.15}) $\Psi=(\varphi\delta_{ij})_{1\leq i\leq d}$ (for
each fixed $1\leq j\leq d$) with $\varphi\in\mathcal{C}_{0}^{\infty}(Q)$ and
where $\delta_{ij}$ are the Kronecker delta, we notice that $\Psi$ does not
depend on $y$, so that we obtain

\begin{equation}%
\begin{array}
[c]{c}%
\displaystyle
\int_{Q}\left(  \int_{I}M(v_{j}(t,\overline{x},\cdot,y_{d}))dy_{d}\right)
\varphi(t,\overline{x})d\overline{x}dt=-\int_{Q}\left(  \int_{I}%
M(1)dy_{d}\right)  u_{0}(t,\overline{x})\frac{\partial\varphi}{\partial x_{j}%
}(t,\overline{x})d\overline{x}dt\\
\displaystyle =-\int_{Q}u_{0}\frac{\partial\varphi}{\partial x_{j}}d\overline{x}dt\int
_{I}dy_{d}=-2\int_{Q}u_{0}\frac{\partial\varphi}{\partial x_{j}}d\overline
{x}dt\text{ as }\int_{I}dy_{d}=2,
\end{array}
\label{2.16}%
\end{equation}

where we recall that $\mathbf{v}=(v_{j})_{1\leq j\leq d}$. Recalling that
$v_{j}\in L^{p}(Q;\mathcal{B}_{A}^{p}(\mathbb{R}^{d-1};L^{p}(I)))$, we infer
that the function $(t,\overline{x})\mapsto\int_{I}M(v_{j}(t,\overline{x}%
,\cdot,y_{d})dy_{d}$ belongs to $L^{p}(Q)$, so that (\ref{2.16}) yields
$\partial u_{0}/\partial x_{j}\in L^{p}(Q)$ for $1\leq j\leq d-1$, where
$\partial u_{0}/\partial x_{j}$ is the distributional derivative of $u_{0}$
with respect to $x_{j}$. We deduce that $u_{0}\in L^{2}(0,T;W^{1,p}(\Omega
_{0}))$. Coming back to (\ref{2.15}), we have
\begin{align*}
& \int_{Q}\int_{I}M(\mathbf{v}(t,\overline{x},\cdot,y_{d})\cdot\Psi
(\cdot,y_{d}))\varphi(t,\overline{x})d\overline{x}dy_{d}dt\\
& =\int_{Q}\int_{I}\left(  \nabla_{\overline{x}}u_{0}(t,\overline{x})\cdot
M(\Psi_{\overline{x}}(\cdot,y_{d})\right)  \varphi(t,\overline{x}%
)d\overline{x}dy_{d}dt\\
& =\int_{Q}\int_{I}\left(  \nabla_{\overline{x}}u_{0}(t,\overline{x})\cdot
M(\Psi(\cdot,y_{d})\right)  \varphi(t,\overline{x})d\overline{x}dy_{d}dt,
\end{align*}
where the last equality above arises from the equality $\nabla_{\overline{x}%
}u_{0}=\left(  \frac{\partial u_{0}}{\partial x_{1}},...,\frac{\partial u_{0}%
}{\partial x_{d-1}},0\right)  $. We obtain readily
\begin{equation}
\int_{Q}\left(  \int_{I}M\left(  (\mathbf{v}(t,\overline{x},\cdot
,y_{d})-\nabla_{\overline{x}}u_{0}(t,\overline{x}))\cdot\Psi(\cdot
,y_{d})\right)  dy_{d}\right)  \varphi(t,\overline{x})d\overline
{x}dt=0.\label{2.17}%
\end{equation}
From the arbitrariness of $\varphi$, (\ref{2.17}) entails
\[
\int_{I}M\left(  (\mathbf{v}(t,\overline{x},\cdot,y_{d})-\nabla_{\overline{x}%
}u_{0}(t,\overline{x}))\cdot\Psi(\cdot,y_{d})\right)  dy_{d}=0\text{ for a.e.
}(t,\overline{x})\in Q,
\]
and for all $\Psi\in(A^{\infty}(\mathbb{R}^{d-1};\mathcal{C}_{0}^{\infty
}(I))^{d}$ with $\operatorname{div}_{y}\Psi=0$. We make use of Lemma
\ref{c2.1} to deduce the existence of $u_{1}(t,\overline{x},\cdot,\cdot)\in
B_{\#A}^{1,p}(\mathbb{R}^{d-1};W^{1,p}(I))$ such that
\[
\mathbf{v}(t,\overline{x},\cdot,\cdot)-\nabla_{\overline{x}}u_{0}%
(t,\overline{x})=\nabla_{y}u_{1}(t,\overline{x},\cdot,\cdot)\text{ for a.e.
}(t,\overline{x})\in Q.
\]
Hence the existence of a function $(t,\overline{x})\mapsto u_{1}%
(t,\overline{x},\cdot,\cdot)$ from $Q$ into $B_{\#A}^{1,p}(\mathbb{R}%
^{d-1};W^{1,p}(I))$, which belongs to $L^{p}(Q;B_{\#A}^{1,p}(\mathbb{R}%
^{d-1};W^{1,p}(I)))$, such that $\mathbf{v}=\nabla_{\overline{x}}u_{0}%
+\nabla_{y}u_{1}$.
\end{proof}

The following result provides us with sufficient conditions for which the
convergence result in (\ref{2.8}) is strong.

\begin{theorem}
\label{t2.3}The assumptions are those of Theorem \emph{\ref{t2.2}} where
\emph{(\ref{2.7})} is replaced by \emph{(\ref{2.7'})} below
\begin{equation}
\sup_{\varepsilon>0}\varepsilon^{-\frac{1}{p}}\left\Vert u_{\varepsilon
}\right\Vert _{L^{\infty}(0,T;W^{1,p}(\Omega_{\varepsilon}))}\leq
C,\label{2.7'}%
\end{equation}
where $C$ is a positive constant. Moreover suppose that
\begin{equation}
\sup_{\varepsilon>0}\left\Vert \frac{\partial M_{\varepsilon}u_{\varepsilon}%
}{\partial t}\right\Vert _{L^{p^{\prime}}(0,T;(W^{1,p}(\Omega))^{\prime})}\leq
C,\label{3.16'}%
\end{equation}
where $M_{\varepsilon}$ is defined by \emph{(\ref{2.25})}. Assume finally that
$\Omega$ is regular enough so that the embedding $W^{1,p}(\Omega
)\hookrightarrow L^{p}(\Omega)$ is compact. Let $(u_{0},u_{1})$ and
$E^{\prime}$ be as in Theorem \emph{\ref{t2.2}}. Then, as $E^{\prime}%
\ni\varepsilon\rightarrow0$, the conclusions of Theorem \emph{\ref{t2.2}} hold
and further
\begin{equation}
u_{\varepsilon}\rightarrow u_{0}\text{ in }L^{p}(Q_{\varepsilon})\text{-strong
}\Sigma_{A}\text{.}\label{2.8'}%
\end{equation}

\end{theorem}

\begin{proof}
Let us first recall the definition of $M_{\varepsilon}$:
\[
(M_{\varepsilon}u_{\varepsilon})(t,\overline{x}%
)=\mathchoice {{\setbox0=\hbox{$\displaystyle{\textstyle
-}{\int}$ } \vcenter{\hbox{$\textstyle -$
}}\kern-.6\wd0}}{{\setbox0=\hbox{$\textstyle{\scriptstyle -}{\int}$ } \vcenter{\hbox{$\scriptstyle -$
}}\kern-.6\wd0}}{{\setbox0=\hbox{$\scriptstyle{\scriptscriptstyle -}{\int}$
} \vcenter{\hbox{$\scriptscriptstyle -$
}}\kern-.6\wd0}}{{\setbox0=\hbox{$\scriptscriptstyle{\scriptscriptstyle
-}{\int}$ } \vcenter{\hbox{$\scriptscriptstyle -$ }}\kern-.6\wd0}}\!\int
_{\varepsilon I}u_{\varepsilon}(t,\overline{x},\zeta)d\zeta\text{ for
}(t,\overline{x})\in Q.
\]
We know that $M_{\varepsilon}u_{\varepsilon}\in L^{\infty}(0,T;W^{1,p}%
(\Omega))$ with
\begin{equation}
\left\Vert M_{\varepsilon}u_{\varepsilon}\right\Vert _{L^{\infty}%
(0,T;W^{1,p}(\Omega))}\leq
C\ \ \ \ \ \ \ \ \ \ \ \ \ \ \ \ \ \ \ \ \ \ \ \ \ \ \label{*0}%
\end{equation}
where $C$ is a positive constant independent of $\varepsilon$, the last
inequality above being a consequence of (\ref{2.7'}). Next the following
Poincar\'{e}-Wirtinger inequality holds:
\begin{equation}
\varepsilon^{-\frac{1}{p}}\left\Vert u_{\varepsilon}-M_{\varepsilon
}u_{\varepsilon}\right\Vert _{L^{\infty}(0,T;L^{p}(\Omega_{\varepsilon}))}\leq
C\varepsilon\left\Vert \nabla u_{\varepsilon}\right\Vert _{L^{\infty
}(0,T;L^{p}(\Omega_{\varepsilon}))},\label{*1}%
\end{equation}
where $C>0$ is independent of $\varepsilon$. Indeed, from the density of
$\mathcal{C}^{1}(\overline{\Omega_{\varepsilon}})$ in $W^{1,p}(\Omega
_{\varepsilon})$, we may assume, without loss of generality, that
$u_{\varepsilon}$ is smooth enough. In that case, one has, for $\xi
\in\varepsilon I$,
\begin{align*}
u_{\varepsilon}(t,\overline{x},\xi)-M_{\varepsilon}u_{\varepsilon}%
(t,\overline{x})  & =\mathchoice {{\setbox0=\hbox{$\displaystyle{\textstyle
-}{\int}$ } \vcenter{\hbox{$\textstyle -$
}}\kern-.6\wd0}}{{\setbox0=\hbox{$\textstyle{\scriptstyle -}{\int}$ } \vcenter{\hbox{$\scriptstyle -$
}}\kern-.6\wd0}}{{\setbox0=\hbox{$\scriptstyle{\scriptscriptstyle -}{\int}$
} \vcenter{\hbox{$\scriptscriptstyle -$
}}\kern-.6\wd0}}{{\setbox0=\hbox{$\scriptscriptstyle{\scriptscriptstyle
-}{\int}$ } \vcenter{\hbox{$\scriptscriptstyle -$ }}\kern-.6\wd0}}\!\int
_{\varepsilon I}(u_{\varepsilon}(t,\overline{x},\xi)-u_{\varepsilon
}(t,\overline{x},z))dz\\
& =\mathchoice {{\setbox0=\hbox{$\displaystyle{\textstyle
-}{\int}$ } \vcenter{\hbox{$\textstyle -$
}}\kern-.6\wd0}}{{\setbox0=\hbox{$\textstyle{\scriptstyle -}{\int}$ } \vcenter{\hbox{$\scriptstyle -$
}}\kern-.6\wd0}}{{\setbox0=\hbox{$\scriptstyle{\scriptscriptstyle -}{\int}$
} \vcenter{\hbox{$\scriptscriptstyle -$
}}\kern-.6\wd0}}{{\setbox0=\hbox{$\scriptscriptstyle{\scriptscriptstyle
-}{\int}$ } \vcenter{\hbox{$\scriptscriptstyle -$ }}\kern-.6\wd0}}\!\int
_{\varepsilon I}\left(  \int_{0}^{1}\frac{\partial u_{\varepsilon}}{\partial
x_{d}}(t,\overline{x},z+s(\xi-z))\cdot(\xi-z)ds\right)  dz,
\end{align*}
so that, using Young's and H\"{o}lder's inequalities,
\begin{align*}
\left\vert u_{\varepsilon}(t,\overline{x},\xi)-M_{\varepsilon}u_{\varepsilon
}(t,\overline{x})\right\vert ^{p}  & \leq
\mathchoice {{\setbox0=\hbox{$\displaystyle{\textstyle
-}{\int}$ } \vcenter{\hbox{$\textstyle -$
}}\kern-.6\wd0}}{{\setbox0=\hbox{$\textstyle{\scriptstyle -}{\int}$ } \vcenter{\hbox{$\scriptstyle -$
}}\kern-.6\wd0}}{{\setbox0=\hbox{$\scriptstyle{\scriptscriptstyle -}{\int}$
} \vcenter{\hbox{$\scriptscriptstyle -$
}}\kern-.6\wd0}}{{\setbox0=\hbox{$\scriptscriptstyle{\scriptscriptstyle
-}{\int}$ } \vcenter{\hbox{$\scriptscriptstyle -$ }}\kern-.6\wd0}}\!\int
_{\varepsilon I}\int_{0}^{1}\left\vert \frac{\partial u_{\varepsilon}%
}{\partial x_{d}}(t,\overline{x},z+s(\xi-z))\right\vert ^{p}\left\vert
\xi-z\right\vert ^{p}dsdz\\
& \leq\mathchoice {{\setbox0=\hbox{$\displaystyle{\textstyle
-}{\int}$ } \vcenter{\hbox{$\textstyle -$
}}\kern-.6\wd0}}{{\setbox0=\hbox{$\textstyle{\scriptstyle -}{\int}$ } \vcenter{\hbox{$\scriptstyle -$
}}\kern-.6\wd0}}{{\setbox0=\hbox{$\scriptstyle{\scriptscriptstyle -}{\int}$
} \vcenter{\hbox{$\scriptscriptstyle -$
}}\kern-.6\wd0}}{{\setbox0=\hbox{$\scriptscriptstyle{\scriptscriptstyle
-}{\int}$ } \vcenter{\hbox{$\scriptscriptstyle -$ }}\kern-.6\wd0}}\!\int
_{\varepsilon I}\left\vert \xi-z\right\vert ^{p}dz\left(  \int_{\varepsilon
I}\left\vert \frac{\partial u_{\varepsilon}}{\partial x_{d}}(t,\overline
{x},\eta)\right\vert ^{p}d\eta\right) \\
& \leq2^{p}\varepsilon^{p}\int_{\varepsilon I}\left\vert \nabla u_{\varepsilon
}(t,\overline{x},\eta)\right\vert ^{p}d\eta.
\end{align*}
Integrating over $\Omega_{\varepsilon}$ the last series of inequalities above
and taking its $\mathrm{esssup}_{0\leq t\leq T}$ gives (\ref{*1}).

In view of (\ref{*0}) together with (\ref{3.16'}), we get that $M_{\varepsilon
}u_{\varepsilon}\in V^{p}=\{v\in L^{\infty}(0,T;W^{1,p}(\Omega)):\partial
v/\partial t\in L^{p^{\prime}}(0,T;(W^{1,p}(\Omega))^{\prime})\}$. It is
classically known that the compactness of the embedding $W^{1,p}%
(\Omega)\hookrightarrow L^{p}(\Omega)$ entails that of the embedding
$V^{p}\hookrightarrow L^{\infty}(0,T;L^{p}(\Omega))$. We therefore infer from
(\ref{*0}), (\ref{3.16'}) and the latter compactness result that there exists
a subsequence of $E^{\prime}$ not relabeled such that, as $E^{\prime}%
\ni\varepsilon\rightarrow0$,
\begin{equation}
M_{\varepsilon}u_{\varepsilon}\rightarrow u_{0}\text{ in }L^{\infty}%
(0,T;L^{p}(\Omega))\text{-strong.}\label{*2}%
\end{equation}
Now the inequality (\ref{*1}) yields, as $E^{\prime}\ni\varepsilon
\rightarrow0$,
\begin{equation}
\varepsilon^{-\frac{1}{p}}\left\Vert u_{\varepsilon}-M_{\varepsilon
}u_{\varepsilon}\right\Vert _{L^{\infty}(0,T;L^{p}(\Omega_{\varepsilon}%
))}\rightarrow0\text{.}\label{*13}%
\end{equation}
Next, we have
\begin{align*}
\varepsilon^{-\frac{1}{p}}\left\Vert u_{\varepsilon}-u_{0}\right\Vert
_{L^{\infty}(0,T;L^{p}(\Omega_{\varepsilon}))}  & \leq\varepsilon^{-\frac
{1}{p}}\left\Vert u_{\varepsilon}-M_{\varepsilon}u_{\varepsilon}\right\Vert
_{L^{\infty}(0,T;L^{p}(\Omega_{\varepsilon}))}\\
& +\varepsilon^{-\frac{1}{p}}\left\Vert M_{\varepsilon}u_{\varepsilon}%
-u_{0}\right\Vert _{L^{\infty}(0,T;L^{p}(\Omega_{\varepsilon}))},
\end{align*}
and
\[
\varepsilon^{-\frac{1}{p}}\left\Vert M_{\varepsilon}u_{\varepsilon}%
-u_{0}\right\Vert _{L^{\infty}(0,T;L^{p}(\Omega_{\varepsilon}))}=2^{\frac
{1}{p}}\left\Vert M_{\varepsilon}u_{\varepsilon}-u_{0}\right\Vert _{L^{\infty
}(0,T;L^{p}(\Omega))}.
\]
It follows readily from (\ref{*2}) and (\ref{*13}) that, as $E^{\prime}%
\ni\varepsilon\rightarrow0$,
\[
\varepsilon^{-\frac{1}{p}}\left\Vert u_{\varepsilon}-u_{0}\right\Vert
_{L^{\infty}(0,T;L^{p}(\Omega_{\varepsilon}))}\rightarrow0.
\]
This completes the proof.
\end{proof}

The next result and its corollary are proved exactly as their homologues in Theorem 6 and Corollary 5 in Ref. 
\cite{DPDE} (see also Ref. \cite{Deterhom}).

\begin{theorem}
\label{t2.4}Let $1<p,q<\infty$ and $r\geq1$ be such that $1/r=1/p+1/q\leq1 $.
Assume $(u_{\varepsilon})_{\varepsilon\in E}\subset L^{q}(Q_{\varepsilon})$ is
weakly $\Sigma_{A}$-convergent in $L^{q}(Q_{\varepsilon})$ to some $u_{0}\in
L^{q}(Q;\mathcal{B}_{A}^{q}(\mathbb{R}^{d-1};L^{q}(I)))$, and $(v_{\varepsilon
})_{\varepsilon\in E}\subset L^{p}(Q_{\varepsilon})$ is strongly $\Sigma_{A}%
$-convergent in $L^{p}(Q_{\varepsilon})$ to some $v_{0}\in L^{p}%
(Q;\mathcal{B}_{A}^{p}(\mathbb{R}^{d-1};L^{p}(I)))$. Then the sequence
$(u_{\varepsilon}v_{\varepsilon})_{\varepsilon\in E}$ is weakly $\Sigma_{A}%
$-convergent in $L^{r}(Q_{\varepsilon})$ to $u_{0}v_{0}$.
\end{theorem}

\begin{corollary}
\label{c2.2}Let $(u_{\varepsilon})_{\varepsilon\in E}\subset L^{p}%
(Q_{\varepsilon})$ and $(v_{\varepsilon})_{\varepsilon\in E}\subset
L^{p^{\prime}}(Q_{\varepsilon})\cap L^{\infty}(Q_{\varepsilon})$ ($1<p<\infty$
and $p^{\prime}=p/(p-1)$) be two sequences such that:

\begin{itemize}
\item[(i)] $u_{\varepsilon}\rightarrow u_{0}$ in $L^{p}(Q_{\varepsilon})$-weak
$\Sigma_{A}$;

\item[(ii)] $v_{\varepsilon}\rightarrow v_{0}$ in $L^{p^{\prime}%
}(Q_{\varepsilon})$-strong $\Sigma_{A}$;

\item[(iii)] $(v_{\varepsilon})_{\varepsilon\in E}$ is bounded in $L^{\infty
}(Q_{\varepsilon})$.
\end{itemize}

\noindent Then $u_{\varepsilon}v_{\varepsilon}\rightarrow u_{0}v_{0}$ in
$L^{p}(Q_{\varepsilon})$-weak $\Sigma_{A}$.
\end{corollary}

Another important result is the following proposition.

\begin{proposition}
\label{p2.2}Assume that $A$ is an ergodic algebra with mean value on
$\mathbb{R}^{d-1}$. Let $(u_{\varepsilon})_{\varepsilon\in E}$ be a sequence
in $L^{p}(0,T;W^{1,p}(\Omega_{\varepsilon}))$ such that
\[
\sup_{\varepsilon\in E}\left(  \varepsilon^{-1/p}\left\Vert u_{\varepsilon
}\right\Vert _{L^{p}(Q_{\varepsilon})}+\varepsilon^{1-1/p}\left\Vert \nabla
u_{\varepsilon}\right\Vert _{L^{p}(Q_{\varepsilon})}\right)  \leq C
\]
where $C>0$ is independent of $\varepsilon$. Then there exist a subsequence
$E^{\prime}$ of $E$ and a function $u\in L^{p}(Q;B_{\#A}^{1,p}(\mathbb{R}%
^{d-1};W^{1,p}(I)))$ with $u_{0}=\varrho(u)\in L^{p}(Q;\mathcal{B}_{A}%
^{1,p}(\mathbb{R}^{d-1};W^{1,p}(I)))$ such that, as $E^{\prime}\ni
\varepsilon\rightarrow0$,
\[
u_{\varepsilon}\rightarrow u_{0}\text{ in }L^{p}(Q_{\varepsilon})\text{-weak
}\Sigma_{A},
\]
and
\[
\varepsilon\nabla u_{\varepsilon}\rightarrow\nabla_{y}u\text{ in }%
L^{p}(Q_{\varepsilon})^{d}\text{-weak }\Sigma_{A}.
\]

\end{proposition}

\begin{proof}
From Theorem \ref{t2.1}, we can find a subsequence $E^{\prime}$ from $E$ and a
couple $(u_{0},u_{1})\in L^{p}(Q;\mathcal{B}_{A}^{p}(\mathbb{R}^{d-1}%
;L^{p}(I)))\times L^{p}(Q;\mathcal{B}_{A}^{p}(\mathbb{R}^{d-1};L^{p}(I)))^{d}$
such that, as $E^{\prime}\ni\varepsilon\rightarrow0$,
\begin{align*}
u_{\varepsilon}  & \rightarrow u_{0}\text{ in }L^{p}(Q_{\varepsilon
})\text{-weak }\Sigma_{A},\\
\varepsilon\nabla u_{\varepsilon}  & \rightarrow u_{1}\text{ in }%
L^{p}(Q_{\varepsilon})^{d}\text{-weak }\Sigma_{A}.
\end{align*}
Let us characterize $u_{1}$ in terms of $u_{0}$. To that end, let $\Phi
\in(\mathcal{C}_{0}^{\infty}(Q)\otimes A^{\infty}(\mathbb{R}^{d-1}%
;\mathcal{C}_{0}^{\infty}(I)))^{d}$; then we have
\[
\varepsilon^{-1}\int_{Q_{\varepsilon}}\varepsilon\nabla u_{\varepsilon}%
\cdot\Phi^{\varepsilon}dxdt=-\varepsilon^{-1}\int_{Q_{\varepsilon}%
}u_{\varepsilon}\left[  \varepsilon(\operatorname{div}_{\overline{x}}%
\Phi)^{\varepsilon}+(\operatorname{div}_{y}\Phi)^{\varepsilon}\right]  dxdt.
\]
Letting $E^{\prime}\ni\varepsilon\rightarrow0$, we get
\begin{equation}
\int_{Q}\int_{I}M(u_{1}(t,\overline{x},\cdot,\zeta)\cdot\Phi(t,\overline
{x},\cdot,\zeta))d\zeta d\overline{x}dt=-\int_{Q}\int_{I}M(u_{0}%
(t,\overline{x},\cdot,\zeta)\operatorname{div}_{y}\Phi(t,\overline{x}%
,\cdot,\zeta))d\zeta d\overline{x}dt.\label{*6}%
\end{equation}
This shows that $u_{1}=\overline{\nabla}_{\overline{y},\zeta}u_{0}$, so that
$u_{0}\in L^{p}(Q;\mathcal{B}_{A}^{1,p}(\mathbb{R}^{d-1};W^{1,p}(I)))$.

Now, coming back to (\ref{*6}) and choosing there $\Phi$ such that
$\operatorname{div}_{y}\Phi=0$, we readily get
\[
\int_{Q}\int_{I}M(u_{1}(t,\overline{x},\cdot,\zeta)\cdot\Phi(t,\overline
{x},\cdot,\zeta))d\zeta d\overline{x}dt=0\text{ for all such }\Phi\text{.}%
\]
Owing to Lemma \ref{c2.1}, there exists $u\in L^{p}(Q;B_{\#A}^{1,p}%
(\mathbb{R}^{d-1};W^{1,p}(I)))$ such that $u_{1}=\nabla_{y}u$. This yields
(since $A$ is ergodic) $u_{0}=\varrho(u)+c$ where $c$ is a constant depending
possibly on $(t,\overline{x})$.
\end{proof}

\section{Homogenized system\label{sec4}}

\subsection{On an auxiliary problem}

Our aim here is to study the well-posedness of the following Stokes system
\begin{equation}
\left\{
\begin{array}
[c]{l}%
\dfrac{\partial u}{\partial t}-\alpha\overline{\Delta}u+\overline{\nabla
}p=\boldsymbol{f}+\overline{\operatorname{div}}\boldsymbol{F}\text{ in
}(0,\infty)\times\mathbb{R}^{d-1}\times I,\\
\\
\overline{\operatorname{div}}u=0\text{ in }(0,\infty)\times\mathbb{R}%
^{d-1}\times I,\\
\\
u=0\text{ on }(0,\infty)\times\mathbb{R}^{d-1}\times\{-1,1\},\\
\\
u(0)=v_{0}\text{ in }\mathbb{R}^{d-1}\times I,
\end{array}
\right. \label{4.1}%
\end{equation}
where $\boldsymbol{f}$ and $\boldsymbol{F}$ are respectively $1\times d$ and
$d\times d$ matrices having their entries in $L^{2}(0,\infty;\mathcal{B}%
_{A}^{2}(\mathbb{R}^{d-1};L^{2}(I)))$ and $v_{0}\in\mathcal{B}_{A}%
^{2}(\mathbb{R}^{d-1};L^{2}(I))^{d}$; $\alpha$ is a given positive constant.
Here, for the sake of simplicity, we use the following notation:
\[
\overline{\Delta}=\left(  \sum_{i=1}^{d-1}\frac{\overline{\partial}^{2}%
}{\partial y_{i}^{2}}\right)  +\frac{\partial^{2}}{\partial y_{d}^{2}%
}\text{,\ }\overline{\nabla}=\left(  \frac{\overline{\partial}}{\partial
y_{1}},...,\frac{\overline{\partial}}{\partial y_{d-1}},\frac{\partial
}{\partial y_{d}}\right)  \text{, and }\overline{\operatorname{div}}%
=\overline{\nabla}\cdot\text{,}%
\]
the dot being denoting the usual Euclidean product.

We endow the space $\mathcal{B}_{A}^{2}(\mathbb{R}^{d-1};L^{2}(I))$ with the
norm
\[
\left\Vert u\right\Vert _{2}=\left[  \int_{-1}^{1}M\left(  \left\vert
u(\cdot,y_{d})\right\vert ^{2}\right)  dy_{d}\right]  ^{1/2},\ u\in
\mathcal{B}_{A}^{2}(\mathbb{R}^{d-1};L^{2}(I)).
\]
This being so, before we proceed forward, we need to establish the following
Poincar\'{e}-type inequality.

\begin{lemma}
\label{l4.1}There exists a positive constant $C$ such that
\[
\left\Vert u\right\Vert _{2}\leq C\left\Vert \overline{\nabla}u\right\Vert
_{2}\text{, all }u\in\mathcal{B}_{A}^{1,2}(\mathbb{R}^{d-1};H_{0}^{1}(I)).
\]

\end{lemma}

\begin{proof}
For $u\in\mathcal{B}_{A}^{1,2}(\mathbb{R}^{d-1};H_{0}^{1}(I))$, we have
\[
u(\overline{y},\zeta)=\int_{-1}^{\zeta}\frac{\partial u}{\partial y_{d}%
}(\overline{y},\tau)d\tau\text{ for all }\zeta\in(-1,1)\text{, }%
\]
so that, from the Cauchy-Schwarz inequality, one has
\[
\left\vert u(\overline{y},\zeta)\right\vert ^{2}\leq\left(  \int_{-1}^{\zeta
}\left\vert \frac{\partial u}{\partial y_{d}}(\overline{y},\tau)\right\vert
^{2}d\tau\right)  \left(  \int_{-1}^{\zeta}d\tau\right)  .
\]
Hence
\[
M\left(  \left\vert u(\cdot,\zeta)\right\vert ^{2}\right)  \leq2M\left(
\int_{-1}^{\zeta}\left\vert \frac{\partial u}{\partial y_{d}}(\cdot
,\tau)\right\vert ^{2}d\tau\right)  =2\int_{-1}^{\zeta}M\left(  \left\vert
\frac{\partial u}{\partial y_{d}}(\cdot,\tau)\right\vert ^{2}\right)  d\tau,
\]
the last equality above being stemming from the continuity of the mean value
operator. Now integrating over $I$, we readily get
\begin{align*}
\left\Vert u\right\Vert _{2}^{2}  & \leq2\int_{-1}^{1}\left(  \int_{-1}%
^{\zeta}M\left(  \left\vert \frac{\partial u}{\partial y_{d}}(\cdot
,\tau)\right\vert ^{2}\right)  d\tau\right)  d\zeta\\
& \leq2\int_{-1}^{1}\left(  \int_{-1}^{1}M\left(  \left\vert \frac{\partial
u}{\partial y_{d}}(\cdot,\tau)\right\vert ^{2}\right)  d\tau\right)  d\zeta\\
& \leq4\left\Vert \overline{\nabla}u\right\Vert _{2}^{2},
\end{align*}
and the proof is complete.
\end{proof}

Owing to Lemma \ref{l4.1}, it is a fact that $\mathcal{B}_{A}^{1,2}%
(\mathbb{R}^{d-1};H_{0}^{1}(I))$, endowed with the gradient norm $\left\Vert
\overline{\nabla}\right\Vert _{2}$, is a Hilbert space.

Now, we define the following function space
\[
\mathcal{V}=\{u\in(A^{\infty}(\mathbb{R}^{d-1};\mathcal{C}_{0}^{1}%
(I)))^{d}:\operatorname{div}u=0\},
\]
and we set $V=~$the closure of $\mathcal{V}$ in $\mathcal{B}_{A}%
^{1,2}(\mathbb{R}^{d-1};H_{0}^{1}(I))^{d}$ and $H=$~the closure of
$\mathcal{V}$ in $\mathcal{B}_{A}^{2}(\mathbb{R}^{d-1};L^{2}(I))^{d}$. We
equip $V$ and $H$ with the relative topologies defined by their respective
norms
\[
\left\Vert u\right\Vert _{V}=\left\Vert \nabla u\right\Vert _{2}=\left(
\int_{I}M\left(  \left\vert \overline{\nabla}\otimes u(\cdot,y_{d})\right\vert
^{2}\right)  dy_{d}\right)  ^{1/2}\text{ for }u\in V,
\]
where $\overline{\nabla}\otimes u=\left(  \frac{\overline{\partial}u_{i}%
}{\partial y_{j}}\right)  _{1\leq i,j\leq d}$ with $\frac{\overline{\partial}%
}{\partial y_{d}}=\frac{\partial}{\partial y_{d}}$ (the classical partial
derivative in the sense of distributions);
\[
\left\Vert u\right\Vert _{H}=\left\Vert u\right\Vert _{2}\text{ for }u\in H.
\]
One can easily see that $V=\{u\in\mathcal{B}_{A}^{1,2}(\mathbb{R}^{d-1}%
;H_{0}^{1}(I))^{d}:\overline{\operatorname{div}}u=0\}$.

The following existence result is in order.

\begin{proposition}
\label{p4.1}Assume $v_{0}\in H$. There exists a unique $u\in\mathcal{C}%
([0,\infty);H)\cap L^{2}(0,T;V)$ solving \emph{(\ref{4.1})}. Moreover
$\partial u/\partial t\in L^{2}(0,T;V^{\prime})$ and there exists $p\in
L^{2}(0,T;H)$ such that $(u,p)$ satisfies \emph{(\ref{4.1})}$_{1}$. $p$ is
unique provided that $\int_{I}M(p(\cdot,\zeta))d\zeta=0$.
\end{proposition}

\begin{proof}
The triple $(V,H,V^{\prime})$ is a Gelfand triple. With this in mind,
(\ref{4.1}) can be rewritten in the following equivalent form:
\begin{equation}
u^{\prime}+\mathcal{A}u=\ell\text{ in }V^{\prime}\text{, a.e. }%
t>0,\ u(0)=v_{0}\text{ in }H,\label{4.2}%
\end{equation}
where the linear operator $\mathcal{A}:V\rightarrow V^{\prime}$ is defined on
$V$ by
\[
\left\langle \mathcal{A}u,v\right\rangle =\alpha\int_{I}M(\overline{\nabla
}u(\cdot,y_{d})\cdot\overline{\nabla}v(\cdot,y_{d}))dy_{d}\text{ \ for }u,v\in
V,
\]
and $\ell\in V^{\prime}$ is defined by
\[
\left\langle \ell,v\right\rangle =\int_{I}M(\boldsymbol{f}(\cdot,y_{d}%
)v(\cdot,y_{d})-\boldsymbol{F}(\cdot,y_{d})\cdot\overline{\nabla}v(\cdot
,y_{d}))dy_{d},\ \ v\in V\text{.}%
\]
Then because of Lemma \ref{l4.1}, $\mathcal{A}$ is bounded and coercive.
Moreover $\ell$ defines a bounded linear functional on $V$. Therefore, using a
well known classical method of solving linear parabolic equations, we see that
(\ref{4.2}) admits a unique solution $u\in\mathcal{C}([0,\infty);H)\cap
L^{2}(0,\infty;V)$ with $u^{\prime}\in L^{2}(0,\infty;V^{\prime})$. The
existence of $p$ is a consequence of Proposition 3.1 in Ref. \cite{JW2022}.
\end{proof}

Of special interest will be the solutions of the following problems:
\begin{equation}
\left\{
\begin{array}
[c]{l}%
\dfrac{\partial\omega^{j}}{\partial t}-\alpha\overline{\Delta}_{y}\omega
^{j}+\overline{\nabla}_{y}\pi^{j}=0\text{ in }(0,T)\times\mathbb{R}%
^{d-1}\times I,\\
\\
\overline{\operatorname{div}}_{y}\omega^{j}=0\text{ in }(0,T)\times
\mathbb{R}^{d-1}\times I,\\
\\
\omega^{j}=0\text{ on }(0,T)\times\mathbb{R}^{d-1}\times\{-1,1\},\\
\\
\omega^{j}(0)=e_{j}\text{ in }\mathbb{R}^{d-1}\times I,\
{\displaystyle\int_{I}}
M(\omega_{d}^{j}(t,\cdot,\zeta))d\zeta=0,
\end{array}
\right. \label{4.1'}%
\end{equation}
for $1\leq j\leq d-1$, and
\[
\left\{
\begin{array}
[c]{l}%
\dfrac{\partial\omega^{d}}{\partial t}-\alpha\overline{\Delta}_{y}\omega
^{d}+\overline{\nabla}_{y}\pi^{d}=0\text{ in }(0,T)\times\mathbb{R}%
^{d-1}\times I,\\
\\
\overline{\operatorname{div}}_{y}\omega^{d}=0\text{ in }(0,T)\times
\mathbb{R}^{d-1}\times I,\\
\\
\omega^{d}=0\text{ on }(0,T)\times\mathbb{R}^{d-1}\times\{-1,1\},\\
\\
\omega^{d}(0)=e_{d}\text{ in }\mathbb{R}^{d-1}\times I,\
\end{array}
\right.
\]
for $j=d$, where $e_{j}$ ($1\leq j\leq d$) is the $j$th vector of the
canonical basis in $\mathbb{R}^{d}$ and $\omega^{j}=(\omega_{i}^{j})_{1\leq
i\leq d}$. Since the space
\[
V_{d}=\left\{  \boldsymbol{u}=(u_{i})_{1\leq i\leq d}\in V:\int_{I}%
M(u_{d}(\cdot,\zeta))d\zeta=0\right\}
\]
is a closed subspace of $V$ endowed with the relative norm, we deduce from
Proposition \ref{p4.1} that (\ref{4.1'}), in the case when $1\leq j\leq d-1$,
possesses a unique solution $\omega^{j}\in\mathcal{C}([0,T];H)\cap
L^{2}(0,T;V_{d})$, for any fixed $T>0$. It is also known from the same
proposition that $\omega^{d}$ exists uniquely in $\mathcal{C}([0,T];H)\cap
L^{2}(0,T;V)$. For such solutions, we define
\begin{align*}
G_{ij}(t)  & =\frac{1}{2}\int_{-1}^{1}M(\omega^{i}(t,\cdot,\zeta))e_{j}%
d\zeta,\ \ t\in\lbrack0,T],\ 1\leq i,j\leq d\\
& \equiv\frac{1}{2}\int_{-1}^{1}M(\omega_{j}^{i}(t,\cdot,\zeta))d\zeta.
\end{align*}
Since $\int_{-1}^{1}M(\omega^{j}(t,\cdot,\zeta))e_{d}d\zeta=0$, we have
$G_{jd}=0$ for all $1\leq j\leq d-1$. We are going to see below that the
matrix $(G_{ij})_{1\leq i,j\leq d}$ is symmetric, so that $G_{dj}=0$, and
therefore setting $G=(G_{ij})_{1\leq i,j\leq d-1}$, the following result holds.

\begin{proposition}
\label{p4.2'}The matrix $G$ is symmetric, positive definite and has entries
which decrease exponentially as $t$ increases.
\end{proposition}

\begin{proof}
Let us first check that $G$ is symmetric. For $1\leq i,j\leq d$ and for any
$t\in(0,T)$, we have, for a.e. $\tau\in(0,t)$,
\begin{align*}
\frac{d}{d\tau}(\omega^{i}(\tau),\omega^{j}(t-\tau))  & =\left\langle
\frac{\partial\omega^{i}}{\partial\tau}(\tau),\omega^{j}(t-\tau)\right\rangle
-\left\langle \frac{\partial\omega^{j}}{\partial\tau}(t-\tau),\omega^{i}%
(\tau)\right\rangle \\
& =-\alpha(\overline{\nabla}_{y}\omega^{i}(\tau),\overline{\nabla}_{y}%
\omega^{j}(t-\tau))+\alpha(\overline{\nabla}_{y}\omega^{i}(\tau),\overline
{\nabla}_{y}\omega^{j}(t-\tau))\\
& =0.
\end{align*}
Integrating over $(0,t)$ we obtain $(\omega^{i}(t),e_{j})-(e_{i},\omega
^{j}(t))=0$, i.e.,
\[
\int_{-1}^{1}M(\omega^{i}(t,\cdot,\zeta))e_{j}d\zeta=\int_{-1}^{1}M(\omega
^{j}(t,\cdot,\zeta))e_{i}d\zeta,
\]
or $G_{ij}(t)=G_{ji}(t)$. We infer $G_{jd}=G_{dj}=0$ for all $1\leq j\leq d-1
$ as $\int_{-1}^{1}M(\omega^{j}(t,\cdot,\zeta))e_{d}d\zeta=0$. This shows that
in the last row and last column of the matrix $(G_{ij})_{1\leq i,j\leq d}$,
only the coefficient $G_{dd}$ is not identically zero, so that $(G_{ij}%
)_{1\leq i,j\leq d}$ may be reduced to the matrix $G=(G_{ij})_{1\leq i,j\leq
d-1}$.

Let us now show that the $G_{ij}(t)$ decrease exponentially as $t$ increases.
To that end, we test (\ref{4.1'}) with $\omega^{j}$; then
\begin{equation}
\frac{1}{2}\frac{d}{dt}\left\Vert \omega^{j}(t)\right\Vert _{2}^{2}%
+\alpha\left\Vert \overline{\nabla}_{y}\omega^{j}(t)\right\Vert _{2}%
^{2}=0.\label{4.2'}%
\end{equation}
But $\left\Vert \omega^{j}(t)\right\Vert _{2}\leq C\left\Vert \overline
{\nabla}_{y}\omega^{j}(t)\right\Vert _{2}$ (see Lemma \ref{l4.1}), where $C>0$
is independent of $\omega_{j}$. It follows from (\ref{4.2'}) that
\[
\frac{1}{2}\frac{d}{dt}\left\Vert \omega^{j}(t)\right\Vert _{2}^{2}%
+\frac{\alpha}{C}\left\Vert \omega^{j}(t)\right\Vert _{2}^{2}\leq0.
\]
Applying Gronwall's inequality leads us at
\[
\left\Vert \omega^{j}(t)\right\Vert _{2}^{2}\leq\left\Vert \omega
^{j}(0)\right\Vert _{2}^{2}\exp\left(  -\frac{\alpha}{C}t\right)  ,
\]
that is,
\begin{equation}
\left\Vert \omega^{j}(t)\right\Vert _{2}\leq\sqrt{2}\exp\left(  -\frac{\alpha
}{2C}t\right)  \text{ for all }t\in\left[  0,T\right]  .\label{4.4'}%
\end{equation}
The final step is to check that $G$ is positive definite. But arguing exactly
as in the proof of Theorem 2 in Ref. \cite{Sandrakov}, we obtain the result.
\end{proof}

\subsection{Passage to the limit in (\ref{1.1})}

Throughout this section, $A$ is an ergodic algebra with mean value on
$\mathbb{R}^{d-1}$.

According to Propositions \ref{pr2.1} and \ref{pr2.2}, the following uniform
estimates hold: there exists a positive constant $C$ such that for all
$\varepsilon>0$,
\begin{equation}%
\begin{array}
[c]{l}%
\left\Vert \boldsymbol{u}_{\varepsilon}\right\Vert _{L^{\infty}(0,T;L^{2}%
(\Omega_{\varepsilon})^{d})}\leq C\varepsilon^{\frac{1}{2}},\ \varepsilon
\left\Vert \nabla\boldsymbol{u}_{\varepsilon}\right\Vert _{L^{2}%
(Q_{\varepsilon})^{d\times d}}\leq C\varepsilon^{\frac{1}{2}},\ \left\Vert
\varphi_{\varepsilon}\right\Vert _{L^{\infty}(0,T;H^{1}(\Omega_{\varepsilon
}))}\leq C\varepsilon^{\frac{1}{2}},\\
\\
\left\Vert \dfrac{\partial M_{\varepsilon}\varphi_{\varepsilon}}{\partial
t}\right\Vert _{L^{2}(0,T;(H^{1}(\Omega))^{\prime})}\leq C\text{, }\left\Vert
\mu_{\varepsilon}\right\Vert _{L^{2}(0,T;H^{1}(\Omega_{\varepsilon}))}\leq
C\varepsilon^{\frac{1}{2}},\ \left\Vert p_{\varepsilon}\right\Vert
_{L^{2}(Q_{\varepsilon})}\leq C\varepsilon^{\frac{1}{2}}\text{,}\\
\\
\text{and }\left\Vert f(\varphi_{\varepsilon})\right\Vert _{L^{\infty
}(0,T;L^{1}(\Omega_{\varepsilon}))}\leq C\varepsilon.
\end{array}
\label{4.3}%
\end{equation}
In view of Proposition \ref{p2.2} and Theorems \ref{t2.1}, \ref{t2.2} and
\ref{t2.3}, given an ordinary sequence $E$, there exist a subsequence
$E^{\prime}$ of $E$ and functions $\boldsymbol{u}_{0}\in L^{2}(Q;\mathcal{B}%
_{A}^{1,2}(\mathbb{R}^{d-1};H_{0}^{1}(I)))^{d}$, $(\varphi_{0},\varphi
_{1}),(\mu_{0},\mu_{1})\in L^{2}(0,T;H^{1}(\Omega_{0}))\times L^{2}%
(Q;B_{\#A}^{1,2}(\mathbb{R}^{d-1};H^{1}(I)))$ and $p_{0}\in L^{2}%
(Q;\mathcal{B}_{A}^{2}(\mathbb{R}^{d-1};L^{2}(I)))$ such that, as $E^{\prime
}\ni\varepsilon\rightarrow0$,
\begin{equation}
\boldsymbol{u}_{\varepsilon}\rightarrow\boldsymbol{u}_{0}\text{ in }%
L^{2}(Q_{\varepsilon})^{d}\text{-weak }\Sigma_{A}\ \ \ \ \ \ \ \ \ \label{4.4}%
\end{equation}%
\begin{equation}
\varepsilon\nabla\boldsymbol{u}_{\varepsilon}\rightarrow\overline{\nabla}%
_{y}\boldsymbol{u}_{0}\text{ in }L^{2}(Q_{\varepsilon})^{d\times d}\text{-weak
}\Sigma_{A}\label{4.5}%
\end{equation}%
\begin{equation}
p_{\varepsilon}\rightarrow p_{0}\text{ in }L^{2}(Q_{\varepsilon})\text{-weak
}\Sigma_{A}\ \ \ \ \ \ \ \ \ \ \ \ \label{4.6}%
\end{equation}%
\begin{equation}
\varphi_{\varepsilon}\rightarrow\varphi_{0}\text{ in }L^{2}(Q_{\varepsilon
})\text{-strong }\Sigma_{A}\ \ \ \ \ \ \ \ \ \ \ \label{4.7}%
\end{equation}%
\begin{equation}
\nabla\varphi_{\varepsilon}\rightarrow\nabla_{\overline{x}}\varphi_{0}%
+\nabla_{y}\varphi_{1}\text{ in }L^{2}(Q_{\varepsilon})^{d}\text{-weak }%
\Sigma_{A}\label{4.8}%
\end{equation}%
\begin{equation}
\mu_{\varepsilon}\rightarrow\mu_{0}\text{ in }L^{2}(Q_{\varepsilon
})\text{-weak }\Sigma_{A}\ \ \ \ \ \ \ \ \ \ \label{4.9}%
\end{equation}%
\begin{equation}
\nabla\mu_{\varepsilon}\rightarrow\nabla_{\overline{x}}\mu_{0}+\nabla_{y}%
\mu_{1}\text{ in }L^{2}(Q_{\varepsilon})^{d}\text{-weak }\Sigma_{A}%
,\label{4.10}%
\end{equation}
where $\nabla_{\overline{x}}\varphi_{0}=(\frac{\partial\varphi_{0}}{\partial
x_{1}},...,\frac{\partial\varphi_{0}}{\partial x_{d-1}},0)$ (and the same for
$\nabla_{\overline{x}}\mu_{0}$). Since $\operatorname{div}\boldsymbol{u}%
_{\varepsilon}=0$ in $Q_{\varepsilon}$, it follows that $\overline
{\operatorname{div}}_{y}\boldsymbol{u}_{0}=0$ in $Q\times\mathbb{R}%
^{d-1}\times I$. Indeed, setting
\[
\overline{\boldsymbol{u}}_{\varepsilon}=(u_{\varepsilon,1},...,u_{\varepsilon
,d-1}),\ \ \ \ \ \ \ \ \ \ \ \ \ \ \
\]
we have, for $\boldsymbol{\varphi}\in\mathcal{C}_{0}^{\infty}(Q)\otimes
A^{\infty}(\mathbb{R}^{d-1};\mathcal{C}_{0}^{\infty}(I))$,
\begin{align*}
0  & =\int_{Q_{\varepsilon}}\operatorname{div}\boldsymbol{u}_{\varepsilon
}(t,x)\boldsymbol{\varphi}\left(  t,\overline{x},\frac{x}{\varepsilon}\right)
dxdt\\
& =-\int_{Q_{\varepsilon}}\overline{\boldsymbol{u}}_{\varepsilon}\cdot
(\nabla_{\overline{x}}\boldsymbol{\varphi})^{\varepsilon}dxdt+\frac
{1}{\varepsilon}\int_{Q_{\varepsilon}}\boldsymbol{u}_{\varepsilon}\cdot
(\nabla_{y}\boldsymbol{\varphi})^{\varepsilon}dxdt,
\end{align*}
where $\boldsymbol{\varphi}^{\varepsilon}(t,x)=\boldsymbol{\varphi}\left(
t,\overline{x},\frac{x}{\varepsilon}\right)  $ for $(t,x)\in Q_{\varepsilon}
$. Letting $E^{\prime}\ni\varepsilon\rightarrow0$ yields
\[
\int_{Q}\int_{-1}^{1}M(\boldsymbol{u}_{0}(t,\overline{x},\cdot,\zeta
)\cdot\nabla_{y}\boldsymbol{\varphi}(t,\overline{x},\cdot,\zeta))d\zeta
d\overline{x}dt=0.
\]
This amounts to $\overline{\operatorname{div}}_{y}\boldsymbol{u}_{0}=0$ in
$Q\times\mathbb{R}^{d-1}\times I$, where $\overline{\operatorname{div}}%
_{y}\boldsymbol{u}_{0}=\overline{\operatorname{div}}_{\overline{y}}%
\overline{\boldsymbol{u}}_{0}+\frac{\partial u_{0,d}}{\partial\zeta}$ with
$\overline{\boldsymbol{u}}_{0}=(u_{0,i})_{1\leq i\leq d-1}$.

Now, set
\begin{align}
\boldsymbol{u}(t,\overline{x})  & =\frac{1}{2}\int_{-1}^{1}M(\boldsymbol{u}%
_{0}(t,\overline{x},\cdot,\zeta))d\zeta\text{ for }(t,\overline{x})\in
Q\label{4.10'}\\
& =(u_{i}(t,\overline{x}))_{1\leq i\leq d}\text{ and }\overline{\boldsymbol{u}%
}=(u_{i})_{1\leq i\leq d-1}.\nonumber
\end{align}
Then $\boldsymbol{u}\in L^{2}(Q)^{d}$. Moreover
\begin{equation}
\operatorname{div}_{\overline{x}}\overline{\boldsymbol{u}}=0\text{ in }Q\text{
and }\overline{\boldsymbol{u}}\cdot\boldsymbol{n}=0\text{ on }(0,T)\times
\partial\Omega,\label{4.10''}%
\end{equation}
where $\boldsymbol{n}$ is the outward unit normal to $\partial\Omega$. First
of all, we have
\begin{equation}
u_{d}=0\text{ in }Q\text{.\ \ \ \ \ \ \ \ \ \ \ \ \ \ \ \ \ \ \ \ \ \ \ \ \ }%
\label{4.10'''}%
\end{equation}
Indeed, from the equality $\overline{\operatorname{div}}_{y}\boldsymbol{u}%
_{0}=0$ in $Q\times\mathbb{R}^{d-1}\times I$, we have $M(\overline
{\operatorname{div}}_{y}\boldsymbol{u}_{0})=0$, that is $\frac{\partial
}{\partial\zeta}M(u_{0,d}(t,\overline{x},\cdot,\zeta))=0$. This shows that
$u_{0,d}$ is independent of $\zeta$. But $u_{\varepsilon,d}=0$ on
$(0,T)\times\Omega\times\{\varepsilon\}$, so that $M(u_{0,d}(t,\overline
{x},\cdot,\zeta))=0 $ on $(0,T)\times\Omega\times\{1\}$, i.e. $M(u_{0,d}%
(t,\overline{x},\cdot))=0$ in $Q$ since $u_{0,d}$ does not depend on $\zeta$.
This shows that $\boldsymbol{u}=(\overline{\boldsymbol{u}},0)$.

This being so, let us check (\ref{4.10''}). To that end, let $\varphi
\in\mathcal{D}(\overline{Q})$. Using the Stokes formula together with the
equality $\operatorname{div}\boldsymbol{u}_{\varepsilon}=0$ in $Q_{\varepsilon
}$, we obtain
\[
\int_{Q_{\varepsilon}}\overline{\boldsymbol{u}}_{\varepsilon}(t,x)\cdot
\nabla_{\overline{x}}\varphi(t,\overline{x}%
)dxdt=0.\ \ \ \ \ \ \ \ \ \ \ \ \ \
\]
Dividing the last equality above by $\varepsilon$ and letting $E^{\prime}%
\ni\varepsilon\rightarrow0$, we are led to
\[
\int_{Q}\overline{\boldsymbol{u}}(t,x)\cdot\nabla_{\overline{x}}%
\varphi(t,\overline{x})d\overline{x}dt=0.\ \ \ \ \ \ \ \ \ \ \ \ \ \ \ \ \
\]
This yields at once (\ref{4.10''}).

Also since $\int_{\Omega_{\varepsilon}}p_{\varepsilon}dx=0$, we have
$\int_{\Omega_{0}}\int_{I}M(p_{0}(t,\overline{x},\cdot,\zeta)d\zeta
d\overline{x}=0$.

The following global homogenized result holds.

\begin{proposition}
\label{p4.2}The functions $\boldsymbol{u}_{0},\varphi_{0},\varphi_{1},\mu
_{0},\mu_{1}$ and $p_{0}$ solve the following system:
\begin{equation}
\left\{
\begin{array}
[c]{l}%
-\dfrac{1}{2}%
{\displaystyle\int_{Q}}
{\displaystyle\int_{I}}
M\left(  \boldsymbol{u}_{0}(t,\overline{x},\cdot,\zeta)\frac{\partial\Psi
}{\partial t}(t,\overline{x},\cdot,\zeta)\right)  d\zeta d\overline{x}dt\\
\\
+\dfrac{\alpha}{2}%
{\displaystyle\int_{Q}}
{\displaystyle\int_{I}}
M(\overline{\nabla}_{y}\boldsymbol{u}_{0}\cdot\nabla_{y}\Psi)d\zeta
d\overline{x}dt\\
\\
\ \ -\dfrac{1}{2}%
{\displaystyle\int_{Q}}
{\displaystyle\int_{I}}
M\left(  \varphi_{0}\left[  (\nabla_{\overline{x}}\mu_{0}+\nabla_{y}\mu
_{1})\Psi+\mu_{0}\operatorname{div}_{\overline{x}}\Psi\right]  \right)  d\zeta
d\overline{x}dt\\
\\
\ \ \ \ \ \ -\dfrac{1}{2}%
{\displaystyle\int_{Q}}
{\displaystyle\int_{I}}
M(p_{0}\operatorname{div}_{\overline{x}}\Psi)d\zeta d\overline{x}dt=\dfrac
{1}{2}%
{\displaystyle\int_{Q}}
{\displaystyle\int_{I}}
M(\boldsymbol{h}\Psi)d\zeta d\overline{x}dt;
\end{array}
\right. \label{4.11}%
\end{equation}%
\begin{equation}
\left\{
\begin{array}
[c]{l}%
-\dfrac{1}{2}%
{\displaystyle\int_{Q}}
{\displaystyle\int_{I}}
M\left(  \varphi_{0}\frac{\partial\phi_{0}}{\partial t}\right)  d\zeta
d\overline{x}dt-\dfrac{1}{2}%
{\displaystyle\int_{Q}}
{\displaystyle\int_{I}}
M(\varphi_{0}\boldsymbol{u}_{0}(\nabla_{\overline{x}}\phi_{0}+\nabla_{y}%
\phi_{1}))d\zeta d\overline{x}dt\\
\\
\ \ \ +\dfrac{1}{2}%
{\displaystyle\int_{Q}}
{\displaystyle\int_{I}}
M\left(  (\nabla_{\overline{x}}\mu_{0}+\nabla_{y}\mu_{1})(\nabla_{\overline
{x}}\phi_{0}+\nabla_{y}\phi_{1})\right)  d\zeta d\overline{x}dt=0;
\end{array}
\right. \label{4.12}%
\end{equation}%
\begin{equation}
\left\{
\begin{array}
[c]{l}%
\dfrac{1}{2}%
{\displaystyle\int_{Q}}
{\displaystyle\int_{I}}
M(\mu_{0}\chi_{0})d\zeta d\overline{x}dt=\dfrac{\lambda}{2}%
{\displaystyle\int_{Q}}
{\displaystyle\int_{I}}
M(f(\varphi_{0})\chi_{0})d\zeta d\overline{x}dt\\
\\
\ \ \ +\dfrac{\beta}{2}%
{\displaystyle\int_{Q}}
{\displaystyle\int_{I}}
M\left(  (\nabla_{\overline{x}}\varphi_{0}+\nabla_{y}\varphi_{1}%
)(\nabla_{\overline{x}}\chi_{0}+\nabla_{y}\chi_{1})\right)  d\zeta
d\overline{x}dt;
\end{array}
\right. \label{4.13}%
\end{equation}%
\begin{equation}
\boldsymbol{u}_{0}(0,\overline{x},y)=\boldsymbol{u}^{0}(\overline{x})\text{
and }\varphi_{0}(0,\overline{x})=\varphi^{0}(\overline{x})\text{ for a.e.
}\overline{x}\in\Omega\text{ and }y\in\mathbb{R}^{d-1}\times I,\label{4.14}%
\end{equation}
for all $\Psi\in(\mathcal{C}_{0}^{\infty}(Q)\otimes A^{\infty}(\mathbb{R}%
^{d-1};\mathcal{C}_{0}^{\infty}(I)))^{d}$ with $\operatorname{div}_{y}\Psi=0$
and all $(\phi_{0},\phi_{1}),(\chi_{0},\chi_{1})\in\mathcal{C}_{0}^{\infty
}(Q)\times(\mathcal{C}_{0}^{\infty}(Q)\otimes A^{\infty}(\mathbb{R}%
^{d-1};\mathcal{C}_{0}^{\infty}(I)))$.
\end{proposition}

\begin{proof}
Let $\Psi\in(\mathcal{C}_{0}^{\infty}(Q)\otimes A^{\infty}(\mathbb{R}%
^{d-1};\mathcal{C}_{0}^{\infty}(I)))^{d}$, and let $(\phi_{0},\phi_{1}%
),(\chi_{0},\chi_{1})\in\mathcal{C}_{0}^{\infty}(Q)\times(\mathcal{C}%
_{0}^{\infty}(Q)\otimes A^{\infty}(\mathbb{R}^{d-1};\mathcal{C}_{0}^{\infty
}(I)))$. We define, for $(t,x)\in Q_{\varepsilon}$
\begin{align*}
\Psi^{\varepsilon}(t,x)  & =\Psi(t,\overline{x},\frac{x}{\varepsilon
}),\ \ \phi_{\varepsilon}(t,x)=\phi_{0}(t,\overline{x})+\varepsilon\phi
_{1}(t,\overline{x},\frac{x}{\varepsilon})\\
\chi_{\varepsilon}(t,x)  & =\chi_{0}(t,\overline{x})+\varepsilon\chi
_{1}(t,\overline{x},\frac{x}{\varepsilon}).
\end{align*}
Taking $(\Psi^{\varepsilon},\phi_{\varepsilon},\chi_{\varepsilon}%
)\in\mathcal{C}_{0}^{\infty}(Q_{\varepsilon})^{d}\times\mathcal{C}_{0}%
^{\infty}(Q_{\varepsilon})\times\mathcal{C}_{0}^{\infty}(Q_{\varepsilon})$ as
test function in the variational form (\ref{e2.1}), (\ref{e2.2}) and
(\ref{e2.3}), we obtain
\begin{equation}%
\begin{array}
[c]{l}%
-%
{\displaystyle\int_{Q_{\varepsilon}}}
\boldsymbol{u}_{\varepsilon}\left(  \dfrac{\partial\Psi}{\partial t}\right)
^{\varepsilon}dxdt+\alpha\varepsilon^{2}%
{\displaystyle\int_{Q_{\varepsilon}}}
\nabla\boldsymbol{u}_{\varepsilon}\cdot\left(  (\nabla_{\overline{x}}%
\Psi)^{\varepsilon}+\dfrac{1}{\varepsilon}(\nabla_{y}\Psi)^{\varepsilon
}\right)  dxdt\\
\\
\ \ \ \ -%
{\displaystyle\int_{Q_{\varepsilon}}}
p_{\varepsilon}\left(  (\operatorname{div}_{\overline{x}}\Psi)^{\varepsilon
}+\dfrac{1}{\varepsilon}(\operatorname{div}_{y}\Psi)^{\varepsilon}\right)
dxdt-%
{\displaystyle\int_{Q_{\varepsilon}}}
\mu_{\varepsilon}\nabla\varphi_{\varepsilon}\Psi^{\varepsilon}dxdt\\
\\
\ \ \ \ \ \ \ \ \ \ =%
{\displaystyle\int_{Q_{\varepsilon}}}
\boldsymbol{h}\Psi^{\varepsilon}dxdt;
\end{array}
\label{4.15}%
\end{equation}%
\begin{equation}%
\begin{array}
[c]{l}%
-%
{\displaystyle\int_{Q_{\varepsilon}}}
\varphi_{\varepsilon}\dfrac{\partial\phi_{\varepsilon}}{\partial t}dxdt+%
{\displaystyle\int_{Q_{\varepsilon}}}
(\boldsymbol{u}_{\varepsilon}\cdot\nabla\varphi_{\varepsilon})\phi
_{\varepsilon}dxdt\\
\\
\ \ \ \ +%
{\displaystyle\int_{Q_{\varepsilon}}}
\nabla\mu_{\varepsilon}\cdot(\nabla_{\overline{x}}\phi_{0}+\varepsilon
(\nabla_{\overline{x}}\phi_{1})^{\varepsilon}+(\nabla_{y}\phi_{1}%
)^{\varepsilon})dxdt=0;
\end{array}
\label{4.16}%
\end{equation}%
\begin{equation}
\int_{Q_{\varepsilon}}\mu_{\varepsilon}\chi_{\varepsilon}dxdt=\beta
\int_{Q_{\varepsilon}}\nabla\varphi_{\varepsilon}\cdot\nabla\chi_{\varepsilon
}dxdt+\lambda\int_{Q_{\varepsilon}}f(\varphi_{\varepsilon})\chi_{\varepsilon
}dxdt.\label{4.17}%
\end{equation}

Let us first deal with (\ref{4.15}): We pass to the limit in (\ref{4.15}) when
$E^{\prime}\ni\varepsilon\rightarrow0$ to get
\[
-\frac{1}{2}\int_{Q}\int_{I}M(p_{0}\operatorname{div}_{y}\Psi)d\zeta
d\overline{x}dt=0.\ \ \ \ \ \ \ \ \ \ \ \ \ \ \ \ \
\]
This shows that $p_{0}$ does not depend on $y$, i.e. $p_{0}(t,\overline
{x},y)=p_{0}(t,\overline{x})$, and thus $\int_{\Omega_{0}}p_{0}(t,\overline
{x})d\overline{x}=0$, so that $p_{0}\in L^{2}(0,T;L_{0}^{2}(\Omega))$.

Next, we choose $\Psi$ such that $\operatorname{div}_{y}\Psi=0$, and we divide
both sides of (\ref{4.15}) by $\varepsilon$ to obtain
\begin{equation}%
\begin{array}
[c]{l}%
-\dfrac{1}{\varepsilon}%
{\displaystyle\int_{Q_{\varepsilon}}}
\boldsymbol{u}_{\varepsilon}\left(  \frac{\partial\Psi}{\partial t}\right)
^{\varepsilon}dxdt+\dfrac{\alpha}{\varepsilon}%
{\displaystyle\int_{Q_{\varepsilon}}}
\varepsilon\nabla\boldsymbol{u}_{\varepsilon}\cdot\left(  (\nabla
_{\overline{x}}\Psi)^{\varepsilon}+\dfrac{1}{\varepsilon}(\nabla_{y}%
\Psi)^{\varepsilon}\right)  dxdt\\
\\
\ \ \ \ -\dfrac{1}{\varepsilon}%
{\displaystyle\int_{Q_{\varepsilon}}}
p_{\varepsilon}(\operatorname{div}_{\overline{x}}\Psi)^{\varepsilon
}dxdt-\dfrac{1}{\varepsilon}%
{\displaystyle\int_{Q_{\varepsilon}}}
\mu_{\varepsilon}\nabla\varphi_{\varepsilon}\Psi^{\varepsilon}dxdt\\
\\
\ \ \ \ \ \ \ \ \ \ =\dfrac{1}{\varepsilon}%
{\displaystyle\int_{Q_{\varepsilon}}}
\boldsymbol{h}\Psi^{\varepsilon}dxdt.
\end{array}
\label{4.18}%
\end{equation}
But
\[
\int_{Q_{\varepsilon}}\mu_{\varepsilon}\nabla\varphi_{\varepsilon}%
\Psi^{\varepsilon}dxdt=-\int_{Q_{\varepsilon}}\varphi_{\varepsilon}(\nabla
\mu_{\varepsilon}\Psi^{\varepsilon}+\mu_{\varepsilon}(\operatorname{div}%
_{\overline{x}}\Psi)^{\varepsilon})dxdt.
\]
Letting $E^{\prime}\ni\varepsilon\rightarrow0$ in (\ref{4.18}),
\begin{equation}%
\begin{array}
[c]{l}%
-\dfrac{1}{2}%
{\displaystyle\int_{Q}}
{\displaystyle\int_{I}}
M\left(  \boldsymbol{u}_{0}(t,\overline{x},\cdot,\zeta)\dfrac{\partial\Psi
}{\partial t}(t,\overline{x},\cdot,\zeta)\right)  d\zeta d\overline{x}dt\\
\\
+\dfrac{\alpha}{2}%
{\displaystyle\int_{Q}}
{\displaystyle\int_{I}}
M(\overline{\nabla}_{y}\boldsymbol{u}_{0}\cdot\nabla_{y}\Psi)d\zeta
d\overline{x}dt\\
\\
\ \ +\dfrac{1}{2}%
{\displaystyle\int_{Q}}
{\displaystyle\int_{I}}
M\left(  \varphi_{0}\left[  (\nabla_{\overline{x}}\mu_{0}+\nabla_{y}\mu
_{1})\Psi+\mu_{0}\operatorname{div}_{\overline{x}}\Psi\right]  \right)  d\zeta
d\overline{x}dt\\
\\
\ \ \ \ \ \ -\dfrac{1}{2}%
{\displaystyle\int_{Q}}
{\displaystyle\int_{I}}
M(p_{0}\operatorname{div}_{\overline{x}}\Psi)d\zeta d\overline{x}dt=\dfrac
{1}{2}%
{\displaystyle\int_{Q}}
{\displaystyle\int_{I}}
M(\boldsymbol{h}\Psi)d\zeta d\overline{x}dt,
\end{array}
\label{4.19}%
\end{equation}
that is (\ref{4.11}). We recall that to obtain the penultimate term of the
left-hand side of (\ref{4.19}), we have used the strong sigma-convergence
(\ref{4.7}) associated to the weak sigma-convergence (\ref{4.10}) in light of
Corollary \ref{c2.2}.

Let us now consider (\ref{4.16}). We divide both sides therein by
$\varepsilon$ and use the equality
\[
\int_{Q_{\varepsilon}}(\boldsymbol{u}_{\varepsilon}\nabla\varphi_{\varepsilon
})\phi_{\varepsilon}dxdt=-\int_{Q_{\varepsilon}}\varphi_{\varepsilon
}\boldsymbol{u}_{\varepsilon}\nabla\phi_{\varepsilon}dxdt.
\]
Then passing to the limit when $E^{\prime}\ni\varepsilon\rightarrow0$ in the
resulting equality, we get (\ref{4.12}).

Let us finally deal with (\ref{4.17}). Therein the limit passage in
$\int_{Q_{\varepsilon}}f(\varphi_{\varepsilon})\chi_{\varepsilon}dxdt$ needs a
careful treatment. Indeed we need to check that
\begin{equation}
\frac{1}{\varepsilon}\int_{Q_{\varepsilon}}f(\varphi_{\varepsilon}%
)\chi_{\varepsilon}dxdt\rightarrow\int_{Q}\int_{I}f(\varphi_{0})\chi_{0}d\zeta
d\overline{x}dt.\label{4.21}%
\end{equation}
First of all, from (\ref{4.7}) we have $\varepsilon^{-\frac{1}{2}}\left\Vert
\varphi_{\varepsilon}-\varphi_{0}\right\Vert _{L^{2}(Q_{\varepsilon}%
)}\rightarrow0$ as $E^{\prime}\ni\varepsilon\rightarrow0$. But
\begin{align*}
\varepsilon^{-1}\int_{Q_{\varepsilon}}\left\vert \varphi_{\varepsilon
}(t,x)-\varphi_{0}(t,\overline{x})\right\vert ^{2}dxdt  & =\int_{Q_{1}%
}\left\vert \varphi_{\varepsilon}(t,\overline{x},\varepsilon x_{d}%
)-\varphi_{0}(t,\overline{x})\right\vert ^{2}dxdt\\
& \rightarrow0\text{ as }E^{\prime}\ni\varepsilon\rightarrow0.
\end{align*}
This shows that the sequence $(\widetilde{\varphi}_{\varepsilon}%
)_{\varepsilon\in E^{\prime}}$ defined by $\widetilde{\varphi}_{\varepsilon
}(t,x)=\varphi_{\varepsilon}(t,\overline{x},\varepsilon x_{d})$ \ ($(t,x)\in
Q_{1}$) converges strongly to $\varphi_{0}$ in $L^{2}(Q_{1})$, and so,
$\widetilde{\varphi}_{\varepsilon}\rightarrow\varphi_{0}$ a.e. in $Q_{1}$. The
continuity of $f$ entails $f(\widetilde{\varphi}_{\varepsilon})\rightarrow
f(\varphi_{0})$ a.e. in $Q_{1}$. Now, the uniform bound $\left\Vert
f(\varphi_{\varepsilon})\right\Vert _{L^{1}(Q_{\varepsilon})}\leq
C\varepsilon$ yields $\left\Vert f(\widetilde{\varphi}_{\varepsilon
})\right\Vert _{L^{1}(Q_{1})}\leq C$ for all $\varepsilon>0$. The Lebesgue
dominated convergence theorem leads us to
\[
f(\widetilde{\varphi}_{\varepsilon})\rightarrow f(\varphi_{0})\text{ in }%
L^{1}(Q_{1})\text{-strong.}%
\]
Thus, setting $x_{d}=\varepsilon\zeta$ with $\zeta\in(-1,1)$, we have
\[
\frac{1}{\varepsilon}\int_{Q_{\varepsilon}}f(\varphi_{\varepsilon}%
)\chi_{\varepsilon}dxdt=\frac{1}{\varepsilon}\int_{Q_{\varepsilon}}%
f(\varphi_{\varepsilon})\chi_{0}dxdt+\int_{Q_{\varepsilon}}f(\varphi
_{\varepsilon})\chi_{1}(t,\overline{x},\frac{x}{\varepsilon})dxdt,
\]
and
\begin{align*}
\frac{1}{\varepsilon}\int_{Q_{\varepsilon}}f(\varphi_{\varepsilon})\chi
_{0}dxdt  & =\int_{Q_{1}}f(\widetilde{\varphi}_{\varepsilon}(t,\overline
{x},\zeta))\chi_{0}(t,\overline{x})d\overline{x}d\zeta dt\\
& \\
& \rightarrow\int_{Q_{1}}f(\varphi_{0})\chi_{0}d\overline{x}d\zeta dt=\int
_{Q}\int_{I}f(\varphi_{0})\chi_{0}d\overline{x}d\zeta dt.
\end{align*}
Likewise we have
\begin{align*}
\int_{Q_{\varepsilon}}f(\varphi_{\varepsilon})\chi_{1}(t,\overline{x},\frac
{x}{\varepsilon})dxdt  & =\varepsilon\int_{Q_{1}}f(\widetilde{\varphi
}_{\varepsilon})\chi_{1}(t,\overline{x},\frac{\overline{x}}{\varepsilon}%
,\zeta)d\overline{x}d\zeta dt\\
& \rightarrow0\text{ as }E^{\prime}\ni\varepsilon\rightarrow0\text{.}%
\end{align*}
The convergence result (\ref{4.21}) is therefore proved.

With this in mind, we pass to the limit in (\ref{4.17}) and get (\ref{4.13}).
Finally, since $\boldsymbol{u}_{0}^{\varepsilon}\rightarrow\boldsymbol{u}^{0}$
in $L^{2}(\Omega_{\varepsilon})^{d}$-strong $\Sigma_{A}$ and $\varphi
_{0}^{\varepsilon}\rightarrow\varphi^{0}$ in $L^{2}(\Omega_{\varepsilon}%
)$-strong $\Sigma_{A}$, we conclude by integration by parts that
$\boldsymbol{u}_{0}(0)=\boldsymbol{u}^{0}$ and $\varphi_{0}(0)=\varphi^{0}$.
Let us note that from (\ref{4.10'''}) we get $\boldsymbol{u}^{0}%
=(\overline{\boldsymbol{u}}^{0},0)$ as the last component $u_{d}$ of
$\boldsymbol{u}_{0}$ is zero. This concludes the proof of the proposition.
\end{proof}

\subsection{Derivation of the homogenized system}

Our goal in this subsection is to find the equivalent problem whose
$(\overline{\boldsymbol{u}},\varphi_{0},\mu_{0},p_{0})$ is solution to. We
recall that $\overline{\boldsymbol{u}}$ is defined by (\ref{4.10'}) and
satisfies (\ref{4.10''}). To that end, we first consider (\ref{4.13}); it is
equivalent to the system consisting of (\ref{4.23}) and (\ref{4.24}) below:
\begin{equation}
\left\{
\begin{array}
[c]{l}%
\dfrac{1}{2}%
{\displaystyle\int_{Q}}
{\displaystyle\int_{I}}
M(\mu_{0}\chi_{0})d\zeta d\overline{x}dt=\dfrac{\beta}{2}%
{\displaystyle\int_{Q}}
{\displaystyle\int_{I}}
M\left(  (\nabla_{\overline{x}}\varphi_{0}+\nabla_{y}\varphi_{1})\cdot
\nabla_{\overline{x}}\chi_{0}\right)  d\zeta d\overline{x}dt\\
\\
\ \ \ +\dfrac{\lambda}{2}%
{\displaystyle\int_{Q}}
{\displaystyle\int_{I}}
f(\varphi_{0})\chi_{0}d\zeta d\overline{x}dt\text{ for all }\chi_{0}%
\in\mathcal{C}_{0}^{\infty}(Q);
\end{array}
\right. \label{4.23}%
\end{equation}%
\begin{equation}
\int_{Q}\int_{I}M\left(  (\nabla_{\overline{x}}\varphi_{0}+\nabla_{y}%
\varphi_{1})\cdot\nabla_{y}\chi_{1}\right)  d\zeta d\overline{x}dt=0\text{,
all }\chi_{1}\in\mathcal{C}_{0}^{\infty}(Q)\otimes A^{\infty}(\mathbb{R}%
^{d-1};\mathcal{C}_{0}^{\infty}(I)).\label{4.24}%
\end{equation}
In (\ref{4.24}) we take $\chi_{1}$ under the form $\chi_{1}(t,\overline
{x},y)=\chi_{1}^{0}(t,\overline{x})\theta(y)$ with $\chi_{1}^{0}\in
\mathcal{C}_{0}^{\infty}(Q)$ and $\theta\in A^{\infty}(\mathbb{R}%
^{d-1};\mathcal{C}_{0}^{\infty}(I))$. Then (\ref{4.24}) becomes
\begin{equation}
\int_{I}M\left(  (\nabla_{\overline{x}}\varphi_{0}+\nabla_{y}\varphi_{1}%
)\cdot\nabla_{y}\theta\right)  d\zeta=0\ \ \forall\theta\in A^{\infty
}(\mathbb{R}^{d-1};\mathcal{C}_{0}^{\infty}(I)),\label{4.26'}%
\end{equation}
or equivalently,
\begin{equation}
\int_{I}M\left(  \nabla_{y}\varphi_{1}\cdot\nabla_{y}\theta\right)
d\zeta=0\ \ \forall\theta\in A^{\infty}(\mathbb{R}^{d-1};\mathcal{C}%
_{0}^{\infty}(I))\label{4.27'}%
\end{equation}
since
\begin{align*}
\int_{I}M&\left(  \nabla_{\overline{x}}\varphi_{0}\cdot\nabla_{y}\theta\right)
d\zeta  =\int_{I}M\left(  \nabla_{\overline{x}}\varphi_{0}\cdot
\nabla_{\overline{y}}\theta\right)  d\zeta\text{ (recall that }\nabla
_{\overline{x}}\varphi_{0}\equiv(\nabla_{\overline{x}}\varphi_{0},0)\text{)}\\
& =\int_{I}\nabla_{\overline{x}}\varphi_{0}\cdot M\left(  \nabla_{\overline
{y}}\theta\right)  d\zeta=0\text{ as }M\left(  \nabla_{\overline{y}}%
\theta\right)  =0\text{ (recall that }\theta(\cdot,\zeta)\in A^{\infty
}\text{).}%
\end{align*}
Now it is a fact that (\ref{4.27'}) possesses a unique solution $\varphi
_{1}\equiv0$.

This being so, going back to (\ref{4.23}), we readily see that it is the
variational form of the following equation
\begin{equation}
\mu_{0}=-\beta\Delta_{\overline{x}}\varphi_{0}+\lambda f(\varphi_{0})\text{ in
}Q.\ \ \ \ \ \ \ \ \ \ \ \ \ \ \ \ \ \ \ \ \ \label{4.25}%
\end{equation}

Next, we consider (\ref{4.12}) and choose there $\phi_{0}=0$ and take
$\phi_{1}$ under the form $\phi_{1}(t,\overline{x},y)=\phi_{1}^{0}%
(t,\overline{x})\theta(y)$ with $\phi_{1}^{0}\in A^{\infty}(\mathbb{R}%
^{d-1};\mathcal{C}_{0}^{\infty}(I))$. Then we obtain
\begin{equation}
\left\{
\begin{array}
[c]{l}%
-%
{\displaystyle\int_{I}}
M(\varphi_{0}\boldsymbol{u}_{0}\cdot\nabla_{y}\theta)d\zeta+%
{\displaystyle\int_{I}}
M((\nabla_{\overline{x}}\mu_{0}+\nabla_{y}\mu_{1})\cdot\nabla_{y}\theta
)d\zeta=0\\
\text{for all }\theta\in A^{\infty}(\mathbb{R}^{d-1};\mathcal{C}_{0}^{\infty
}(I)).
\end{array}
\right. \label{4.28''}%
\end{equation}
But
\begin{align*}%
{\displaystyle\int_{I}}
M(\varphi_{0}\boldsymbol{u}_{0}\cdot\nabla_{y}\theta)d\zeta & =%
{\displaystyle\int_{I}}
M(\varphi_{0}\overline{\operatorname{div}}_{y}(\boldsymbol{u}_{0}%
\theta))d\zeta\text{ since }\overline{\operatorname{div}}_{y}\boldsymbol{u}%
_{0}=0\\
& =0\text{ because }\varphi_{0}\text{ does not depend on }y.
\end{align*}
Therefore, (\ref{4.28''}) has the same form like (\ref{4.26'}), and since
$\mu_{0}$ is independent of $y$, we deduce as for (\ref{4.26'}) that $\mu
_{1}=0$.

Taking into account the equality $\operatorname{div}_{\overline{x}%
}\boldsymbol{u}_{0}=0$, we see that (\ref{4.12}) (in which we choose $\phi
_{1}=0$) is the variational form of
\begin{equation}
\frac{\partial\varphi_{0}}{\partial t}+\overline{\boldsymbol{u}}\cdot
\nabla_{\overline{x}}\varphi_{0}-\Delta_{\overline{x}}\mu_{0}=0\text{ in
}Q,\ \ \ \ \ \ \ \ \ \ \label{4.26}%
\end{equation}
where once again we recall that $\overline{\boldsymbol{u}}$ is defined by
(\ref{4.10'}).

Let us move to (\ref{4.11}). It is equivalent to: there exists $p_{1}\in
L^{2}(Q;\mathcal{B}_{A}^{2}(\mathbb{R}^{d-1};L^{2}(I)))$ such that
\begin{equation}
\frac{\partial\boldsymbol{u}_{0}}{\partial t}-\alpha\overline{\Delta}%
_{y}\boldsymbol{u}_{0}+\overline{\nabla}_{y}p_{1}=\boldsymbol{h}%
-\nabla_{\overline{x}}p_{0}+\mu_{0}\nabla_{\overline{x}}\varphi_{0}\text{ in
}Q\times\mathbb{R}^{d-1}\times I.\label{4.27}%
\end{equation}
The existence of $p_{1}$ is provided by Proposition 2.1 in Ref. \cite{JW2022}. To
analyze (\ref{4.27}), let $\omega^{j}=(\omega_{i}^{j})_{1\leq i\leq d}%
\in\mathcal{C}([0,T];\mathcal{B}_{A}^{2}(\mathbb{R}^{d-1};L^{2}(I))^{d})\cap
L^{2}(0,T;\mathcal{B}_{A}^{1,2}(\mathbb{R}^{d-1};H_{0}^{1}(I))^{d})$
satisfying (see Propositions \ref{p4.1} and \ref{p4.2'}) the following
auxiliary problem
\begin{equation}
\left\{
\begin{array}
[c]{l}%
\dfrac{\partial\omega^{j}}{\partial t}-\alpha\overline{\Delta}_{y}\omega
^{j}+\overline{\nabla}_{y}\pi^{j}=0\text{ in }(0,T)\times\mathbb{R}%
^{d-1}\times I,\\
\\
\overline{\operatorname{div}}_{y}\omega^{j}=0\text{ in }(0,T)\times
\mathbb{R}^{d-1}\times I,\\
\\
\omega^{j}(0)=e_{j}\text{ in }\mathbb{R}^{d-1}\times I,\
{\displaystyle\int_{I}}
M(\omega_{d}^{j}(t,\cdot,\zeta))d\zeta=0,
\end{array}
\right. \label{4.28}%
\end{equation}
where $e_{j}$ ($1\leq j\leq d-1$) is the $j$th vector of the canonical basis
in $\mathbb{R}^{d}$. As in the previous subsection, we define
\begin{equation}
G_{ij}(t)=\frac{1}{2}\int_{-1}^{1}M(\omega^{i}(t,\cdot,\zeta))e_{j}%
d\zeta,\ \ t\in\lbrack0,T],\ 1\leq i,j\leq d-1,\label{4.28'}%
\end{equation}
and set $G=(G_{ij})_{1\leq i,j\leq d-1}$. As seen in Proposition \ref{p4.2'},
$G$ is a $(d-1)\times(d-1)$ symmetric positive definite matrix. We fix
$(t,\overline{x})\in Q$ and we take $v(\tau,y)=\boldsymbol{u}_{0}%
(t-\tau,\overline{x},y)$ ($(\tau,y)\in(0,t)\times\mathbb{R}^{d-1}\times I$) as
test function in (\ref{4.28}):
\[
\left\langle \frac{\partial\omega^{j}}{\partial\tau}(\tau),\boldsymbol{u}%
_{0}(t-\tau)\right\rangle +\frac{\alpha}{2}\int_{-1}^{1}M(\overline{\nabla
}\omega^{j}(\tau)\cdot\overline{\nabla}\boldsymbol{u}_{0}(t-\tau))d\zeta=0,
\]
or equivalently,
\[%
\begin{array}
[c]{l}%
\dfrac{1}{2}\dfrac{d}{d\tau}%
{\displaystyle\int_{-1}^{1}}
M(\omega^{j}(\tau)\boldsymbol{u}_{0}(t-\tau))d\zeta+\left\langle
\dfrac{\partial\boldsymbol{u}_{0}}{\partial\tau}(t-\tau),\omega^{j}%
(\tau)\right\rangle \\
\\
\ \ \ \ +\dfrac{\alpha}{2}%
{\displaystyle\int_{-1}^{1}}
M(\overline{\nabla}\omega^{j}(\tau)\cdot\overline{\nabla}\boldsymbol{u}%
_{0}(t-\tau))d\zeta=0.
\end{array}
\]
Integrating over $(0,t)$ the last equality above, we obtain
\begin{equation}%
\begin{array}
[c]{l}%
\dfrac{1}{2}%
{\displaystyle\int_{-1}^{1}}
M(\omega^{j}(t)\boldsymbol{u}_{0}(0))d\zeta-\dfrac{1}{2}%
{\displaystyle\int_{-1}^{1}}
M(\boldsymbol{u}_{0}(t)e_{j})d\zeta+\dfrac{1}{2}%
{\displaystyle\int_{0}^{t}}
\left\langle \dfrac{\partial\boldsymbol{u}_{0}}{\partial\tau}(t-\tau
),\omega^{j}(\tau)\right\rangle d\tau\\
\\
\ \ \ +\dfrac{\alpha}{2}%
{\displaystyle\int_{0}^{t}}
{\displaystyle\int_{-1}^{1}}
M(\overline{\nabla}\omega^{j}(\tau)\cdot\overline{\nabla}\boldsymbol{u}%
_{0}(t-\tau))d\zeta d\tau=0,
\end{array}
\label{4.29}%
\end{equation}
where the brackets $\left\langle ,\right\rangle $ denote the duality pairings
between $\left[  \mathcal{B}_{A}^{1,2}(\mathbb{R}^{d-1};H_{0}^{1}%
(I))^{d}\right]  ^{\prime}$ and $\mathcal{B}_{A}^{1,2}(\mathbb{R}^{d-1}%
;H_{0}^{1}(I))^{d}$.

Next we go back to the variational form of (\ref{4.27}) and multiply it by the
function
\[
\Psi(\tau,\overline{x},y)=\varphi(\overline{x})\omega^{j}(t-\tau,y)\text{ with
}\varphi\in\mathcal{C}_{0}^{\infty}(\Omega)\text{,}%
\]
and next integrate over $(0,t)$. Then we obtain the following equality, which
holds in the sense of distributions in $\Omega$:
\begin{equation}%
\begin{array}
[c]{l}%
\dfrac{1}{2}%
{\displaystyle\int_{0}^{t}}
\left\langle \dfrac{\partial\boldsymbol{u}_{0}}{\partial\tau}(\tau),\omega
^{j}(t-\tau)\right\rangle d\tau+\dfrac{\alpha}{2}%
{\displaystyle\int_{0}^{t}}
{\displaystyle\int_{-1}^{1}}
M(\overline{\nabla}\boldsymbol{u}_{0}(\tau)\cdot\overline{\nabla}\omega
^{j}(t-\tau))d\zeta d\tau\\
\\
\ \ -\dfrac{1}{2}%
{\displaystyle\int_{0}^{t}}
{\displaystyle\int_{-1}^{1}}
\mu_{0}(\tau)\nabla_{\overline{x}}\varphi_{0}(\tau)M(\omega^{j}(t-\tau))d\zeta
d\tau\\
\\
\ \ \ \ +\dfrac{1}{2}%
{\displaystyle\int_{0}^{t}}
{\displaystyle\int_{-1}^{1}}
\nabla_{\overline{x}}p_{0}(\tau)M(\omega^{j}(t-\tau))d\zeta d\tau\\
\\
\ \ \ \ \ \ =\dfrac{1}{2}%
{\displaystyle\int_{0}^{t}}
{\displaystyle\int_{-1}^{1}}
M(\omega^{j}(t-\tau))\boldsymbol{h}(\tau)d\zeta d\tau.
\end{array}
\label{4.30}%
\end{equation}
But
\[
\int_{0}^{t}\left\langle \frac{\partial\boldsymbol{u}_{0}}{\partial\tau}%
(\tau),\omega^{j}(t-\tau)\right\rangle d\tau=\int_{0}^{t}\left\langle
\frac{\partial\boldsymbol{u}_{0}}{\partial\tau}(t-\tau),\omega^{j}%
(\tau)\right\rangle d\tau,
\]
so that, comparing (\ref{4.29}) and (\ref{4.30}), we are led to
\[%
\begin{array}
[c]{l}%
-\dfrac{1}{2}%
{\displaystyle\int_{-1}^{1}}
M(\omega^{j}(t))\boldsymbol{u}^{0}d\zeta+\dfrac{1}{2}%
{\displaystyle\int_{-1}^{1}}
M(\boldsymbol{u}_{0}(t))e_{j}d\zeta+\dfrac{1}{2}%
{\displaystyle\int_{0}^{t}}
{\displaystyle\int_{-1}^{1}}
M(\omega^{j}(t-\tau))\nabla_{\overline{x}}p_{0}(\tau)d\tau d\zeta\\
\\
\ \ -\dfrac{1}{2}%
{\displaystyle\int_{0}^{t}}
{\displaystyle\int_{-1}^{1}}
\mu_{0}(\tau)\nabla_{\overline{x}}\varphi_{0}(\tau)M(\omega^{j}(t-\tau))d\zeta
d\tau=\dfrac{1}{2}%
{\displaystyle\int_{0}^{t}}
{\displaystyle\int_{-1}^{1}}
M(\omega^{j}(t-\tau))\boldsymbol{h}(\tau)d\zeta d\tau,
\end{array}
\]
i.e.
\[%
\begin{array}
[c]{l}%
-G_{j}(t)\boldsymbol{u}^{0}+u_{j}(t)+(G_{j}\ast\nabla_{\overline{x}}%
p_{0})(t)-(G_{j}\ast\mu_{0}\nabla_{\overline{x}}\varphi_{0})(t)\\
\\
\ \ =(G_{j}\ast\boldsymbol{h}_{1})(t),\ \ 1\leq j\leq d-1,
\end{array}
\]
or,
\begin{equation}
\overline{\boldsymbol{u}}(t)=G(t)\overline{\boldsymbol{u}}^{0}+(G\ast
(\boldsymbol{h}_{1}-\nabla_{\overline{x}}p_{0}+\mu_{0}\nabla_{\overline{x}%
}\varphi_{0}))(t)\text{ in }\Omega\text{, }t\in\lbrack0,T],\label{4.31}%
\end{equation}
where $G=(G_{j})_{1\leq j\leq d-1}$.

We have just proved the following result.

\begin{theorem}
\label{t4.1}The quadruplet $(\overline{\boldsymbol{u}},\varphi_{0},\mu
_{0},p_{0})$ defined by \emph{(\ref{4.10'})}, \emph{(\ref{4.7})},
\emph{(\ref{4.9})} and \emph{(\ref{4.6})} solves in the weak sense the
homogenized system \emph{(\ref{4.31})}, \emph{(\ref{4.26})}, \emph{(\ref{4.25}%
)} with appropriate boundary and initial conditions, viz.
\begin{equation}
\left\{
\begin{array}
[c]{l}%
\overline{\boldsymbol{u}}=G\overline{\boldsymbol{u}}^{0}+G\ast(\boldsymbol{h}%
_{1}+\mu_{0}\nabla_{\overline{x}}\varphi_{0}-\nabla_{\overline{x}}p_{0})\text{
in }Q,\\
\\
\operatorname{div}_{\overline{x}}\overline{\boldsymbol{u}}=0\text{ in }Q\text{
and }\overline{\boldsymbol{u}}\cdot\boldsymbol{n}=0\text{ on }(0,T)\times
\partial\Omega,\\
\\
\dfrac{\partial\varphi_{0}}{\partial t}+\overline{\boldsymbol{u}}\cdot
\nabla_{\overline{x}}\varphi_{0}-\Delta_{\overline{x}}\mu_{0}=0\text{ in }Q,\\
\\
\mu_{0}=-\beta\Delta_{\overline{x}}\varphi_{0}+\lambda f(\varphi_{0})\text{ in
}Q,\\
\\
\dfrac{\partial\varphi_{0}}{\partial\boldsymbol{n}}=\dfrac{\partial\mu_{0}%
}{\partial\boldsymbol{n}}=0\text{ on }(0,T)\times\partial\Omega,\\
\\
\varphi_{0}(0)=\varphi^{0}\text{ in }\Omega.
\end{array}
\right.  \ \ \ \ \ \ \ \ \ \ \ \ \ \ \ \ \ \ \ \ \ \ \ \ \ \ \ \ \label{4.33}%
\end{equation}

\end{theorem}

The equation (\ref{4.33})$_{1}$ is a Hele-Shaw equation with memory, that is,
a non-local (in time) Hele-Shaw equation. Thus, system (\ref{4.33}) is a
\textit{non-local Hele-Shaw-Cahn-Hilliard} (HSCH) system arising from
transient flow through thin domains, and modeling in particular tumor growth.
To the best of our knowledge, this is the first time that such a system is
obtained in the literature. For that reason, we need to make a qualitative
analysis of (\ref{4.33}) in order to prove some regularity results and its
well-posedness. This is the aim of the next section.

\section{Proof of the main results\label{sec5}}

\subsection{Analysis of the homogenized system: Proof of Theorem \ref{t1.1}}

In this subsection, we are concerned with the $2D$ non-local HSCH system
(\ref{4.33}) derived from the upscaling of the $\varepsilon$-model (\ref{1.1})
in $3D$.

\subsubsection{\textbf{Well-posedness of the homogenized system}}

We aim at proving the well-posedness of the system (\ref{4.33}). This will
give rise to the proof of the main result of the work. We start with some
basic estimates. To that end, we shall need the following Gronwall-type
inequality. We recall that, throughout this section $\Omega$ is a bounded
Lipschitz domain in $\mathbb{R}^{2}$.

\begin{lemma}
({see p. 384 in Ref. \cite{MPF1994}})\label{l5.1} Let $u$, $v$ and $h$ be nonnegative
functions, and $c_{1}$, $c_{2}$ be nonnegative constants. If
\[
u(t)\leq c_{1}+c_{2}\int_{0}^{t}\left[  v(s)u(s)+\int_{0}^{s}%
h(s,r)u(r)dr\right]  ds,\ \ t\geq0,
\]
then for any $t\geq0$,
\[
u(t)\leq c_{1}\exp\left[  c_{2}\int_{0}^{t}\left(  v(s)+\int_{0}%
^{s}h(s,r)dr\right)  ds\right]  .
\]

\end{lemma}

We also gather below some classical results, namely the Agmon and
Gagliardo-Nirenberg inequalities in $2$ space dimensions.

\begin{lemma}
(see Ref. \cite{Temam2001})\label{l5.2} Let $\Omega$ be a bounded $\mathcal{C}^{4}%
$-domain in $\mathbb{R}^{2}$. Then

\begin{itemize}
\item[(i)] $\left\Vert f\right\Vert _{L^{4}}\leq C(\left\Vert f\right\Vert
_{L^{2}}^{1/2}\left\Vert \nabla f\right\Vert _{L^{2}}^{1/2}+\left\Vert
f\right\Vert _{L^{2}})$ for any $f\in H^{1}(\Omega)$,

\item[(ii)] $\left\Vert f\right\Vert _{L^{p}}\leq C\left\Vert f\right\Vert
_{H^{1}}$ for any $1\leq p<\infty$ and for any $f\in H^{1}(\Omega)$,

\item[(iii)] $\left\Vert f\right\Vert _{L^{\infty}}\leq C\left\Vert
f\right\Vert _{L^{2}}^{1/2}\left\Vert f\right\Vert _{H^{2}}^{1/2}$ for any
$f\in H^{2}(\Omega)$,

\item[(iv)] $\left\Vert
f-\mathchoice {{\setbox0=\hbox{$\displaystyle{\textstyle
-}{\int}$ } \vcenter{\hbox{$\textstyle -$
}}\kern-.6\wd0}}{{\setbox0=\hbox{$\textstyle{\scriptstyle -}{\int}$ } \vcenter{\hbox{$\scriptstyle -$
}}\kern-.6\wd0}}{{\setbox0=\hbox{$\scriptstyle{\scriptscriptstyle -}{\int}$
} \vcenter{\hbox{$\scriptscriptstyle -$
}}\kern-.6\wd0}}{{\setbox0=\hbox{$\scriptscriptstyle{\scriptscriptstyle
-}{\int}$ } \vcenter{\hbox{$\scriptscriptstyle -$ }}\kern-.6\wd0}}\!\int
_{\Omega}f\right\Vert _{H^{2}}\leq C\left\Vert \Delta f\right\Vert _{L^{2}}$
for any $f\in H^{2}(\Omega)$ with $\nabla f\cdot\boldsymbol{n}=0$ on
$\partial\Omega$,

\item[(v)] $\left\Vert f\right\Vert _{H^{3}}\leq C(\left\Vert \nabla\Delta
f\right\Vert _{L^{2}}+\left\Vert f\right\Vert _{L^{2}})$ for any $f\in
H^{3}(\Omega)$,

\item[(vi)] $\left\Vert f\right\Vert _{H^{2}}\leq C(\left\Vert \Delta
f\right\Vert _{L^{2}}+\left\Vert f\right\Vert _{L^{2}})$ for any $f\in
H^{2}(\Omega)$,
\end{itemize}

\noindent where $C=C(p,\Omega)>0$.
\end{lemma}

\begin{remark}
\label{r5.1}\emph{Putting together (iii) and (iv) of Lemma \ref{l5.2}, we
obtain }%
\begin{equation}
\left\Vert f\right\Vert _{L^{\infty}}\leq C\left\Vert f\right\Vert _{L^{2}%
}^{\frac{1}{2}}\left\Vert \Delta f\right\Vert _{L^{2}}^{\frac{1}{2}%
}\text{\emph{\ for any }}f\in H^{2}(\Omega)\text{\emph{\ with }}\nabla
f\cdot\boldsymbol{n}=0\text{\emph{\ on }}\partial\Omega\text{\emph{\ and }%
}\mathchoice {{\setbox0=\hbox{$\displaystyle{\textstyle
-}{\int}$ } \vcenter{\hbox{$\textstyle -$
}}\kern-.6\wd0}}{{\setbox0=\hbox{$\textstyle{\scriptstyle -}{\int}$ } \vcenter{\hbox{$\scriptstyle -$
}}\kern-.6\wd0}}{{\setbox0=\hbox{$\scriptstyle{\scriptscriptstyle -}{\int}$
} \vcenter{\hbox{$\scriptscriptstyle -$
}}\kern-.6\wd0}}{{\setbox0=\hbox{$\scriptscriptstyle{\scriptscriptstyle
-}{\int}$ } \vcenter{\hbox{$\scriptscriptstyle -$ }}\kern-.6\wd0}}\!\int
_{\Omega}f=0.\label{e5.0}%
\end{equation}
\emph{where }$C=C(\Omega)>0$\emph{.}
\end{remark}

Before proceeding further, let us recall the statement of (\ref{4.33}) below.
We drop the subscripts on the unknown functions and we assume without loss of
generality that $\beta=\lambda=1$. Then (\ref{4.33}) therefore reads as
follows
\begin{equation}
\left\{
\begin{array}
[c]{l}%
\boldsymbol{u}=G\boldsymbol{u}^{0}+G\ast(\boldsymbol{h}_{1}+\mu\nabla
\varphi-\nabla p)\text{ in }Q,\\
\\
\operatorname{div}\boldsymbol{u}=0\text{ in }Q\text{ and }\boldsymbol{u}%
\cdot\boldsymbol{n}=0\text{ on }(0,T)\times\partial\Omega,\\
\\
\dfrac{\partial\varphi}{\partial t}+\boldsymbol{u}\cdot\nabla\varphi-\Delta
\mu=0\text{ in }Q,\\
\\
\mu=-\Delta\varphi+f(\varphi)\text{ in }Q,\\
\\
\dfrac{\partial\varphi}{\partial\boldsymbol{n}}=\dfrac{\partial\mu}%
{\partial\boldsymbol{n}}=0\text{ on }(0,T)\times\partial\Omega,\\
\\
\varphi(0)=\varphi^{0}\text{ in }\Omega.
\end{array}
\right.  \ \ \ \ \ \ \ \ \ \ \ \ \ \ \ \ \ \ \ \ \ \ \ \ \ \ \ \ \label{5.1}%
\end{equation}
In (\ref{5.1}), $\boldsymbol{n}$ denotes the outward unit normal to
$\partial\Omega$. We know from the homogenization process that there exists at
least a quadruple $(\boldsymbol{u},\varphi,\mu,p)$ solving (\ref{5.1}) such
that $\boldsymbol{u}\in L^{2}(0,T;\mathbb{H})$, $\varphi\in L^{\infty
}(0,T;H^{1}(\Omega))$, $\mu\in L^{2}(0,T;H^{1}(\Omega))$ and $p\in
L^{2}(0,T;L_{0}^{2}(\Omega))$, where
\[
\mathbb{H}=\{\boldsymbol{u}\in L^{2}(\Omega)^{2}:\operatorname{div}%
\boldsymbol{u}=0\text{ in }\Omega\text{ and }\boldsymbol{u}\cdot
\boldsymbol{n}=0\text{ on }\partial\Omega\}.
\]
Our first goal here is to improve the regularity on $\varphi$, $\boldsymbol{u}%
$, $\mu$ and $p$. We start with the following result.

\begin{lemma}
\label{l5.3}The order parameter $\varphi$ in \emph{(\ref{5.1})} satisfies
$\varphi\in\mathcal{C}([0,T];H^{1}(\Omega))\cap L^{4}(0,T;H^{2}(\Omega))\cap
L^{2}(0,T;H^{3}(\Omega))$.
\end{lemma}

\begin{proof}
First of all, we infer from (\ref{5.1})$_{4}$-(\ref{5.1})$_{5}$ that
$\varphi(t)$ (for a.e. $t\in(0,T)$) solves the Neumann problem
\begin{equation}
-\Delta\varphi=\mu-f(\varphi)\text{ in }\Omega\text{, }\frac{\partial\varphi
}{\partial\boldsymbol{n}}=0\text{ on }\partial\Omega.\label{6.0}%
\end{equation}
Since $\mu\in L^{2}(0,T;H^{1}(\Omega))$, we have that $\mu(t)\in H^{1}%
(\Omega)$ for a.e. $t\in(0,T)$. Next, because of (\ref{1.7}), it holds that
$f(\varphi(t))\in L^{2}(\Omega)$ for a.e. $t\in(0,T)$. Indeed, one has
\[
\int_{\Omega}\left\vert f(\varphi(t))\right\vert ^{2}dx\leq C\int_{\Omega
}(1+\left\vert \varphi(t)\right\vert ^{6})dx,
\]
so that the continuous embedding $H^{1}(\Omega)\hookrightarrow L^{6}(\Omega)$
yields $\left\Vert \varphi(t)\right\Vert _{L^{6}(\Omega)}\leq C\left\Vert
\varphi(t)\right\Vert _{H^{1}(\Omega)}$, and hence
\[
\int_{\Omega}\left\vert f(\varphi(t))\right\vert ^{2}dx\leq C+C\left\Vert
\varphi(t)\right\Vert _{H^{1}(\Omega)}^{6}.
\]
Thus $f(\varphi)\in L^{\infty}(0,T;L^{2}(\Omega))$. Therefore $\mu
(t)-f(\varphi(t))\in L^{2}(\Omega)$, a.e. $t\in(0,T)$. By a classical
regularity result, we get $\varphi(t)\in H^{2}(\Omega)$, so that $\varphi\in
L^{2}(0,T;H^{2}(\Omega))$. Next, the continuous Sobolev embedding
$H^{2}(\Omega)\hookrightarrow L^{\infty}(\Omega)$ yields $\varphi\in
L^{2}(0,T;L^{\infty}(\Omega))$, in such a way that, still from (\ref{1.7}), we
have $f(\varphi)\in L^{2}(0,T;H^{1}(\Omega))$. We infer that $\varphi\in
L^{2}(0,T;H^{3}(\Omega))$. It follows that $\varphi\in L^{\infty}%
(0,T;H^{1}(\Omega))\cap L^{2}(0,T;H^{3}(\Omega))$, and by (\ref{5.1})$_{3}$,
we have that $\partial\varphi/\partial t\in L^{2}(0,T;H^{1}(\Omega)^{\prime}%
)$; thus it comes that $\varphi\in\mathcal{C}([0,T];H^{1}(\Omega))$.

Now, noticing that since $\varphi\in\mathcal{C}([0,T];H^{1}(\Omega))\cap
L^{2}(0,T;H^{3}(\Omega))$, it follows by interpolation that, for any $q\geq
1$,
\[
\int_{0}^{T}\left\Vert \varphi(t)\right\Vert _{H^{2}}^{q}dt\leq\int_{0}%
^{T}\left\Vert \varphi(t)\right\Vert _{H^{1}}^{q/2}\left\Vert \varphi
(t)\right\Vert _{H^{3}}^{q/2}dt\leq C\int_{0}^{T}\left\Vert \varphi
(t)\right\Vert _{H^{3}}^{q/2}dt,
\]
so that if $q\leq4$, one has $\int_{0}^{T}\left\Vert \varphi(t)\right\Vert
_{H^{2}}^{q}dt\leq C$. In particular we have $\int_{0}^{T}\left\Vert
\varphi(t)\right\Vert _{H^{2}}^{4}dt\leq C$, so that $\varphi\in
L^{4}(0,T;H^{2}(\Omega))$. The proof is completed.
\end{proof}

In the sequel we shall deal with the space $H_{N}^{m}(\Omega)$ (integer
$m\geq1$) defined as
\[
H_{N}^{m}(\Omega)=\{u\in H^{m}(\Omega):\partial u/\partial\boldsymbol{n}%
=0\text{ on }\partial\Omega\}.
\]
It is known that $H_{N}^{2}(\Omega)$ is the domain of the unbounded Laplace
operator in $\Omega$ with homogeneous Neumann boundary condition. This being
so, the next result shows that the weak solution of (\ref{5.1}) is actually a
strong one, provided that $\varphi^{0}\in H_{N}^{2}(\Omega)$. It reads as follows.

\begin{proposition}
\label{p5.1}Let $\boldsymbol{u}^{0}\in\mathbb{H}$, $\varphi^{0}\in H_{N}%
^{2}(\Omega)$ and $T>0$ be given. Then the solution $(\boldsymbol{u}%
,\varphi,\mu,p)$ of \emph{(\ref{5.1})} satisfies $\boldsymbol{u}\in
\mathcal{C}([0,T];\mathbb{H})$, $\varphi\in\mathcal{C}([0,T];H^{2}%
(\Omega))\cap L^{2}(0,T;H^{4}(\Omega))\cap H^{1}(0,T;L^{2}(\Omega))$, $\mu
\in\mathcal{C}([0,T];H^{1}(\Omega))\cap L^{2}(0,T;H^{2}(\Omega))$ and $p\in
L^{2}(0,T;H^{1}(\Omega)\cap L_{0}^{2}(\Omega))$. Furthermore it holds that
\begin{equation}
\left\Vert \Delta\varphi(t)\right\Vert _{L^{2}}^{2}+\int_{0}^{t}\left(
\left\Vert \Delta^{2}\varphi(s)\right\Vert _{L^{2}}^{2}+\left\Vert
\mu(s)\right\Vert _{H^{2}}^{2}+\left\Vert \frac{\partial\varphi}{\partial
t}(s)\right\Vert _{L^{2}}^{2}\right)  ds\leq C,\label{5.16'}%
\end{equation}
all $t\in\lbrack0,T]$, where $C>0$ depends on $\left\Vert \boldsymbol{h}%
_{1}\right\Vert _{L^{2}(Q)}$, $\left\Vert \boldsymbol{u}^{0}\right\Vert
_{L^{2}(\Omega)}$, $\left\Vert \varphi^{0}\right\Vert _{H^{2}(\Omega)}$ and
$T$.
\end{proposition}

\begin{proof}
The proof is done in three steps.

\emph{Step 1}. It is a fact from the definition of $\boldsymbol{u}$ in
(\ref{5.1})$_{1}$ that it belongs to $\mathcal{C}([0,T];\mathbb{H})$ (recall
that $G$ is continuous). Let us check that the pressure $p$ lies in
$L^{2}(0,T;H^{1}(\Omega))$. In order to do that, we need to establish an
estimate on the term $\mu\nabla\varphi$. We first recall that from Lemma
\ref{l5.3}, it holds that
\begin{equation}
\int_{0}^{T}\left\Vert \varphi(t)\right\Vert _{H^{2}}^{4}dt\leq
C.\ \ \ \ \ \ \ \ \ \ \ \ \ \ \ \label{5.16}%
\end{equation}

Now, concerning $\mu\nabla\varphi$, we have, for any $v\in L^{8/3}%
(0,T;\mathbb{H})$,
\[
\int_{\Omega}\mu\nabla\varphi\cdot vdx=\int_{\partial\Omega}(v\cdot
\boldsymbol{n})\varphi\mu d\sigma-\int_{\Omega}\varphi v\nabla\mu
dx=-\int_{\Omega}\varphi v\nabla\mu dx,
\]
so that
\begin{align*}
\left\vert \left\langle \mu\nabla\varphi,v\right\rangle \right\vert  &
=\left\vert \int_{\Omega}\varphi v\nabla\mu dx\right\vert \leq\left\Vert
v\right\Vert _{L^{2}}\left\Vert \nabla\mu\right\Vert _{L^{2}}\left\Vert
\varphi\right\Vert _{L^{\infty}}\\
& \leq\left\Vert v\right\Vert _{L^{2}}\left\Vert \nabla\mu\right\Vert _{L^{2}%
}\left\Vert \varphi\right\Vert _{L^{2}}^{1/2}\left\Vert \varphi\right\Vert
_{H^{2}}^{1/2}\text{ by Agmon's inequality}\\
& \leq C\left\Vert v\right\Vert _{L^{2}}\left\Vert \nabla\mu\right\Vert
_{L^{2}}\left\Vert \varphi\right\Vert _{H^{2}}^{1/2}.
\end{align*}
Making use of (\ref{5.16}), we get
\begin{align*}
\left\vert \int_{0}^{T}\left\langle \mu\nabla\varphi,v\right\rangle
dt\right\vert  & \leq C\left(  \int_{0}^{T}\left\Vert v\right\Vert _{L^{2}%
}^{8/3}dt\right)  ^{3/8}\left(  \int_{0}^{T}\left\Vert \nabla\mu\right\Vert
_{L^{2}}^{2}dt\right)  ^{1/2}\left(  \int_{0}^{T}\left\Vert \varphi\right\Vert
_{H^{2}}^{4}\right)  ^{1/8}\\
& \leq C\left(  \int_{0}^{T}\left\Vert v\right\Vert _{L^{2}}^{8/3}dt\right)
^{3/8}.
\end{align*}
This gives
\begin{equation}
\mu\nabla\varphi\in L^{8/5}(0,T;\mathbb{H}^{\prime}%
).\ \ \ \ \ \ \ \ \ \ \ \ \ \ \ \ \ \ \ \ \ \ \ \ \ \ \ \ \ \ \label{5.17}%
\end{equation}
Owing to (\ref{5.17}) and thanks to the fact that $\boldsymbol{h}_{1}\in
L^{2}(0,T;L^{2}(\Omega)^{2})$, we infer that $\boldsymbol{h}_{1}+\mu
\nabla\varphi\in L^{8/5}(0,T;L^{2}(\Omega)^{2})$. Also $G(t)\boldsymbol{u}%
^{0}\in L^{8/5}(0,T;L^{2}(\Omega)^{2})$. At this level, we proceed as in
Ref. \cite{Mikelic} by using the Laplace transform, which is well defined in
$\mathcal{D}_{+}^{\prime}((0,\infty);L_{0}^{2}(\Omega))$ (see for instance 
Ref. \cite{Vladimirov}, p.p. 158-170): we apply it to (\ref{5.1})$_{1}$ and
(\ref{5.1})$_{2}$ to obtain the following equation
\begin{equation}
\left\{
\begin{array}
[c]{l}%
\operatorname{div}\left(  \widehat{G}(\tau)(\widehat{\boldsymbol{h}}_{1}%
(\tau)+\widehat{\mu\nabla\varphi}(\tau)-\nabla\widehat{p}(\tau))+\widehat
{G}(\tau)\boldsymbol{u}^{0}\right)  =0\text{ in }\Omega,\\
\left(  \widehat{G}(\tau)(\widehat{\boldsymbol{h}}_{1}(\tau)+\widehat
{\mu\nabla\varphi}(\tau)-\nabla\widehat{p}(\tau))+\widehat{G}(\tau
)\boldsymbol{u}^{0}\right)  \cdot\boldsymbol{n}=0\text{ on }\partial\Omega.
\end{array}
\right. \label{5.18}%
\end{equation}
In (\ref{5.18}) the hat $\widehat{}$ stands for the Laplace transform which is
a function of variable $\tau$. We recall that $\widehat{G}(\tau)$ is an
analytic function of $\tau\in\mathbb{C}$ (the complex field) for
$\operatorname{Re}\tau>0$. Also, as $G$ is a symmetric positive definite
$(d-1)\times(d-1)$ matrix, so is $\widehat{G}(\tau)$. Now, since, for any
$\tau\in\mathbb{C}$ with $\operatorname{Re}\tau>0$, the functions $\widehat
{G}(\tau)(\widehat{\boldsymbol{h}}_{1}+\widehat{\mu\nabla\varphi})(\tau)$ and
$\widehat{G}(\tau)\boldsymbol{u}^{0}$ belong to $L^{2}(\Omega)^{2}$, we get
that (\ref{5.18}) possesses a unique solution $\widehat{p}(\tau)$ in
$H^{1}(\Omega)$ for such $\tau$. Therefore $p\in L^{2}(0,T;H^{1}(\Omega)\cap
L_{0}^{2}(\Omega))$.

With the existence of the pressure $p$ as above, let us first estimate
$\left\Vert G\ast p\right\Vert _{L^{2}}$ in terms of the other unknowns. Set
$q=G\ast p$, $\boldsymbol{g}=G\boldsymbol{u}^{0}+G\ast\boldsymbol{h}_{1}$.
Then (\ref{5.1})$_{1}$ and (\ref{5.1})$_{2}$ amount to
\begin{equation}
\left\{
\begin{array}
[c]{l}%
-\Delta q+\operatorname{div}(\boldsymbol{g}+G\ast\mu\nabla\varphi)=0\text{ in
}\Omega\\
\nabla q\cdot\boldsymbol{n}=\left(  \boldsymbol{g}+G\ast\mu\nabla
\varphi\right)  \cdot\boldsymbol{n}\text{ on }\partial\Omega\text{ and }%
\int_{\Omega}qdx=0.
\end{array}
\right. \label{5.2'}%
\end{equation}
We multiply (\ref{5.2'})$_{1}$ by $q$ and integrate by parts to obtain
\[
\left\Vert \nabla q\right\Vert _{L^{2}}^{2}\leq\left\Vert \boldsymbol{g}%
+G\ast\mu\nabla\varphi\right\Vert _{L^{2}}\left\Vert \nabla q\right\Vert
_{L^{2}},
\]
so that
\begin{equation}
\left\Vert \nabla q\right\Vert _{L^{2}}\leq\left\Vert \boldsymbol{g}%
\right\Vert _{L^{2}}+\left\Vert G\ast\mu\nabla\varphi\right\Vert _{L^{2}%
}.\ \ \ \ \ \ \ \ \ \ \ \ \ \ \ \ \ \ \ \ \ \ \ \ \ \ \ \ \ \ \label{5.3'}%
\end{equation}
Now, noticing that $\boldsymbol{u}=\boldsymbol{g}+G\ast\mu\nabla\varphi-\nabla
q$, we see that
\[
\left\vert \boldsymbol{u}\right\vert ^{2}\leq2(\left\vert \boldsymbol{g}%
\right\vert ^{2}+\left\vert G\ast\mu\nabla\varphi\right\vert ^{2}+\left\vert
\nabla q\right\vert ^{2}).
\]

\medskip

\emph{Step 2}. We need to check that $\mu\in\mathcal{C}([0,T];H_{N}^{1}%
(\Omega))\cap L^{2}(0,T;H_{N}^{2}(\Omega))$. To that end, we notice that the
evolution of the potential is governed by the equation (\ref{6.1}) below
\begin{equation}
\frac{\partial\mu}{\partial t}+\Delta^{2}\mu-f^{\prime}(\varphi)\Delta
\mu=-f^{\prime}(\varphi)(\boldsymbol{u}\cdot\nabla\varphi)+\Delta
(\boldsymbol{u}\cdot\nabla\varphi)\text{ in }Q\text{.}\label{6.1}%
\end{equation}
This is obtained by differentiating (in the sense of distributions in $Q$)
formally (\ref{5.1})$_{4}$ with respect to time and taking advantage of
(\ref{5.1})$_{3}$. Letting $H=-f^{\prime}(\varphi)(\boldsymbol{u}\cdot
\nabla\varphi)+\Delta(\boldsymbol{u}\cdot\nabla\varphi)$, it is an easy task,
using the series of equalities
\begin{align*}
\left\langle \Delta(\boldsymbol{u}\cdot\nabla\varphi),\phi\right\rangle  &
=\int_{\Omega}(\boldsymbol{u}\cdot\nabla\varphi)\Delta\phi dx-\int
_{\partial\Omega}\left[  \phi\frac{\partial}{\partial\boldsymbol{n}%
}(\boldsymbol{u}\cdot\nabla\varphi)-(\boldsymbol{u}\cdot\nabla\varphi
)\frac{\partial\phi}{\partial\boldsymbol{n}}\right]  d\sigma\\
& =\int_{\Omega}(\boldsymbol{u}\cdot\nabla\varphi)\Delta\phi dx\text{ for all
}\phi\in H_{N}^{2}(\Omega)
\end{align*}
(recall that $\frac{\partial}{\partial\boldsymbol{n}}(\boldsymbol{u}%
\cdot\nabla\varphi)=\frac{\partial}{\partial\boldsymbol{n}}\left(
-\frac{\partial\varphi}{\partial t}+\Delta\mu\right)  =0$ on $\partial\Omega$)
to see that $H\in L^{2}(0,T;(H_{N}^{2}(\Omega)^{\prime})$. With this in mind,
we observe that $\mu$ solves the equation
\begin{equation}
\left\{
\begin{array}
[c]{l}%
\dfrac{\partial\mu}{\partial t}+\Delta^{2}\mu-f^{\prime}(\varphi)\Delta
\mu=-f^{\prime}(\varphi)(\boldsymbol{u}\cdot\nabla\varphi)+\Delta
(\boldsymbol{u}\cdot\nabla\varphi)\text{ in }Q,\\
\\
\dfrac{\partial\mu}{\partial\boldsymbol{n}}=\dfrac{\partial\Delta\mu}%
{\partial\boldsymbol{n}}=0\text{ on }(0,T)\times\partial\Omega,\\
\\
\mu(0)=\mu^{0}\text{ in }\Omega,
\end{array}
\right. \label{E6.1}%
\end{equation}
where $\mu^{0}=-\Delta\varphi^{0}+f(\varphi^{0})\in L^{2}(\Omega)$ (remind
that $\varphi^{0}\in H_{N}^{2}(\Omega)$). Our aim is to show that (\ref{E6.1})
possesses a unique solution $\mu\in L^{\infty}(0,T;L^{2}(\Omega))\cap
\mathcal{C}([0,T];H_{N}^{1}(\Omega))\cap L^{2}(0,T;H_{N}^{2}(\Omega))$. To
achieve this, we set
\[
\mathcal{B}(u,v)=\int_{\Omega}\left[  (\Delta u)(\Delta v)-f^{\prime}%
(\varphi)(\Delta u)v\right]  dx\text{ for }u,v\in H_{N}^{2}(\Omega).
\]
Then, using the obvious inequality
\[
\left\vert \int_{\Omega}f^{\prime}(\varphi)(\Delta v)vdx\right\vert \leq
\frac{1}{4}\left\Vert \Delta v\right\Vert _{L^{2}(\Omega)}^{2}+\left\Vert
f^{\prime}(\varphi)\right\Vert _{L^{\infty}(Q)}^{2}\left\Vert v\right\Vert
_{L^{2}(\Omega)}^{2},
\]
we get that
\begin{align}
\mathcal{B}(v,v)+\left(  \frac{3}{4}+\left\Vert f^{\prime}(\varphi)\right\Vert
_{L^{\infty}(Q)}^{2}\right)  \left\Vert v\right\Vert _{L^{2}(\Omega)}^{2}  &
\geq\frac{3}{4}\left(  \left\Vert \Delta v\right\Vert _{L^{2}(\Omega)}%
^{2}+\left\Vert v\right\Vert _{L^{2}(\Omega)}^{2}\right) \label{E6.2}\\
& =\frac{3}{4}\left\Vert v\right\Vert _{H_{N}^{2}(\Omega)}^{2}\text{ fot all
}v\in H_{N}^{2}(\Omega),\nonumber
\end{align}
where we have used parts (iii) and (vi) of Lemma \ref{l5.2} to get
respectively that $\frac{3}{4}+\left\Vert f^{\prime}(\varphi)\right\Vert
_{L^{\infty}(Q)}^{2}<\infty$ and the equality of the right-hand side of
(\ref{E6.2}).

It follows from a classical existence result that (\ref{E6.1}) possesses a
unique solution $\mu\in L^{\infty}(0,T;L^{2}(\Omega))\cap L^{2}(0,T;H_{N}%
^{2}(\Omega))$. We also infer from (\ref{E6.1})$_{1}$ that
\[
\left\Vert \frac{\partial\mu}{\partial t}\right\Vert _{L^{2}(0,T,(H_{N}%
^{2}(\Omega))^{\prime})}\leq
C.\ \ \ \ \ \ \ \ \ \ \ \ \ \ \ \ \ \ \ \ \ \ \ \ \ \ \ \ \ \ \ \
\]

This shows that $\mu\in L^{2}(0,T,H_{N}^{2}(\Omega))$ with $\partial
\mu/\partial t\in L^{2}(0,T,(H_{N}^{2}(\Omega))^{\prime})$. Thus $\mu
\in\mathcal{C}([0,T];H_{N}^{1}(\Omega))$ by a classical embedding result.

\medskip

\emph{Step 3}. Let us check (\ref{5.16'}). With Step 2 in mind, if we go back
to (\ref{6.0}) then we notice that assuming there $\varphi^{0}\in H^{2}%
(\Omega)$ gives easily (with the properties of $f$) $\mu-f(\varphi)\in
L^{2}(0,T;H^{2}(\Omega))$, so that, by a classical regularity result, it holds
that $\varphi\in L^{2}(0,T;H^{4}(\Omega))$. This being so, we multiply
(\ref{5.1})$_{3}$ by $\Delta^{2}\varphi$ and use integration by parts to get
\begin{align*}
\frac{1}{2}\frac{d}{dt}\left\Vert \Delta\varphi\right\Vert _{L^{2}}%
^{2}+\left\Vert \Delta^{2}\varphi\right\Vert _{L^{2}}^{2}  & =(\Delta
f(\varphi),\Delta^{2}\varphi)+(\boldsymbol{u}\nabla\varphi,\Delta^{2}%
\varphi)\\
& \leq\frac{1}{4}\left\Vert \Delta^{2}\varphi\right\Vert _{L^{2}}^{2}%
+3\int_{\Omega}(\left\vert \Delta f(\varphi)\right\vert ^{2}+\left\vert
\boldsymbol{u}\right\vert ^{2}\left\vert \nabla\varphi\right\vert ^{2})dx.
\end{align*}

First, we have
\begin{align*}
3\int_{\Omega}\left\vert \boldsymbol{u}\right\vert ^{2}\left\vert
\nabla\varphi\right\vert ^{2}dx  & \leq C\int_{\Omega}(\left\vert
\boldsymbol{g}\right\vert ^{2}+\left\vert G\ast\mu\nabla\varphi\right\vert
^{2}+\left\vert \nabla q\right\vert ^{2})\left\vert \nabla\varphi\right\vert
^{2}dx\\
& \leq C\left(  \left\Vert \boldsymbol{g}\right\Vert _{L^{2}}^{2}+\left\Vert
G\ast\mu\nabla\varphi\right\Vert _{L^{2}}^{2}+\left\Vert \nabla q\right\Vert
_{L^{2}}^{2}\right)  \left\Vert \nabla\varphi\right\Vert _{L^{\infty}}^{2}\\
& \leq C\left(  \left\Vert \boldsymbol{g}\right\Vert _{L^{2}}^{2}+\left\Vert
G\ast\mu\nabla\varphi\right\Vert _{L^{2}}^{2}\right)  \left\Vert \nabla
\varphi\right\Vert _{L^{\infty}}^{2}\text{ by (\ref{5.3'}).}%
\end{align*}
But
\begin{align*}
\left\Vert G\ast\mu\nabla\varphi\right\Vert _{L^{2}}^{2}  & =\int_{\Omega
}\left\vert G\ast\mu\nabla\varphi\right\vert ^{2}dx=\int_{\Omega}\left\vert
\int_{0}^{t}G(t-\tau)\mu(\tau)\nabla\varphi(\tau)d\tau\right\vert ^{2}dx\\
& \leq\int_{\Omega}\left(  \left[  \int_{0}^{t}\left\vert G(t-\tau)\right\vert
^{2}d\tau\right]  ^{\frac{1}{2}}\left[  \int_{0}^{t}\left\vert \mu
(\tau)\right\vert ^{2}\left\vert \nabla\varphi(\tau)\right\vert ^{2}%
d\tau\right]  ^{\frac{1}{2}}\right)  ^{2}dx\\
& \leq\int_{0}^{\infty}\left\vert G(\tau)\right\vert ^{2}d\tau\int_{0}%
^{t}\left\Vert \mu(\tau)\right\Vert _{L^{2}}^{2}\left\Vert \nabla\varphi
(\tau)\right\Vert _{L^{\infty}}^{2}d\tau\\
& \leq C\int_{0}^{t}\left\Vert \mu(\tau)\right\Vert _{L^{2}}^{2}\left\Vert
\nabla\varphi(\tau)\right\Vert _{L^{\infty}}^{2}d\tau.
\end{align*}
Now, we use Agmon's inequality for $\nabla\varphi$ to obtain
\begin{align*}
\left\Vert \mu(\tau)\right\Vert _{L^{2}}^{2}\left\Vert \nabla\varphi
(\tau)\right\Vert _{L^{\infty}}^{2}  & \leq C\left\Vert \varphi(\tau
)\right\Vert _{H^{1}}\left\Vert \varphi(\tau)\right\Vert _{H^{3}}\left\Vert
\mu(\tau)\right\Vert _{L^{2}}^{2}\\
& \leq C(1+\left\Vert \nabla\Delta\varphi(\tau)\right\Vert _{L^{2}%
})(1+\left\Vert \Delta\varphi(\tau)\right\Vert _{L^{2}}^{2}),
\end{align*}
where we have taken advantage of the estimate $\left\Vert \varphi
(\tau)\right\Vert _{H^{1}}\leq C$ for all $\tau\in\lbrack0,T]$, so that
\[
\left\Vert G\ast\mu\nabla\varphi\right\Vert _{L^{2}}^{2}\leq C\int_{0}%
^{t}(1+\left\Vert \nabla\Delta\varphi(\tau)\right\Vert _{L^{2}})(1+\left\Vert
\Delta\varphi(\tau)\right\Vert _{L^{2}}^{2})d\tau.
\]
Therefore
\begin{align}
3\int_{\Omega}\left\vert \boldsymbol{u}\right\vert ^{2}\left\vert
\nabla\varphi\right\vert ^{2}dx  & \leq C\left\Vert \boldsymbol{g}\right\Vert
_{L^{2}}^{2}(1+\left\Vert \nabla\Delta\varphi(t)\right\Vert _{L^{2}%
})\label{e6.1}\\
& +C\int_{0}^{t}(1+\left\Vert \nabla\Delta\varphi(t)\right\Vert _{L^{2}%
})(1+\left\Vert \nabla\Delta\varphi(\tau)\right\Vert _{L^{2}})(1+\left\Vert
\Delta\varphi(\tau)\right\Vert _{L^{2}}^{2})d\tau.\nonumber
\end{align}

As for $\int_{\Omega}\left\vert \Delta f(\varphi)\right\vert ^{2}dx$, we have
$\Delta f(\varphi)=f^{\prime}(\varphi)\Delta\varphi+f^{\prime\prime}%
(\varphi)\left\vert \nabla\varphi\right\vert ^{2}$, and so, using
(\ref{1.7}),
\begin{align*}
\left\Vert \Delta f(\varphi)\right\Vert _{L^{2}}  & \leq\left\Vert f^{\prime
}(\varphi)\Delta\varphi\right\Vert _{L^{2}}+\left\Vert f^{\prime\prime
}(\varphi)\left\vert \nabla\varphi\right\vert ^{2}\right\Vert _{L^{2}}\\
& \leq C(1+\left\Vert \varphi\right\Vert _{L^{\infty}}^{2})\left\Vert
\Delta\varphi\right\Vert _{L^{2}}+C(1+\left\Vert \varphi\right\Vert
_{L^{\infty}})\left\Vert \nabla\varphi\right\Vert _{L^{4}}^{2}\\
& \leq C(1+\left\Vert \Delta\varphi\right\Vert _{L^{2}})\left\Vert
\Delta\varphi\right\Vert _{L^{2}}+C(1+\left\Vert \Delta\varphi\right\Vert
_{L^{2}}^{\frac{1}{2}})\left\Vert \nabla\varphi\right\Vert _{H^{1}}^{2}\\
& \leq C(1+\left\Vert \Delta\varphi\right\Vert _{L^{2}})\left\Vert
\Delta\varphi\right\Vert _{L^{2}}+C(1+\left\Vert \Delta\varphi\right\Vert
_{L^{2}}^{\frac{1}{2}})\left\Vert \varphi\right\Vert _{H^{2}}^{2}\\
& \leq C(1+\left\Vert \Delta\varphi\right\Vert _{L^{2}})\left\Vert
\Delta\varphi\right\Vert _{L^{2}}+C(1+\left\Vert \Delta\varphi\right\Vert
_{L^{2}}^{\frac{1}{2}})(1+\left\Vert \Delta\varphi\right\Vert _{L^{2}}^{2}).
\end{align*}
Thus,
\begin{align}
\left\Vert \Delta f(\varphi)\right\Vert _{L^{2}}^{2}  & \leq C(1+\left\Vert
\Delta\varphi\right\Vert _{L^{2}}^{2})(1+\left\Vert \Delta\varphi\right\Vert
_{L^{2}}+\left\Vert \Delta\varphi\right\Vert _{L^{2}}^{2}+\left\Vert
\Delta\varphi\right\Vert _{L^{2}}^{3})\label{e5.10}\\
& \leq C(1+\left\Vert \Delta\varphi\right\Vert _{L^{2}}^{2})(1+\left\Vert
\Delta\varphi\right\Vert _{L^{2}}^{4}),\nonumber
\end{align}
that is, using the fact that $\left\Vert \Delta\varphi\right\Vert _{L^{2}}%
^{2}\leq\left\Vert \nabla\varphi\right\Vert _{L^{2}}\left\Vert \nabla
\Delta\varphi\right\Vert _{L^{2}}$ (recall that $\nabla\varphi\cdot
\boldsymbol{n}=0$ on $\partial\Omega$),
\[
\left\Vert \Delta f(\varphi)\right\Vert _{L^{2}}^{2}\leq C(1+\left\Vert
\nabla\Delta\varphi\right\Vert _{L^{2}}^{2})(1+\left\Vert \Delta
\varphi\right\Vert _{L^{2}}^{2}).
\]
It follows immediately that
\begin{align*}
& \frac{1}{2}\frac{d}{dt}\left\Vert \Delta\varphi\right\Vert _{L^{2}}%
^{2}+\frac{3}{4}\left\Vert \Delta^{2}\varphi\right\Vert _{L^{2}}^{2}\\
& \leq C(1+\left\Vert \nabla\Delta\varphi\right\Vert _{L^{2}}^{2}%
)(1+\left\Vert \Delta\varphi\right\Vert _{L^{2}}^{2})+C\left\Vert
\boldsymbol{g}\right\Vert _{L^{2}}^{2}(1+\left\Vert \nabla\Delta
\varphi\right\Vert _{L^{2}})\\
& +C\int_{0}^{t}(1+\left\Vert \nabla\Delta\varphi(t)\right\Vert _{L^{2}%
})(1+\left\Vert \nabla\Delta\varphi(\tau)\right\Vert _{L^{2}})(1+\left\Vert
\Delta\varphi(\tau)\right\Vert _{L^{2}}^{2})d\tau,
\end{align*}
or, integrating over $(0,t)$,
\begin{align}
& \left\Vert \Delta\varphi(t)\right\Vert _{L^{2}}^{2}+\int_{0}^{t}\left\Vert
\Delta^{2}\varphi(s)\right\Vert _{L^{2}}^{2}ds\label{*10}\\
& \leq\left\Vert \Delta\varphi^{0}\right\Vert _{L^{2}}^{2}+C\int_{0}%
^{t}\left\Vert \boldsymbol{g}(s)\right\Vert _{L^{2}}^{2}(1+\left\Vert
\nabla\Delta\varphi(s)\right\Vert _{L^{2}})ds\nonumber\\
& +C\int_{0}^{t}(1+\left\Vert \nabla\Delta\varphi(s)\right\Vert _{L^{2}}%
^{2})(1+\left\Vert \Delta\varphi(s)\right\Vert _{L^{2}}^{2})ds\nonumber\\
& +C\int_{0}^{t}\left(  \int_{0}^{s}(1+\left\Vert \nabla\Delta\varphi
(s)\right\Vert _{L^{2}})(1+\left\Vert \nabla\Delta\varphi(\tau)\right\Vert
_{L^{2}})(1+\left\Vert \Delta\varphi(\tau)\right\Vert _{L^{2}}^{2}%
)d\tau\right)  ds.\nonumber
\end{align}
Set
\begin{align*}
x(t)  & =1+\left\Vert \Delta\varphi(t)\right\Vert _{L^{2}}^{2},\\
a_{0}  & =1+\left\Vert \Delta\varphi^{0}\right\Vert _{L^{2}}^{2}+C\int_{0}%
^{T}\left\Vert \boldsymbol{g}(s)\right\Vert _{L^{2}}^{2}(1+\left\Vert
\nabla\Delta\varphi(s)\right\Vert _{L^{2}})ds,\\
a_{1}(t)  & =1+\left\Vert \nabla\Delta\varphi(t)\right\Vert _{L^{2}}^{2},\\
a_{2}(t,s)  & =C(1+\left\Vert \nabla\Delta\varphi(t)\right\Vert _{L^{2}%
})(1+\left\Vert \nabla\Delta\varphi(s)\right\Vert _{L^{2}}).
\end{align*}
Then (\ref{*10}) yields
\[
x(t)\leq a_{0}+\int_{0}^{t}\left(  a_{1}(s)x(s)+\int_{0}^{s}a_{2}%
(s,\tau)x(\tau)d\tau\right)  ds,\ \ t\in\lbrack0,T].
\]
Since $a_{0}<\infty$, and the functions $a_{1}$ and $a_{2}$ are integrable on
$[0,T]$ and $[0,T]^{2}$ respectively, Lemma \ref{l5.1} entails
\[
x(t)\leq a_{0}\exp\left[  \int_{0}^{T}\left(  a_{1}(s)+\int_{0}^{T}%
a_{2}(s,\tau)d\tau\right)  ds\right]  \leq C\text{, all }t\in\lbrack
0,T]\text{.}%
\]
We infer that $\varphi\in L^{\infty}(0,T;H^{2}(\Omega))$, and from
(\ref{*10}), that
\begin{equation}
\int_{0}^{T}\left\Vert \Delta^{2}\varphi(s)\right\Vert _{L^{2}}^{2}ds\leq
C,\ \ \ \ \ \ \ \ \ \ \ \ \ \ \ \ \ \ \ \ \ \ \ \ \ \ \ \ \ \ \ \ \ \ \ \ \label{e5.11}%
\end{equation}
so that $\varphi\in L^{2}(0,T;H^{4}(\Omega))$.

Next, we have
\begin{align*}
\int_{0}^{T}\left\Vert \mu(t)\right\Vert _{H^{2}}^{2}dt  & \leq C\int_{0}%
^{T}\left(  \left\Vert \Delta\mu(t)\right\Vert _{L^{2}}^{2}+\left\Vert
\mu(t)\right\Vert _{L^{2}}^{2}\right)  dt\\
& \leq C\int_{0}^{T}\left(  \left\Vert \Delta^{2}\varphi(t)\right\Vert
_{L^{2}}^{2}+\left\Vert \Delta f(\varphi(t))\right\Vert _{L^{2}}%
^{2}+\left\Vert \mu(t)\right\Vert _{L^{2}}^{2}\right)  dt.
\end{align*}
From (\ref{e5.10}), we have
\begin{align*}
\left\Vert \Delta f(\varphi)(t)\right\Vert _{L^{2}}^{2}  & \leq C(1+\left\Vert
\Delta\varphi(t)\right\Vert _{L^{2}}^{2})(1+\left\Vert \Delta\varphi
(t)\right\Vert _{L^{2}}+\left\Vert \Delta\varphi(t)\right\Vert _{L^{2}}%
^{2}+\left\Vert \Delta\varphi(t)\right\Vert _{L^{2}}^{3})\\
& \leq C\text{ for a.e. }t\in\lbrack0,T],
\end{align*}
where $C$ in the last inequality above is independent of $t$, so that,
appealing to (\ref{e5.11}), we conclude that
\[
\int_{0}^{T}\left\Vert \mu(t)\right\Vert _{H^{2}}^{2}dt\leq
C.\ \ \ \ \ \ \ \ \ \ \ \ \ \ \ \ \ \ \ \ \ \ \ \ \ \ \ \ \ \ \ \ \ \ \ \ \ \ \ \ \ \ \ \ \ \ \ \ \ \ \ \ \ \ \ \ \ \ \ \ \
\]
Finally, concerning $\partial\varphi/\partial t$, we have
\[
\frac{\partial\varphi}{\partial t}=-\boldsymbol{u}\cdot\nabla\varphi+\Delta
\mu\text{ in }Q.
\]
Thus,
\begin{align*}
\left\Vert \frac{\partial\varphi}{\partial t}\right\Vert _{L^{2}}  &
\leq\left\Vert \boldsymbol{u}\cdot\nabla\varphi\right\Vert _{L^{2}}+\left\Vert
\Delta\mu\right\Vert _{L^{2}} \leq\left\Vert \boldsymbol{u}\right\Vert _{L^{2}}\left\Vert \nabla
\varphi\right\Vert _{L^{2}}+\left\Vert \Delta\mu\right\Vert _{L^{2}}\\
& \leq C\left\Vert \boldsymbol{u}\right\Vert _{L^{2}}+\left\Vert \Delta
\mu\right\Vert _{L^{2}},
\end{align*}
where $C$ in the last inequality above is independent of $t$. Hence
\[
\int_{0}^{T}\left\Vert \frac{\partial\varphi}{\partial t}(t)\right\Vert
_{L^{2}}^{2}dt\leq C\int_{0}^{T}\left(  \left\Vert \boldsymbol{u}%
(t)\right\Vert _{L^{2}}^{2}+\left\Vert \Delta\mu(t)\right\Vert _{L^{2}}%
^{2}\right)  dt\leq C.
\]
We have shown that $\varphi$ belongs to $L^{\infty}(0,T;H^{2}(\Omega))\cap
L^{2}(0,T;H^{4}(\Omega))$ and is such that $\partial\varphi/\partial t\in
L^{2}(0,T;L^{2}(\Omega))$. This yields $\varphi\in\mathcal{C}([0,T];H^{2}%
(\Omega))$. Finally the fact that $\mu\in\mathcal{C}([0,T];L^{2}(\Omega))$ is
an easy consequence of the definition of $\mu$ together with the properties
$\varphi\in\mathcal{C}([0,T];H^{2}(\Omega))$ and $f(\varphi)\in\mathcal{C}%
([0,T];L^{2}(\Omega))$. This completes the proof.
\end{proof}

We are now able to prove the uniqueness of the solution to (\ref{5.1}).

\begin{theorem}
\label{t5.1}Let $(\boldsymbol{u},\varphi,\mu,p)$ be a solution of
\emph{(\ref{5.1})}. If further $\varphi^{0}\in H_{N}^{2}(\Omega)$, then
$(\boldsymbol{u},\varphi,\mu,p)$ is the unique solution of Problem
\emph{(\ref{5.1})}.
\end{theorem}

\begin{proof}
The existence of the solution is obtained through the homogenization process,
and some of its properties are obtained in Lemma \ref{l5.3} and in Proposition
\ref{p5.1}. Our aim here is to check the uniqueness of the solution of
(\ref{5.1}). Let $(\boldsymbol{u}_{1},\varphi_{1},\mu_{1},p_{1})$ and
$(\boldsymbol{u}_{2},\varphi_{2},\mu_{2},p_{2})$ be two solutions of
(\ref{5.1}) on the same interval $(0,T)$ having the same initial condition. We
set $\boldsymbol{u}=\boldsymbol{u}_{1}-\boldsymbol{u}_{2}$, $\varphi
=\varphi_{1}-\varphi_{2}$, $\mu=\mu_{1}-\mu_{2}$ and $p=p_{1}-p_{2}$. Then the
quadruple $(\boldsymbol{u},\varphi,\mu,p)$ satisfies
\begin{equation}
\left\{
\begin{array}
[c]{l}%
\boldsymbol{u}=G\ast(\mu\nabla\varphi_{1}+\mu_{2}\nabla\varphi-\nabla p)\\
\\
\operatorname{div}\boldsymbol{u}=0\\
\\
\dfrac{\partial\varphi}{\partial t}+\boldsymbol{u}\nabla\varphi_{1}%
+\boldsymbol{u}_{2}\nabla\varphi-\Delta\mu=0\\
\\
\mu=-\Delta\varphi+f(\varphi_{1})-f(\varphi_{2})\\
\\
\dfrac{\partial\varphi}{\partial\boldsymbol{n}}=\dfrac{\partial\mu}%
{\partial\boldsymbol{n}}=\boldsymbol{u}\cdot\boldsymbol{n}=0\\
\\
\varphi(0)=\varphi_{1}(0)-\varphi_{2}(0)=0.
\end{array}
\right.  \ \ \ \ \ \ \ \ \ \ \ \ \ \ \ \ \ \ \ \ \ \ \ \ \ \ \ \ \label{5.19}%
\end{equation}
We consider the variational form of (\ref{5.19}) and we get, for a.e.
$t\in(0,T)$,
\begin{equation}
\left\langle \frac{\partial\varphi}{\partial t},\psi\right\rangle +(\nabla
\mu,\nabla\psi)=(\boldsymbol{u}\varphi_{1},\nabla\psi)+(\boldsymbol{u}%
_{2}\varphi,\nabla\psi)\text{\ }\forall\psi\in H^{1}(\Omega)\text{ with }%
\frac{\partial\psi}{\partial\boldsymbol{n}}=0\text{ on }\partial
\Omega,\label{5.21}%
\end{equation}%
\begin{equation}
(\mu,\phi)-(\nabla\varphi,\nabla\phi)-(f(\varphi_{1})-f(\varphi_{2}%
),\phi)=0\ \ \forall\phi\in H^{1}(\Omega),\label{5.22}%
\end{equation}%
\begin{equation}
(\boldsymbol{u},v)=(G\ast\mu\nabla\varphi_{1},v)+(G\ast\mu_{2}\nabla
\varphi,v)\ \forall v\in\mathbb{H},\label{5.22'}%
\end{equation}
where, to get (\ref{5.22'}), we used the equality $(G\ast\nabla p,v)=0$ since
$(G\ast\nabla p,v)=-(G\ast p,\operatorname{div}v)=0$ as $\operatorname{div}%
v=0$. Choosing $\psi=1$ in (\ref{5.21}) we readily get $\left\langle
\varphi(t)\right\rangle =\varphi(0)=0$ $\forall t\in\lbrack0,T]$, where
$\left\langle \varphi(t)\right\rangle
=\mathchoice {{\setbox0=\hbox{$\displaystyle{\textstyle
-}{\int}$ } \vcenter{\hbox{$\textstyle -$
}}\kern-.6\wd0}}{{\setbox0=\hbox{$\textstyle{\scriptstyle -}{\int}$ } \vcenter{\hbox{$\scriptstyle -$
}}\kern-.6\wd0}}{{\setbox0=\hbox{$\scriptstyle{\scriptscriptstyle -}{\int}$
} \vcenter{\hbox{$\scriptscriptstyle -$
}}\kern-.6\wd0}}{{\setbox0=\hbox{$\scriptscriptstyle{\scriptscriptstyle
-}{\int}$ } \vcenter{\hbox{$\scriptscriptstyle -$ }}\kern-.6\wd0}}\!\int
_{\Omega}\varphi(t,x)dx$. Therefore, owing to the Poincar\'{e}-Wirtinger
inequality, $\left\Vert \varphi(t)\right\Vert _{H^{1}}\sim\left\Vert
\nabla\varphi(t)\right\Vert _{L^{2}}$. With this in mind, we choose the test
functions $\psi=\varphi$ in (\ref{5.21}) and $\phi=\mu$ in (\ref{5.22}), next
adding the resulting equalities, we obtain
\begin{equation}
\frac{1}{2}\frac{d}{dt}\left\Vert \varphi(t)\right\Vert _{L^{2}}%
^{2}+\left\Vert \mu\right\Vert _{L^{2}}^{2}=-(\boldsymbol{u}\nabla\varphi
_{1},\varphi)+(f(\varphi_{1})-f(\varphi_{2}),\mu)=0.\label{5.23}%
\end{equation}
We recall that to obtain (\ref{5.23}), we used the obvious equalities
$(\boldsymbol{u}_{2},\nabla(\varphi^{2}))=-\left\langle \operatorname{div}%
\boldsymbol{u}_{2},\varphi^{2}\right\rangle =0$ and $(\boldsymbol{u}%
,\nabla(\varphi\varphi_{1}))=-\left\langle \operatorname{div}\boldsymbol{u}%
,\varphi\varphi_{1}\right\rangle =0$. Now, we use (\ref{e5.0}) (in Remark
\ref{r5.1}) to get
\begin{align*}
\frac{d}{dt}\left\Vert \varphi\right\Vert _{L^{2}}^{2}+2\left\Vert
\mu\right\Vert _{L^{2}}^{2}  & \leq2\left\Vert \boldsymbol{u}\right\Vert
_{L^{2}}\left\Vert \nabla\varphi_{1}\right\Vert _{L^{2}}\left\Vert
\varphi\right\Vert _{L^{\infty}}+2\left\Vert f(\varphi_{1})-f(\varphi
_{2})\right\Vert _{L^{2}}\left\Vert \mu\right\Vert _{L^{2}}\\
& \leq\frac{1}{4}\left\Vert \boldsymbol{u}\right\Vert _{L^{2}}^{2}+C\left\Vert
\nabla\varphi_{1}\right\Vert _{L^{2}}^{2}\left\Vert \Delta\varphi\right\Vert
_{L^{2}}\left\Vert \varphi\right\Vert _{L^{2}}\\
& +C(1+\left\Vert \varphi_{1}\right\Vert _{L^{\infty}}^{2}+\left\Vert
\varphi_{2}\right\Vert _{L^{\infty}}^{2})\left\Vert \varphi\right\Vert
_{L^{2}}\left\Vert \mu\right\Vert _{L^{2}}\\
& \leq\frac{1}{4}\left\Vert \boldsymbol{u}\right\Vert _{L^{2}}^{2}+\frac
{1}{16}\left\Vert \Delta\varphi\right\Vert _{L^{2}}^{2}+C\left\Vert
\nabla\varphi_{1}\right\Vert _{L^{2}}^{4}\left\Vert \varphi\right\Vert
_{L^{2}}^{2}+\left\Vert \mu\right\Vert _{L^{2}}^{2}\\
& +C(1+\left\Vert \varphi_{1}\right\Vert _{L^{\infty}}^{4}+\left\Vert
\varphi_{2}\right\Vert _{L^{\infty}}^{4})\left\Vert \varphi\right\Vert
_{L^{2}}^{2}.
\end{align*}
Thus,
\begin{align}
\frac{d}{dt}\left\Vert \varphi\right\Vert _{L^{2}}^{2}+\left\Vert
\mu\right\Vert _{L^{2}}^{2}  & \leq\frac{1}{4}\left\Vert \boldsymbol{u}%
\right\Vert _{L^{2}}^{2}+\frac{1}{16}\left\Vert \Delta\varphi\right\Vert
_{L^{2}}^{2}+\label{5.24}\\
& +C\left(  1+\left\Vert \varphi_{1}\right\Vert _{L^{\infty}}^{4}+\left\Vert
\varphi_{2}\right\Vert _{L^{\infty}}^{4}+\left\Vert \nabla\varphi
_{1}\right\Vert _{L^{2}}^{4}\right)  \left\Vert \varphi\right\Vert _{L^{2}%
}^{2}.\nonumber
\end{align}
We consider once again (\ref{5.22}) and take there $\phi=\Delta\varphi$; then
\[
\left\Vert \Delta\varphi\right\Vert _{L^{2}}^{2}=-\left(  \mu,\Delta
\varphi\right)  +\left(  f(\varphi_{1})-f(\varphi_{2}),\Delta\varphi\right)  .
\]
The use of the Young inequality in the last equality above gives
\[
\left\Vert \Delta\varphi\right\Vert _{L^{2}}^{2}\leq\frac{1}{4}\left\Vert
\Delta\varphi\right\Vert _{L^{2}}^{2}+\left\Vert \mu\right\Vert _{L^{2}}%
^{2}+\frac{1}{4}\left\Vert \Delta\varphi\right\Vert _{L^{2}}^{2}+C\left(
1+\left\Vert \varphi_{1}\right\Vert _{L^{\infty}}^{4}+\left\Vert \varphi
_{2}\right\Vert _{L^{\infty}}^{4}\right)  \left\Vert \varphi\right\Vert
_{L^{2}}^{2},
\]
that is,
\begin{equation}
\left\Vert \Delta\varphi\right\Vert _{L^{2}}^{2}\leq2\left\Vert \mu\right\Vert
_{L^{2}}^{2}+C\left(  1+\left\Vert \varphi_{1}\right\Vert _{L^{\infty}}%
^{4}+\left\Vert \varphi_{2}\right\Vert _{L^{\infty}}^{4}\right)  \left\Vert
\varphi\right\Vert _{L^{2}}^{2}.\label{5.25}%
\end{equation}
Now, in (\ref{5.22'}) we take $v=\boldsymbol{u}$; then
\begin{equation}
\left\Vert \boldsymbol{u}\right\Vert _{L^{2}}^{2}=(G\ast\mu\nabla\varphi
_{1},\boldsymbol{u})+(G\ast\mu_{2}\nabla\varphi,\boldsymbol{u}),\label{5.26'}%
\end{equation}
and next, we take $\psi=\mu$ in (\ref{5.21}), and since $\mu(t)\in
H^{1}(\Omega)$ for a.e. $t\in\lbrack0,T]$, we consider the well defined
expression $-\left\langle \frac{\partial\varphi}{\partial t},\mu\right\rangle
$, and we obtain in (\ref{5.22})
\begin{equation}
\left\langle \frac{\partial\varphi}{\partial t},\mu\right\rangle +\left\Vert
\nabla\mu\right\Vert _{L^{2}}^{2}=(\varphi_{1}\boldsymbol{u}+\varphi
\boldsymbol{u}_{2},\nabla\mu),\label{5.27}%
\end{equation}
and
\begin{equation}
-\left\langle \frac{\partial\varphi}{\partial t},\mu\right\rangle +\frac{1}%
{2}\frac{d}{dt}\left\Vert \nabla\varphi\right\Vert _{L^{2}}^{2}+\left\langle
\frac{\partial\varphi}{\partial t},f(\varphi_{1})-f(\varphi_{2})\right\rangle
=0.\label{5.28}%
\end{equation}
We add (\ref{5.26'}), (\ref{5.27}) and (\ref{5.28}), and we obtain
\begin{equation}%
\begin{array}
[c]{l}%
\left\Vert \boldsymbol{u}\right\Vert _{L^{2}}^{2}+\left\Vert \nabla
\mu\right\Vert _{L^{2}}^{2}+\dfrac{1}{2}\dfrac{d}{dt}\left\Vert \nabla
\varphi\right\Vert _{L^{2}}^{2}+\left\langle \dfrac{\partial\varphi}{\partial
t},f(\varphi_{1})-f(\varphi_{2})\right\rangle \\
\ \ =(G\ast\mu\nabla\varphi_{1},\boldsymbol{u})+(G\ast\mu_{2}\nabla
\varphi,\boldsymbol{u})-(\mu\nabla\varphi_{1},\boldsymbol{u})-(\boldsymbol{u}%
_{2}\nabla\varphi,\mu)\\
\ \ \ =I_{1}+I_{2}+I_{3}+I_{4}.
\end{array}
\label{5.29}%
\end{equation}
Let us bound from above each $I_{i}$. Starting from $I_{1}$, one has
\begin{equation}
\left\vert I_{1}\right\vert \leq\frac{1}{4}\left\Vert \boldsymbol{u}%
\right\Vert _{L^{2}}^{2}+C\int_{0}^{t}\left\Vert \nabla\varphi_{1}%
(\tau)\right\Vert _{L^{4}}^{2}\left\Vert \mu(\tau)\right\Vert _{H^{1}}%
^{2}d\tau,\label{5.30'}%
\end{equation}
where to get (\ref{5.30'}), we have used the Sobolev embedding $H^{1}%
\hookrightarrow L^{4}$. Concerning $I_{2}$, we use the Gagliardo-Nirenberg
inequality (see (i) in Lemma \ref{l5.2}) associated to the continuous
embedding $H^{1}\hookrightarrow L^{4}$ to get%
\begin{align}
\left\vert I_{2}\right\vert  & \leq\int_{0}^{t}\int_{\Omega}\left\vert
G(t-\tau)\right\vert \left\vert \mu_{2}(\tau)\right\vert \left\vert
\nabla\varphi(\tau)\right\vert \left\vert \boldsymbol{u}(t)\right\vert
d\tau\label{5.31'}\\
& \leq C\int_{0}^{t}\left\Vert \mu_{2}(\tau)\right\Vert _{L^{4}}\left\Vert
\nabla\varphi(\tau)\right\Vert _{L^{4}}\left\Vert \boldsymbol{u}(t)\right\Vert
_{L^{2}}d\tau\nonumber\\
& \leq C\int_{0}^{t}\left\Vert \boldsymbol{u}(t)\right\Vert _{L^{2}}\left\Vert
\mu_{2}(\tau)\right\Vert _{H^{1}}\left(  \left\Vert \Delta\varphi
(\tau)\right\Vert _{L^{2}}^{\frac{1}{2}}\left\Vert \nabla\varphi
(\tau)\right\Vert _{L^{2}}^{\frac{1}{2}}+\left\Vert \nabla\varphi
(\tau)\right\Vert _{L^{2}}\right)  d\tau\nonumber\\
& \leq\frac{1}{4}\left\Vert \boldsymbol{u}(t)\right\Vert _{L^{2}}^{2}+\int
_{0}^{t}\left[  \frac{1}{16}\left\Vert \Delta\varphi(\tau)\right\Vert _{L^{2}%
}^{2}+C\left(  \left\Vert \mu_{2}(\tau)\right\Vert _{H^{1}}^{4}+\left\Vert
\mu_{2}(\tau)\right\Vert _{H^{1}}^{2}\right)  \left\Vert \nabla\varphi
(\tau)\right\Vert _{L^{2}}^{2}\right]  d\tau.\nonumber
\end{align}
As for $I_{3}$, noting that $\boldsymbol{u}=G\ast(\mu\nabla\varphi_{1}+\mu
_{2}\nabla\varphi-\nabla p)$, we have
\begin{align*}
\left\vert I_{3}\right\vert  & \leq\int_{\Omega}\left\vert \mu\right\vert
\left\vert \nabla\varphi_{1}\right\vert \left\vert \boldsymbol{u}\right\vert
dx\\
& \leq\int_{\Omega}\left\vert \mu\right\vert \left\vert \nabla\varphi
_{1}\right\vert \left\vert G\ast\mu\nabla\varphi_{1}\right\vert dx+\int
_{\Omega}\left\vert \mu\right\vert \left\vert \nabla\varphi_{1}\right\vert
\left\vert G\ast\mu_{2}\nabla\varphi\right\vert dx+\int_{\Omega}\left\vert
\mu\right\vert \left\vert \nabla\varphi_{1}\right\vert \left\vert G\ast\nabla
p\right\vert dx\\
& \leq\left\Vert \mu\right\Vert _{L^{4}}\left\Vert \nabla\varphi
_{1}\right\Vert _{L^{4}}\left(  \left\Vert G\ast\mu\nabla\varphi
_{1}\right\Vert _{L^{2}}+\left\Vert G\ast\mu_{2}\nabla\varphi\right\Vert
_{L^{2}}+\left\Vert G\ast\nabla p\right\Vert _{L^{2}}\right)  .
\end{align*}
But, as $\boldsymbol{u}$ is defined by (\ref{5.19}) it holds that $q=G\ast p$
solves the Neumann problem
\[
\left\{
\begin{array}
[c]{c}%
-\Delta q+\operatorname{div}(G\ast(\mu\nabla\varphi_{1}+\mu_{2}\nabla
\varphi))=0\text{ in }\Omega,\\
\nabla q\cdot\boldsymbol{n}=(G\ast(\mu\nabla\varphi_{1}+\mu_{2}\nabla
\varphi))\cdot\boldsymbol{n}\text{ on }\partial\Omega,
\end{array}
\right.
\]
so that
\begin{equation}
\left\Vert \nabla q\right\Vert _{L^{2}}\leq\left\Vert G\ast\mu\nabla
\varphi_{1}\right\Vert _{L^{2}}+\left\Vert G\ast\mu_{2}\nabla\varphi
\right\Vert _{L^{2}}.\label{5.20}%
\end{equation}
We deduce from (\ref{5.20}) that
\begin{align*}
\left\vert I_{3}\right\vert  & \leq2\left\Vert \mu\right\Vert _{L^{4}%
}\left\Vert \nabla\varphi_{1}\right\Vert _{L^{4}}\left(  \left\Vert G\ast
\mu\nabla\varphi_{1}\right\Vert _{L^{2}}+\left\Vert G\ast\mu_{2}\nabla
\varphi\right\Vert _{L^{2}}\right) \\
& \leq C\left\Vert \mu\right\Vert _{H^{1}}\left\Vert \nabla\varphi
_{1}\right\Vert _{H^{1}}\left(  \left\Vert G\ast\mu\nabla\varphi
_{1}\right\Vert _{L^{2}}+\left\Vert G\ast\mu_{2}\nabla\varphi\right\Vert
_{L^{2}}\right) \\
& \leq\frac{1}{4}\left\Vert \mu\right\Vert _{H^{1}}^{2}+C\left\Vert
\varphi_{1}\right\Vert _{H^{2}}^{2}\left(  \left\Vert G\ast\mu\nabla
\varphi_{1}\right\Vert _{L^{2}}^{2}+\left\Vert G\ast\mu_{2}\nabla
\varphi\right\Vert _{L^{2}}^{2}\right)  .
\end{align*}
But
\begin{align*}
\left\Vert G\ast\mu\nabla\varphi_{1}\right\Vert _{L^{2}}^{2}  & =\int_{\Omega
}\left\vert \int_{0}^{t}G(t-\tau)\mu(\tau)\nabla\varphi_{1}(\tau
)d\tau\right\vert ^{2}dx\\
& \leq C\int_{0}^{t}\left(  \int_{\Omega}\left\vert \mu(\tau)\right\vert
^{2}\left\vert \nabla\varphi_{1}(\tau)\right\vert ^{2}dx\right)  d\tau\\
& \leq C\int_{0}^{t}\left(  \int_{\Omega}\left\vert \mu(\tau)\right\vert
^{4}dx\right)  ^{\frac{1}{2}}\left(  \int_{\Omega}\left\vert \nabla\varphi
_{1}(\tau)\right\vert ^{4}dx\right)  ^{\frac{1}{2}}d\tau\\
& \leq C\int_{0}^{t}\left\Vert \mu(\tau)\right\Vert _{L^{4}}^{2}\left\Vert
\nabla\varphi_{1}(\tau)\right\Vert _{L^{4}}^{2}d\tau\\
& \leq C\int_{0}^{t}\left\Vert \mu(\tau)\right\Vert _{H^{1}}^{2}\left\Vert
\varphi_{1}(\tau)\right\Vert _{H^{2}}^{2}d\tau.
\end{align*}
Also it holds that
\[
\left\Vert G\ast\mu_{2}\nabla\varphi\right\Vert _{L^{2}}^{2}\leq C\int_{0}%
^{t}\left\Vert \mu_{2}(\tau)\right\Vert _{L^{2}}^{2}\left\Vert \nabla
\varphi(\tau)\right\Vert _{L^{2}}^{2}d\tau.
\]
We are therefore led to
\[
\left\vert I_{3}\right\vert \leq\frac{1}{4}\left\Vert \mu\right\Vert _{H^{1}%
}^{2}+C\left\Vert \varphi_{1}\right\Vert _{H^{2}}^{2}\left(  \int_{0}%
^{t}\left\Vert \mu(\tau)\right\Vert _{H^{1}}^{2}\left\Vert \varphi_{1}%
(\tau)\right\Vert _{H^{2}}^{2}d\tau+\int_{0}^{t}\left\Vert \mu_{2}%
(\tau)\right\Vert _{L^{2}}^{2}\left\Vert \nabla\varphi(\tau)\right\Vert
_{L^{2}}^{2}d\tau\right)  .
\]
Finally, dealing with $I_{4}$, one has
\begin{align*}
\left\vert I_{4}\right\vert  & \leq\left\Vert \boldsymbol{u}_{2}\right\Vert
_{L^{2}}\left\Vert \nabla\varphi\right\Vert _{L^{4}}\left\Vert \mu\right\Vert
_{L^{4}}\\
& \leq C\left\Vert \boldsymbol{u}_{2}\right\Vert _{L^{2}}\left(  \left\Vert
\nabla\varphi\right\Vert _{L^{2}}^{\frac{1}{2}}\left\Vert \Delta
\varphi\right\Vert _{L^{2}}^{\frac{1}{2}}+\left\Vert \nabla\varphi\right\Vert
_{L^{2}}\right)  \left\Vert \mu\right\Vert _{H^{1}}\\
& \leq C\left\Vert \boldsymbol{u}_{2}\right\Vert _{L^{2}}\left\Vert
\nabla\varphi\right\Vert _{L^{2}}^{\frac{1}{2}}\left\Vert \Delta
\varphi\right\Vert _{L^{2}}^{\frac{1}{2}}\left\Vert \mu\right\Vert _{H^{1}%
}+\left\Vert \boldsymbol{u}_{2}\right\Vert _{L^{2}}\left\Vert \nabla
\varphi\right\Vert _{L^{2}}\left\Vert \mu\right\Vert _{H^{1}}\\
& \leq\frac{1}{16}\left\Vert \Delta\varphi\right\Vert _{L^{2}}^{2}+\frac{1}%
{4}\left\Vert \mu\right\Vert _{H^{1}}^{2}+C\left(  \left\Vert \boldsymbol{u}%
_{2}\right\Vert _{L^{2}}^{4}+\left\Vert \boldsymbol{u}_{2}\right\Vert _{L^{2}%
}^{2}\right)  \left\Vert \nabla\varphi\right\Vert _{L^{2}}^{2}.
\end{align*}
Putting together the inequalities for $I_{1}$ to $I_{4}$, we are led to
\begin{align*}
& \left\Vert \boldsymbol{u}\right\Vert _{L^{2}}^{2}+\left\Vert \nabla
\mu\right\Vert _{L^{2}}^{2}+\dfrac{1}{2}\dfrac{d}{dt}\left\Vert \nabla
\varphi\right\Vert _{L^{2}}^{2}+\left\langle \dfrac{\partial\varphi}{\partial
t},f(\varphi_{1})-f(\varphi_{2})\right\rangle \\
& \leq\frac{1}{4}\left\Vert \boldsymbol{u}\right\Vert _{L^{2}}^{2}+C\int
_{0}^{t}\left\Vert \varphi_{1}(\tau)\right\Vert _{H^{2}}^{2}\left\Vert
\mu(\tau)\right\Vert _{H^{1}}^{2}d\tau+\frac{1}{4}\left\Vert \boldsymbol{u}%
\right\Vert _{L^{2}}^{2}+\frac{1}{4}\left\Vert \mu\right\Vert _{H^{1}}^{2}\\
& +\int_{0}^{t}\left[  \frac{1}{16}\left\Vert \Delta\varphi(\tau)\right\Vert
_{L^{2}}^{2}+C\left(  \left\Vert \mu_{2}(\tau)\right\Vert _{H^{1}}%
^{4}+\left\Vert \mu_{2}(\tau)\right\Vert _{H^{1}}^{2}\right)  \left\Vert
\nabla\varphi(\tau)\right\Vert _{L^{2}}^{2}\right]  d\tau\\
& +C\left\Vert \varphi_{1}\right\Vert _{H^{2}}^{2}\left(  \int_{0}%
^{t}\left\Vert \mu(\tau)\right\Vert _{H^{1}}^{2}\left\Vert \varphi_{1}%
(\tau)\right\Vert _{H^{2}}^{2}d\tau+\int_{0}^{t}\left\Vert \mu_{2}%
(\tau)\right\Vert _{L^{2}}^{2}\left\Vert \nabla\varphi(\tau)\right\Vert
_{L^{2}}^{2}d\tau\right) \\
& +\frac{1}{16}\left\Vert \Delta\varphi\right\Vert _{L^{2}}^{2}+\frac{1}%
{4}\left\Vert \mu\right\Vert _{H^{1}}^{2}+C\left(  \left\Vert \boldsymbol{u}%
_{2}\right\Vert _{L^{2}}^{4}+\left\Vert \boldsymbol{u}_{2}\right\Vert _{L^{2}%
}^{2}\right)  \left\Vert \nabla\varphi\right\Vert _{L^{2}}^{2},
\end{align*}
that is,
\begin{align}
& \frac{1}{2}\left\Vert \boldsymbol{u}\right\Vert _{L^{2}}^{2}+\left\Vert
\nabla\mu\right\Vert _{L^{2}}^{2}+\dfrac{1}{2}\dfrac{d}{dt}\left\Vert
\nabla\varphi\right\Vert _{L^{2}}^{2}\label{5.30''}\\
& \leq\frac{1}{16}\left\Vert \Delta\varphi\right\Vert _{L^{2}}^{2}+\frac{1}%
{2}\left\Vert \mu\right\Vert _{H^{1}}^{2}+C\int_{0}^{t}\left\Vert \varphi
_{1}(\tau)\right\Vert _{H^{2}}^{2}\left\Vert \mu(\tau)\right\Vert _{H^{1}}%
^{2}d\tau\nonumber\\
& +\int_{0}^{t}\left[  \frac{1}{16}\left\Vert \Delta\varphi(\tau)\right\Vert
_{L^{2}}^{2}+C\left(  \left\Vert \mu_{2}(\tau)\right\Vert _{H^{1}}%
^{4}+\left\Vert \mu_{2}(\tau)\right\Vert _{H^{1}}^{2}\right)  \left\Vert
\nabla\varphi(\tau)\right\Vert _{L^{2}}^{2}\right]  d\tau\nonumber\\
& C\left\Vert \varphi_{1}\right\Vert _{H^{2}}^{2}\left(  \int_{0}%
^{t}\left\Vert \mu(\tau)\right\Vert _{H^{1}}^{2}\left\Vert \varphi_{1}%
(\tau)\right\Vert _{H^{2}}^{2}d\tau+\int_{0}^{t}\left\Vert \mu_{2}%
(\tau)\right\Vert _{L^{2}}^{2}\left\Vert \nabla\varphi(\tau)\right\Vert
_{L^{2}}^{2}d\tau\right) \nonumber\\
& +C\left(  \left\Vert \boldsymbol{u}_{2}\right\Vert _{L^{2}}^{4}+\left\Vert
\boldsymbol{u}_{2}\right\Vert _{L^{2}}^{2}\right)  \left\Vert \nabla
\varphi\right\Vert _{L^{2}}^{2}-\left\langle \dfrac{\partial\varphi}{\partial
t},f(\varphi_{1})-f(\varphi_{2})\right\rangle .\nonumber
\end{align}
To bound the last term on the right-hand side of (\ref{5.30''}), we appeal to
(\ref{5.19})$_{3}$ and get
\begin{align*}
-\left\langle \dfrac{\partial\varphi}{\partial t},f(\varphi_{1})-f(\varphi
_{2})\right\rangle  & =\left(  \boldsymbol{u}\cdot\nabla\varphi_{1}%
,f(\varphi_{1})-f(\varphi_{2})\right)  +\left(  \boldsymbol{u}_{2}\cdot
\nabla\varphi,f(\varphi_{1})-f(\varphi_{2})\right) \\
& +\left(  \nabla\mu,\nabla\lbrack f(\varphi_{1})-f(\varphi_{2})]\right) \\
& =J_{1}+J_{2}+J_{3}.
\end{align*}
Firstly, we have
\begin{align*}
\left\vert J_{3}\right\vert  & \leq\left\Vert \nabla\mu\right\Vert _{L^{2}%
}\left\Vert \nabla\lbrack f(\varphi_{1})-f(\varphi_{2})]\right\Vert _{L^{2}}\\
& \leq\frac{1}{8}\left\Vert \nabla\mu\right\Vert _{L^{2}}^{2}+C\left\Vert
\nabla\lbrack f(\varphi_{1})-f(\varphi_{2})]\right\Vert _{L^{2}}^{2},
\end{align*}
and
\begin{align*}
\left\Vert \nabla\lbrack f(\varphi_{1})-f(\varphi_{2})]\right\Vert _{L^{2}%
}^{2}  & \leq2\left\Vert \left(  f^{\prime}(\varphi_{1})-f^{\prime}%
(\varphi_{2})\right)  \nabla\varphi_{1}\right\Vert _{L^{2}}^{2}+2\left\Vert
f^{\prime}(\varphi_{2})\nabla\varphi\right\Vert _{L^{2}}^{2}\\
& \leq C(1+\left\Vert \varphi_{1}\right\Vert _{L^{\infty}}^{2}+\left\Vert
\varphi_{2}\right\Vert _{L^{\infty}}^{2})\left\Vert \nabla\varphi
_{1}\right\Vert _{L^{\infty}}^{2}\left\Vert \varphi\right\Vert _{L^{2}}^{2}\\
& +C(1+\left\Vert \varphi_{2}\right\Vert _{L^{\infty}}^{4})\left\Vert
\nabla\varphi\right\Vert _{L^{2}}^{2}.
\end{align*}
Thus,
\begin{align}
\left\vert J_{3}\right\vert  & \leq\frac{1}{8}\left\Vert \nabla\mu\right\Vert
_{L^{2}}^{2}+C(1+\left\Vert \varphi_{1}\right\Vert _{L^{\infty}}%
^{2}+\left\Vert \varphi_{2}\right\Vert _{L^{\infty}}^{2})\left\Vert
\nabla\varphi_{1}\right\Vert _{L^{\infty}}^{2}\left\Vert \varphi\right\Vert
_{L^{2}}^{2}\label{5.32}\\
& +C(1+\left\Vert \varphi_{2}\right\Vert _{L^{\infty}}^{4})\left\Vert
\nabla\varphi\right\Vert _{L^{2}}^{2}.\nonumber
\end{align}
Secondly, we have
\begin{align*}
\left\vert J_{1}\right\vert  & \leq\left\Vert \boldsymbol{u}\right\Vert
_{L^{2}}\left\Vert \nabla\varphi_{1}\right\Vert _{L^{4}}\left\Vert
f(\varphi_{1})-f(\varphi_{2})\right\Vert _{L^{4}}\\
& \leq C(1+\left\Vert \varphi_{1}\right\Vert _{L^{\infty}}^{2}+\left\Vert
\varphi_{2}\right\Vert _{L^{\infty}}^{2})\left\Vert \boldsymbol{u}\right\Vert
_{L^{2}}\left\Vert \varphi_{1}\right\Vert _{H^{2}}\left\Vert \varphi
\right\Vert _{L^{4}}\\
& \leq C(1+\left\Vert \varphi_{1}\right\Vert _{L^{\infty}}^{2}+\left\Vert
\varphi_{2}\right\Vert _{L^{\infty}}^{2})\left\Vert \boldsymbol{u}\right\Vert
_{L^{2}}\left\Vert \varphi_{1}\right\Vert _{H^{2}}(\left\Vert \varphi
\right\Vert _{L^{2}}^{1/2}\left\Vert \nabla\varphi\right\Vert _{L^{2}}%
^{1/2}+\left\Vert \varphi\right\Vert _{L^{2}}).
\end{align*}
Now, using the inequality $\left\Vert \varphi\right\Vert _{L^{2}}\leq
C\left\Vert \nabla\varphi\right\Vert _{L^{2}}$ together with Young's
inequality, it follows that
\begin{equation}
\left\vert J_{1}\right\vert \leq\frac{1}{8}\left\Vert \boldsymbol{u}%
\right\Vert _{L^{2}}^{2}+C(1+\left\Vert \varphi_{1}\right\Vert _{L^{\infty}%
}^{4}+\left\Vert \varphi_{2}\right\Vert _{L^{\infty}}^{4})\left\Vert
\varphi_{1}\right\Vert _{H^{2}}^{2}\left\Vert \nabla\varphi\right\Vert
_{L^{2}}^{2}.\label{5.33}%
\end{equation}
Thirdly, it holds that
\begin{align*}
\left\vert J_{2}\right\vert  & \leq\left\Vert \boldsymbol{u}_{2}\right\Vert
_{L^{2}}\left\Vert \nabla\varphi\right\Vert _{L^{4}}\left\Vert f(\varphi
_{1})-f(\varphi_{2})\right\Vert _{L^{4}}\\
& \leq C(1+\left\Vert \varphi_{1}\right\Vert _{L^{\infty}}^{2}+\left\Vert
\varphi_{2}\right\Vert _{L^{\infty}}^{2})\left\Vert \boldsymbol{u}%
_{2}\right\Vert _{L^{2}}\left\Vert \nabla\varphi\right\Vert _{L^{4}}\left\Vert
\varphi\right\Vert _{L^{4}}\\
& \leq C(1+\left\Vert \varphi_{1}\right\Vert _{L^{\infty}}^{2}+\left\Vert
\varphi_{2}\right\Vert _{L^{\infty}}^{2})\left\Vert \boldsymbol{u}%
_{2}\right\Vert _{L^{2}}(\left\Vert \Delta\varphi\right\Vert _{L^{2}}%
^{1/2}\left\Vert \nabla\varphi\right\Vert _{L^{2}}^{1/2}+\left\Vert
\nabla\varphi\right\Vert _{L^{2}})\left\Vert \nabla\varphi\right\Vert _{L^{2}%
},
\end{align*}
where in the last inequality above, the Gagliardo-Nirenberg and
Poincar\'{e}-Wirtinger inequalities yield
\[
\left\Vert \varphi\right\Vert _{L^{4}}\leq C\left(  \left\Vert \varphi
\right\Vert _{L^{2}}^{1/2}\left\Vert \nabla\varphi\right\Vert _{L^{2}}%
^{1/2}+\left\Vert \varphi\right\Vert _{L^{2}}\right)  \leq C\left\Vert
\nabla\varphi\right\Vert _{L^{2}}.
\]
Thus,
\begin{align*}
\left\vert J_{2}\right\vert  & \leq\frac{1}{16}\left\Vert \Delta
\varphi\right\Vert _{L^{2}}^{2}+C(1+\left\Vert \varphi_{1}\right\Vert
_{L^{\infty}}^{2}+\left\Vert \varphi_{2}\right\Vert _{L^{\infty}}^{2}%
)^{\frac{4}{3}}\left\Vert \boldsymbol{u}_{2}\right\Vert _{L^{2}}^{\frac{4}{3}%
}\left\Vert \nabla\varphi\right\Vert _{L^{2}}^{2}\\
& +C(1+\left\Vert \varphi_{1}\right\Vert _{L^{\infty}}^{2}+\left\Vert
\varphi_{2}\right\Vert _{L^{\infty}}^{2})\left\Vert \boldsymbol{u}%
_{2}\right\Vert _{L^{2}}\left\Vert \nabla\varphi\right\Vert _{L^{2}}^{2}\\
& \leq\frac{1}{16}\left\Vert \Delta\varphi\right\Vert _{L^{2}}^{2}%
+C(1+\left\Vert \varphi_{1}\right\Vert _{L^{\infty}}^{2}+\left\Vert
\varphi_{2}\right\Vert _{L^{\infty}}^{2})^{2}\left\Vert \boldsymbol{u}%
_{2}\right\Vert _{L^{2}}^{\frac{4}{3}}\left\Vert \nabla\varphi\right\Vert
_{L^{2}}^{2}\\
& +C(1+\left\Vert \varphi_{1}\right\Vert _{L^{\infty}}^{2}+\left\Vert
\varphi_{2}\right\Vert _{L^{\infty}}^{2})\left\Vert \boldsymbol{u}%
_{2}\right\Vert _{L^{2}}\left\Vert \nabla\varphi\right\Vert _{L^{2}}^{2}.
\end{align*}
It therefore follows that
\begin{equation}%
\begin{array}
[c]{l}%
\left\vert J_{2}\right\vert \leq\frac{1}{16}\left\Vert \Delta\varphi
\right\Vert _{L^{2}}^{2}\\
+C\left(  (1+\left\Vert \varphi_{1}\right\Vert _{L^{\infty}}^{2}+\left\Vert
\varphi_{2}\right\Vert _{L^{\infty}}^{2})\left\Vert \boldsymbol{u}%
_{2}\right\Vert _{L^{2}}+(1+\left\Vert \varphi_{1}\right\Vert _{L^{\infty}%
}^{4}+\left\Vert \varphi_{2}\right\Vert _{L^{\infty}}^{4})\left\Vert
\boldsymbol{u}_{2}\right\Vert _{L^{2}}^{\frac{4}{3}}\right)  \left\Vert
\nabla\varphi\right\Vert _{L^{2}}^{2}.
\end{array}
\label{5.34}%
\end{equation}
Collecting (\ref{5.32}), (\ref{5.33}) and (\ref{5.34}), we are led to
\begin{equation}%
\begin{array}
[c]{l}%
\left\vert \left\langle \dfrac{\partial\varphi}{\partial t},f(\varphi
_{1})-f(\varphi_{2})\right\rangle \right\vert \\
\leq\frac{1}{8}\left\Vert \boldsymbol{u}\right\Vert _{L^{2}}^{2}+\frac{1}%
{16}\left\Vert \Delta\varphi\right\Vert _{L^{2}}^{2}+\frac{1}{8}\left\Vert
\nabla\mu\right\Vert _{L^{2}}^{2}+C(1+\left\Vert \varphi_{1}\right\Vert
_{L^{\infty}}^{2}+\left\Vert \varphi_{2}\right\Vert _{L^{\infty}}%
^{2})\left\Vert \nabla\varphi_{1}\right\Vert _{L^{\infty}}^{2}\left\Vert
\varphi\right\Vert _{L^{2}}^{2}\\
+C[(1+\left\Vert \varphi_{1}\right\Vert _{L^{\infty}}^{2}+\left\Vert
\varphi_{2}\right\Vert _{L^{\infty}}^{2})\left\Vert \boldsymbol{u}%
_{2}\right\Vert _{L^{2}}+(1+\left\Vert \varphi_{1}\right\Vert _{L^{\infty}%
}^{4}+\left\Vert \varphi_{2}\right\Vert _{L^{\infty}}^{4})\left\Vert
\boldsymbol{u}_{2}\right\Vert _{L^{2}}^{\frac{4}{3}}\\
+1+\left\Vert \varphi_{2}\right\Vert _{L^{\infty}}^{4}+\left(  1+\left\Vert
\varphi_{1}\right\Vert _{L^{\infty}}^{4}+\left\Vert \varphi_{2}\right\Vert
_{L^{\infty}}^{4}\right)  \left\Vert \varphi_{1}\right\Vert _{H^{2}}%
^{2}]\left\Vert \nabla\varphi\right\Vert _{L^{2}}^{2}.
\end{array}
\label{5.35}%
\end{equation}
Finally (\ref{5.30''}) becomes
\begin{align*}
& \frac{1}{2}\left\Vert \boldsymbol{u}\right\Vert _{L^{2}}^{2}+\left\Vert
\nabla\mu\right\Vert _{L^{2}}^{2}+\dfrac{1}{2}\dfrac{d}{dt}\left\Vert
\nabla\varphi\right\Vert _{L^{2}}^{2}\\
& \leq\frac{1}{8}\left\Vert \boldsymbol{u}\right\Vert _{L^{2}}^{2}+\frac
{1}{16}\left\Vert \Delta\varphi\right\Vert _{L^{2}}^{2}+\frac{1}{8}\left\Vert
\nabla\mu\right\Vert _{L^{2}}^{2}+C\left\Vert \nabla\varphi_{1}\right\Vert
_{L^{\infty}}^{2}+C(1+\left\Vert \varphi_{2}\right\Vert _{L^{\infty}}%
^{4})\left\Vert \nabla\varphi\right\Vert _{L^{2}}^{2}\\
& +C\left(  (1+\left\Vert \varphi_{1}\right\Vert _{L^{\infty}}^{2}+\left\Vert
\varphi_{2}\right\Vert _{L^{\infty}}^{2})\left\Vert \boldsymbol{u}%
_{2}\right\Vert _{L^{2}}+(1+\left\Vert \varphi_{1}\right\Vert _{L^{\infty}%
}^{4}+\left\Vert \varphi_{2}\right\Vert _{L^{\infty}}^{4})\left\Vert
\boldsymbol{u}_{2}\right\Vert _{L^{2}}^{\frac{4}{3}}\right)  \left\Vert
\nabla\varphi\right\Vert _{L^{2}}^{2}\\
& \left(  1+\left\Vert \varphi_{1}\right\Vert _{L^{\infty}}^{4}+\left\Vert
\varphi_{2}\right\Vert _{L^{\infty}}^{4}\right)  \left\Vert \varphi
_{1}\right\Vert _{H^{2}}^{2}\left\Vert \nabla\varphi\right\Vert _{L^{2}}^{2}\\
& +\frac{1}{16}\left\Vert \Delta\varphi\right\Vert _{L^{2}}^{2}+\frac{1}%
{2}\left\Vert \mu\right\Vert _{H^{1}}^{2}+C\int_{0}^{t}\left\Vert \varphi
_{1}(\tau)\right\Vert _{H^{2}}^{2}\left\Vert \mu(\tau)\right\Vert _{H^{1}}%
^{2}d\tau\\
& +\int_{0}^{t}\left[  \frac{1}{16}\left\Vert \Delta\varphi(\tau)\right\Vert
_{L^{2}}^{2}+C\left(  \left\Vert \mu_{2}(\tau)\right\Vert _{H^{1}}%
^{4}+\left\Vert \mu_{2}(\tau)\right\Vert _{H^{1}}^{2}\right)  \left\Vert
\nabla\varphi(\tau)\right\Vert _{L^{2}}^{2}\right]  d\tau\\
& C\left\Vert \varphi_{1}\right\Vert _{H^{2}}^{2}\left(  \int_{0}%
^{t}\left\Vert \mu(\tau)\right\Vert _{H^{1}}^{2}\left\Vert \varphi_{1}%
(\tau)\right\Vert _{H^{2}}^{2}d\tau+\int_{0}^{t}\left\Vert \mu_{2}%
(\tau)\right\Vert _{L^{2}}^{2}\left\Vert \nabla\varphi(\tau)\right\Vert
_{L^{2}}^{2}d\tau\right) \\
& +C\left(  \left\Vert \boldsymbol{u}_{2}\right\Vert _{L^{2}}^{4}+\left\Vert
\boldsymbol{u}_{2}\right\Vert _{L^{2}}^{2}\right)  \left\Vert \nabla
\varphi\right\Vert _{L^{2}}^{2}+\frac{1}{8}\left\Vert \nabla\mu\right\Vert
_{L^{2}}^{2},
\end{align*}
or equivalently,
\begin{align}
& \frac{3}{8}\left\Vert \boldsymbol{u}\right\Vert _{L^{2}}^{2}+\frac{7}%
{8}\left\Vert \nabla\mu\right\Vert _{L^{2}}^{2}+\frac{1}{2}\frac{d}%
{dt}\left\Vert \nabla\varphi\right\Vert _{L^{2}}^{2}\label{5.36}\\
& \leq\frac{1}{8}\left\Vert \Delta\varphi\right\Vert _{L^{2}}^{2}+\dfrac{1}%
{2}\left\Vert \mu\right\Vert _{H^{1}}^{2}+C\int_{0}^{t}\left\Vert \varphi
_{1}(\tau)\right\Vert _{H^{2}}^{2}\left\Vert \mu(\tau)\right\Vert _{H^{1}}%
^{2}d\tau\nonumber\\
& +C\left[  \left\Vert \nabla\varphi_{1}\right\Vert _{L^{\infty}}%
^{2}+(1+\left\Vert \varphi_{1}\right\Vert _{L^{\infty}}^{2}+\left\Vert
\varphi_{2}\right\Vert _{L^{\infty}}^{2})\left\Vert \nabla\varphi
_{1}\right\Vert _{L^{\infty}}^{2}\right]  \left\Vert \varphi\right\Vert
_{L^{2}}^{2}\nonumber\\
& +C[1+\left\Vert \varphi_{2}\right\Vert _{L^{\infty}}^{4}+\left\Vert
\boldsymbol{u}_{2}\right\Vert _{L^{2}}^{4}+\left\Vert \boldsymbol{u}%
_{2}\right\Vert _{L^{2}}^{2}+\left(  1+\left\Vert \varphi_{1}\right\Vert
_{L^{\infty}}^{4}+\left\Vert \varphi_{2}\right\Vert _{L^{\infty}}^{4}\right)
\left\Vert \varphi_{1}\right\Vert _{H^{2}}^{2}\nonumber\\
& +(1+\left\Vert \varphi_{1}\right\Vert _{L^{\infty}}^{2}+\left\Vert
\varphi_{2}\right\Vert _{L^{\infty}}^{2})\left\Vert \boldsymbol{u}%
_{2}\right\Vert _{L^{2}}+(1+\left\Vert \varphi_{1}\right\Vert _{L^{\infty}%
}^{4}+\left\Vert \varphi_{2}\right\Vert _{L^{\infty}}^{4})\left\Vert
\boldsymbol{u}_{2}\right\Vert _{L^{2}}^{\frac{4}{3}}]\left\Vert \nabla
\varphi\right\Vert _{L^{2}}^{2}\nonumber\\
& +\int_{0}^{t}\left[  \frac{1}{16}\left\Vert \Delta\varphi(\tau)\right\Vert
_{L^{2}}^{2}+C\left(  \left\Vert \mu_{2}(\tau)\right\Vert _{H^{1}}%
^{4}+\left\Vert \mu_{2}(\tau)\right\Vert _{H^{1}}^{2}\right)  \left\Vert
\nabla\varphi(\tau)\right\Vert _{L^{2}}^{2}\right]  d\tau\nonumber\\
& +C\left\Vert \varphi_{1}\right\Vert _{H^{2}}^{2}\left(  \int_{0}%
^{t}\left\Vert \varphi_{1}(\tau)\right\Vert _{H^{2}}^{2}\left\Vert \mu
(\tau)\right\Vert _{H^{1}}^{2}d\tau+\int_{0}^{t}\left\Vert \mu_{2}%
(\tau)\right\Vert _{L^{2}}^{2}\left\Vert \nabla\varphi(\tau)\right\Vert
_{L^{2}}^{2}d\tau\right)  .\nonumber
\end{align}
Putting together (\ref{5.24}) and (\ref{5.36}) where we take into account
(\ref{5.25}), we obtain
\begin{align*}
& \frac{d}{dt}\left\Vert \varphi\right\Vert _{L^{2}}^{2}+\frac{1}{2}\frac
{d}{dt}\left\Vert \nabla\varphi\right\Vert _{L^{2}}^{2}+\left\Vert
\mu\right\Vert _{L^{2}}^{2}+\dfrac{7}{8}\left\Vert \nabla\mu\right\Vert
_{L^{2}}^{2}+\frac{3}{8}\left\Vert \boldsymbol{u}\right\Vert _{L^{2}}^{2}\\
& \leq\frac{1}{4}\left\Vert \boldsymbol{u}\right\Vert _{L^{2}}^{2}+\frac{3}%
{8}\left\Vert \mu\right\Vert _{L^{2}}^{2}+C(1+\left\Vert \varphi
_{1}\right\Vert _{L^{\infty}}^{4}+\left\Vert \varphi_{2}\right\Vert
_{L^{\infty}}^{4})\left\Vert \varphi\right\Vert _{L^{2}}^{2}\\
& +\frac{1}{2}\left\Vert \mu\right\Vert _{H^{1}}^{2}+C\int_{0}^{t}\left\Vert
\varphi_{1}(\tau)\right\Vert _{H^{2}}^{2}\left\Vert \mu(\tau)\right\Vert
_{H^{1}}^{2}d\tau\\
& +C\left[  \left\Vert \nabla\varphi_{1}\right\Vert _{L^{\infty}}%
^{2}+(1+\left\Vert \varphi_{1}\right\Vert _{L^{\infty}}^{2}+\left\Vert
\varphi_{2}\right\Vert _{L^{\infty}}^{2})\left\Vert \nabla\varphi
_{1}\right\Vert _{L^{\infty}}^{2}\right]  \left\Vert \varphi\right\Vert
_{L^{2}}^{2}\\
& +C[1+\left\Vert \varphi_{2}\right\Vert _{L^{\infty}}^{4}+\left\Vert
\boldsymbol{u}_{2}\right\Vert _{L^{2}}^{4}+\left\Vert \boldsymbol{u}%
_{2}\right\Vert _{L^{2}}^{2}+\left(  1+\left\Vert \varphi_{1}\right\Vert
_{L^{\infty}}^{4}+\left\Vert \varphi_{2}\right\Vert _{L^{\infty}}^{4}\right)
\left\Vert \varphi_{1}\right\Vert _{H^{2}}^{2}\\
& +(1+\left\Vert \varphi_{1}\right\Vert _{L^{\infty}}^{2}+\left\Vert
\varphi_{2}\right\Vert _{L^{\infty}}^{2})\left\Vert \boldsymbol{u}%
_{2}\right\Vert _{L^{2}}+(1+\left\Vert \varphi_{1}\right\Vert _{L^{\infty}%
}^{4}+\left\Vert \varphi_{2}\right\Vert _{L^{\infty}}^{4})\left\Vert
\boldsymbol{u}_{2}\right\Vert _{L^{2}}^{\frac{4}{3}}]\left\Vert \nabla
\varphi\right\Vert _{L^{2}}^{2}\\
& +\int_{0}^{t}[\frac{1}{8}\left\Vert \mu(\tau)\right\Vert _{H^{1}}%
^{2}+C\left(  \left(  1+\left\Vert \varphi_{1}(\tau)\right\Vert _{L^{\infty}%
}^{4}+\left\Vert \varphi_{2}(\tau)\right\Vert _{L^{\infty}}^{4}\right)
\left\Vert \varphi_{1}(\tau)\right\Vert _{L^{2}}^{2}\right)  +\\
& +C(\left\Vert \mu_{2}(\tau)\right\Vert _{H^{1}}^{4}+\left\Vert \mu_{2}%
(\tau)\right\Vert _{H^{1}}^{2})\left\Vert \nabla\varphi(\tau)\right\Vert
_{L^{2}}^{2}]d\tau\\
& +C\left\Vert \varphi_{1}\right\Vert _{H^{2}}^{2}\left(  \int_{0}%
^{t}\left\Vert \varphi_{1}(\tau)\right\Vert _{H^{2}}^{2}\left\Vert \mu
(\tau)\right\Vert _{H^{1}}^{2}d\tau+\int_{0}^{t}\left\Vert \mu_{2}%
(\tau)\right\Vert _{L^{2}}^{2}\left\Vert \nabla\varphi(\tau)\right\Vert
_{L^{2}}^{2}d\tau\right)  .
\end{align*}
This leads us at
\begin{align*}
& \frac{d}{dt}\left\Vert \varphi\right\Vert _{L^{2}}^{2}+\frac{1}{2}\frac
{d}{dt}\left\Vert \nabla\varphi\right\Vert _{L^{2}}^{2}+\frac{1}{8}\left\Vert
\mu\right\Vert _{L^{2}}^{2}+\dfrac{3}{8}\left\Vert \nabla\mu\right\Vert
_{L^{2}}^{2}+\frac{1}{8}\left\Vert \boldsymbol{u}\right\Vert _{L^{2}}^{2}\\
& \leq C\left[  1+\left\Vert \varphi_{1}\right\Vert _{L^{\infty}}%
^{4}+\left\Vert \varphi_{2}\right\Vert _{L^{\infty}}^{4}+\left\Vert
\nabla\varphi_{1}\right\Vert _{L^{\infty}}^{2}+(1+\left\Vert \varphi
_{1}\right\Vert _{L^{\infty}}^{2}+\left\Vert \varphi_{2}\right\Vert
_{L^{\infty}}^{2})\left\Vert \nabla\varphi_{1}\right\Vert _{L^{\infty}}%
^{2}\right]  \left\Vert \varphi\right\Vert _{L^{2}}^{2}\\
& +C[1+\left\Vert \varphi_{2}\right\Vert _{L^{\infty}}^{4}+\left\Vert
\boldsymbol{u}_{2}\right\Vert _{L^{2}}^{4}+\left\Vert \boldsymbol{u}%
_{2}\right\Vert _{L^{2}}^{2}+\left(  1+\left\Vert \varphi_{1}\right\Vert
_{L^{\infty}}^{4}+\left\Vert \varphi_{2}\right\Vert _{L^{\infty}}^{4}\right)
\left\Vert \varphi_{1}\right\Vert _{H^{2}}^{2}\\
& +(1+\left\Vert \varphi_{1}\right\Vert _{L^{\infty}}^{2}+\left\Vert
\varphi_{2}\right\Vert _{L^{\infty}}^{2})\left\Vert \boldsymbol{u}%
_{2}\right\Vert _{L^{2}}+(1+\left\Vert \varphi_{1}\right\Vert _{L^{\infty}%
}^{4}+\left\Vert \varphi_{2}\right\Vert _{L^{\infty}}^{4})\left\Vert
\boldsymbol{u}_{2}\right\Vert _{L^{2}}^{\frac{4}{3}}]\left\Vert \nabla
\varphi\right\Vert _{L^{2}}^{2}\\
& +C\int_{0}^{t}\left(  \left\Vert \varphi_{1}(\tau)\right\Vert _{H^{2}}%
^{2}+\left\Vert \varphi_{1}(t)\right\Vert _{H^{2}}^{2}\left\Vert \varphi
_{1}(\tau)\right\Vert _{H^{2}}^{2}+1\right)  \left\Vert \mu(\tau)\right\Vert
_{H^{1}}^{2}d\tau\\
& +C\int_{0}^{t}\left(  1+\left\Vert \varphi_{1}(\tau)\right\Vert _{L^{\infty
}}^{4}+\left\Vert \varphi_{2}(\tau)\right\Vert _{L^{\infty}}^{4}\right)
\left\Vert \varphi(\tau)\right\Vert _{L^{2}}^{2}d\tau\\
& +C\int_{0}^{t}\left(  \left\Vert \mu_{2}(\tau)\right\Vert _{H^{1}}%
^{4}+\left\Vert \mu_{2}(\tau)\right\Vert _{H^{1}}^{2}\right)  \left\Vert
\nabla\varphi(\tau)\right\Vert _{L^{2}}^{2}d\tau\\
& +C\int_{0}^{t}\left\Vert \varphi_{1}(t)\right\Vert _{H^{2}}^{2}\left\Vert
\mu_{2}(\tau)\right\Vert _{L^{2}}^{2}\left\Vert \nabla\varphi(\tau)\right\Vert
_{L^{2}}^{2}d\tau.
\end{align*}
Integrating the last inequality above with respect to $s$ on $(0,t)$, we get
\begin{align}
& \left\Vert \varphi\right\Vert _{L^{2}}^{2}+\left\Vert \nabla\varphi
\right\Vert _{L^{2}}^{2}+\int_{0}^{t}\left\Vert \mu(\tau)\right\Vert _{H^{1}%
}^{2}d\tau\label{5.40}\\
& \leq\int_{0}^{t}a_{1}(s)\left\Vert \varphi(s)\right\Vert _{L^{2}}^{2}%
ds+\int_{0}^{t}a_{2}(s)\left\Vert \nabla\varphi(s)\right\Vert _{L^{2}}%
^{2}ds+\int_{0}^{t}\left(  \int_{0}^{s}a_{3}(s,\tau)\left\Vert \mu
(\tau)\right\Vert _{H^{1}}^{2}d\tau\right)  ds\nonumber\\
& +\int_{0}^{t}\left(  \int_{0}^{s}a_{4}(\tau)\left\Vert \varphi
(\tau)\right\Vert _{L^{2}}^{2}d\tau\right)  ds+\int_{0}^{t}\left(  \int
_{0}^{s}a_{5}(\tau)\left\Vert \nabla\varphi(\tau)\right\Vert _{L^{2}}^{2}%
d\tau\right)  ds\nonumber\\
& +\int_{0}^{t}\left(  \int_{0}^{s}a_{6}(s,\tau)\left\Vert \nabla\varphi
(\tau)\right\Vert _{L^{2}}^{2}d\tau\right)  ds,\nonumber
\end{align}
where
\[%
\begin{array}
[c]{l}%
a_{1}(t)=C[1+1+\left\Vert \varphi_{1}(t)\right\Vert _{L^{\infty}}%
^{4}+\left\Vert \varphi_{2}(t)\right\Vert _{L^{\infty}}^{4}+\left\Vert
\nabla\varphi_{1}(t)\right\Vert _{L^{\infty}}^{2}\\
\ \ \ \ \ \ \ \ \ \ \ \ \ \ \ \ \ \ +(1+\left\Vert \varphi_{1}(t)\right\Vert
_{L^{\infty}}^{2}+\left\Vert \varphi_{2}(t)\right\Vert _{L^{\infty}}%
^{2})\left\Vert \nabla\varphi_{1}(t)\right\Vert _{L^{\infty}}^{2}],\\
\\
a_{2}(t)=C[1+\left\Vert \varphi_{2}(t)\right\Vert _{L^{\infty}}^{4}+\left\Vert
\boldsymbol{u}_{2}(t)\right\Vert _{L^{2}}^{4}+\left\Vert \boldsymbol{u}%
_{2}(t)\right\Vert _{L^{2}}^{2}\\
\ \ \ \ \ \ \ \ \ \ \ \ \ +\left(  1+\left\Vert \varphi_{1}(t)\right\Vert
_{L^{\infty}}^{4}+\left\Vert \varphi_{2}(t)\right\Vert _{L^{\infty}}%
^{4}\right)  \left\Vert \varphi_{1}(t)\right\Vert _{H^{2}}^{2}\\
\ \ \ \ \ \ \ \ \ \ \ \ \ \ +(1+\left\Vert \varphi_{1}(t)\right\Vert
_{L^{\infty}}^{2}+\left\Vert \varphi_{2}(t)\right\Vert _{L^{\infty}}%
^{2})\left\Vert \boldsymbol{u}_{2}(t)\right\Vert _{L^{2}}\\
\ \ \ \ \ \ \ \ \ \ \ \ \ \ \ +(1+\left\Vert \varphi_{1}(t)\right\Vert
_{L^{\infty}}^{4}+\left\Vert \varphi_{2}(t)\right\Vert _{L^{\infty}}%
^{4})\left\Vert \boldsymbol{u}_{2}(t)\right\Vert _{L^{2}}^{\frac{4}{3}}],\\
\\
a_{3}(t,s)=C(\left\Vert \varphi_{1}(s)\right\Vert _{H^{2}}^{2}+\left\Vert
\varphi_{1}(t)\right\Vert _{H^{2}}^{2}\left\Vert \varphi_{1}(s)\right\Vert
_{H^{2}}^{2}+1),\\
\\
a_{4}(t)=C(1+\left\Vert \varphi_{1}(t)\right\Vert _{L^{\infty}}^{4}+\left\Vert
\varphi_{2}(t)\right\Vert _{L^{\infty}}^{4}),\\
\\
a_{5}(t)=C(\left\Vert \mu_{2}(t)\right\Vert _{H^{1}}^{4}+\left\Vert \mu
_{2}(t)\right\Vert _{H^{1}}^{2}),\\
\\
a_{6}(t,s)=C\left\Vert \varphi_{1}(t)\right\Vert _{H^{2}}^{2}\left\Vert
\mu_{2}(t)\right\Vert _{L^{2}}^{2}.
\end{array}
\]
Now, let $c_{0}=\max_{0\leq s,t\leq T}a_{3}(t,s)$. Then since $\varphi_{1}%
\in\mathcal{C}([0,T],H^{2}(\Omega))$, $c_{0}$ is well defined and is a
positive constant. This being so, we set
\begin{align*}
x(t)  & =\left\Vert \varphi(t)\right\Vert _{H^{1}}^{2}+\int_{0}^{t}\left\Vert
\mu(s)\right\Vert _{H^{1}}^{2}ds,\\
A_{1}(t)  & =a_{1}(t)+a_{2}(t)+c_{0},\\
A_{2}(t,s)  & =a_{4}(s)+a_{5}(s)+a_{6}(t,s).
\end{align*}
Then (\ref{5.40}) yields
\[
x(t)\leq\int_{0}^{t}\left(  A_{1}(s)x(s)+\int_{0}^{s}A_{2}(s,\tau)x(\tau
)d\tau\right)  ds.
\]
The functions $A_{1}$ and $A_{2}$ are integrable on $[0,T]$ and on
$[0,T]\times\lbrack0,T]$, respectively. Applying the Gronwall-type inequality
of Lemma \ref{l5.1}, we readily get $x(t)=0$ for all $t\in\lbrack0,T]$, that
is, $\varphi=0$ and $\mu=0$. This also yields $\boldsymbol{u}=0$. Coming back
to (\ref{5.19})$_{1}$, we see that $G\ast\nabla p=0$, or, applying the Laplace
transform, $\widehat{G}(\tau)\nabla\widehat{p}(\tau,x)=0$ $\forall\tau
\in\mathbb{C}$ with $\operatorname{Re}\tau>0$. Since $\widehat{G}(\tau)$ is
positive definite, $\nabla\widehat{p}(\tau,x)=0$ $\forall\tau\in\mathbb{C}$
with $\operatorname{Re}\tau>0$, that is, $\widehat{p}(\tau,\cdot)$ is a
constant depending on $\tau$. Because $\widehat{p}(\tau,\cdot)\in L_{0}%
^{2}(\Omega)$, this leads to $\widehat{p}(\tau,\cdot)=0$ for such $\tau$, or
equivalently, $p=0$.
\end{proof}
We are now able to prove the first main result of the work.
\subsubsection{\textbf{Proof of Theorem \ref{t1.1}}}

Given any ordinary sequence $E$ of positive real numbers converging to zero,
we have derived the existence of a subsequence $E^{\prime}$ from $E$ and of a
quadruple $(\boldsymbol{u}_{0},\varphi_{0},\mu_{0},p_{0})$ with
$\boldsymbol{u}_{0}\in L^{2}(Q;\mathcal{B}_{A}^{1,2}(\mathbb{R}^{d-1}%
;H_{0}^{1}(I))^{d})$, $\varphi_{0}\in L^{\infty}(0,T;H^{1}(\Omega))$, $\mu
_{0}\in L^{2}(0,T;H^{1}(\Omega))$ and $p_{0}\in L^{2}(0,T;L_{0}^{2}(\Omega))$
such that, as $E^{\prime}\ni\varepsilon\rightarrow0$,

\begin{equation*}
\begin{array}
[c]{l}%
\boldsymbol{u}_{\varepsilon}\rightarrow\boldsymbol{u}_{0}\text{ in }
L^{2}(Q_{\varepsilon})^{d}\text{-weak }\Sigma_{A}\\
\\
\varepsilon
\nabla\boldsymbol{u}_{\varepsilon}\rightarrow\overline{\nabla}_{y}%
\boldsymbol{u}_{0}\text{ in }L^{2}(Q_{\varepsilon})^{d\times d}\text{-weak
}\Sigma_{A}\\
\\
\varphi_{\varepsilon}\rightarrow\varphi_{0}\text{ in }L^{2}(Q_{\varepsilon
})\text{-strong }\Sigma_{A},\\
\\
\mu_{\varepsilon}\rightarrow\mu_{0}\text{ in }L^{2}(Q_{\varepsilon
})\text{-weak }\Sigma_{A},\\
\\
p_{\varepsilon}\rightarrow p_{0}\text{ in }L^{2}(Q_{\varepsilon})\text{-weak
}\Sigma_{A}.
\end{array}
\end{equation*}
Next, setting $\boldsymbol{u}(t,\overline{x})=\frac{1}{2}\int_{I}%
M(\boldsymbol{u}_{0}(t,\overline{x},\cdot,\zeta))d\zeta=(\overline
{\boldsymbol{u}}(t,\overline{x}),u_{d}(t,\overline{x}))$, we have shown that
$u_{d}=0$ and that the quadruple $(\overline{\boldsymbol{u}},\varphi_{0}%
,\mu_{0},p_{0})$ solves the system (\ref{1.8}). Furthermore we have that
$\overline{\boldsymbol{u}}\in\mathcal{C}([0,T];\mathbb{H})$, $\varphi_{0}%
\in\mathcal{C}([0,T];H^{1}(\Omega))\cap L^{2}(0,T;H^{3}(\Omega))$ and
$p_{0}\in L^{2}(0,T;H^{1}(\Omega)\cap L_{0}^{2}(\Omega))$. Next, assuming that
$\varphi^{0}\in H_{N}^{2}(\Omega)$, we get that $\varphi_{0}\in\mathcal{C}%
([0,T];H^{2}(\Omega))\cap L^{2}(0,T;H^{4}(\Omega))$, $\mu\in\mathcal{C}%
([0,T];H^{1}(\Omega))\subset L^{4}(0,T;H^{1}(\Omega))$ where the fact that
$\mu\in L^{4}(0,T;H^{1}(\Omega))$ has been used in the proof of Theorem
\ref{t5.1} in order to obtain the uniqueness of the solution of (\ref{5.1}).
Therefore, the convergence of the whole sequence stems from the uniqueness of
the solution to (\ref{1.8}) in that case. This completes the proof of Theorem
\ref{t1.1}.

\subsection{Proof of Theorem \ref{t1.2}}

The existence of $(\boldsymbol{u}_{0},\varphi_{0},\mu_{0},p_{0})$\ is obtained
as at the beginning of the proof of Theorem \ref{t1.1}. So we focus on system
(\ref{4.33}) which reads in the special case $d-1=1$ as follows:
\begin{equation}
\left\{
\begin{array}
[c]{l}%
\overline{\boldsymbol{u}}=G\ast(\boldsymbol{h}_{1}+\mu_{0}\dfrac
{\partial\varphi_{0}}{\partial x_{1}}-\dfrac{\partial p_{0}}{\partial x_{1}%
})\text{ in }(0,T)\times(a,b)=Q,\\
\\
\dfrac{\partial\overline{\boldsymbol{u}}}{\partial x_{1}}=0\text{ in }Q\text{
and }\overline{\boldsymbol{u}}(t,a)=\overline{\boldsymbol{u}}(t,b)=0\text{ in
}(0,T),\\
\\
\dfrac{\partial\varphi_{0}}{\partial t}+\overline{\boldsymbol{u}}\cdot
\dfrac{\partial\varphi_{0}}{\partial x_{1}}-\dfrac{\partial^{2}\mu_{0}%
}{\partial x_{1}^{2}}=0\text{ in }Q,\\
\\
\mu_{0}=-\beta\dfrac{\partial^{2}\varphi_{0}}{\partial x_{1}^{2}}+\lambda
f(\varphi_{0})\text{ in }Q,\\
\\
\varphi_{0}^{\prime}(t,a)=\varphi_{0}^{\prime}(t,b)=0,\ \mu_{0}^{\prime
}(t,a)=\mu_{0}^{\prime}(t,b)=0\text{ in }(0,T),\\
\\
\varphi_{0}(0)=\varphi^{0}\text{ in }(a,b).
\end{array}
\right. \label{5.30}%
\end{equation}
We note that we have assumed $\boldsymbol{u}^{0}=0$. From the equality
$\dfrac{\partial\overline{\boldsymbol{u}}}{\partial x_{1}}=0$ in $Q$, we
deduce that $\overline{\boldsymbol{u}}(t,x_{1})=\overline{\boldsymbol{u}}(t)$
for all $t\in(0,T)$. Now, since $\overline{\boldsymbol{u}}(t,a)=0$ in $(0,T)
$, we infer $\overline{\boldsymbol{u}}=0$ in $Q$. Therefore the first equation
in (\ref{5.30}) becomes
\begin{equation}
G\ast(\boldsymbol{h}_{1}+\mu_{0}\dfrac{\partial\varphi_{0}}{\partial x_{1}%
}-\dfrac{\partial p_{0}}{\partial x_{1}})=0\text{ \ in }Q,\label{5.31}%
\end{equation}
and the third one becomes $\dfrac{\partial\varphi_{0}}{\partial t}%
-\dfrac{\partial^{2}\mu_{0}}{\partial x_{1}^{2}}=0$ in $Q$. The last four
equations in (\ref{5.30}) amounts in the end to the Cahn-Hilliard equation in
one spatial dimension, which is known to possess a unique solution in the
underlying spaces. Now, applying the Laplace transform to (\ref{5.31}), we get
that $\boldsymbol{h}_{1}+\mu_{0}\dfrac{\partial\varphi_{0}}{\partial x_{1}%
}-\dfrac{\partial p_{0}}{\partial x_{1}}=0$. Taking into account the fact that
$p_{0}\in L_{0}^{2}(a,b)$, we deduce that $p_{0}$ solves (\ref{1.10}). The
proof of Theorem \ref{t1.2} is complete.

\section*{Acknowledgement}
\emph{The authors are grateful to the referees for their helpful comments and
suggestions.}

\emph{J.L. Woukeng acknowledges the support of the Alexander von Humboldt
Foundation. G. Cardone is member of GNAMPA (INDAM).}

\end{document}